\documentclass[a4paper,10pt,hidelinks]{amsart}
\usepackage{graphicx} 
\usepackage[a4paper, margin=2 cm]{geometry}

\usepackage{xcolor}
\definecolor{cblue}{RGB}{0,70,140}
\definecolor{cgreen}{RGB}{100,140,0}
\definecolor{cred}{RGB}{190,10,50}

\usepackage[hidelinks,colorlinks,pagebackref]{hyperref}
\hypersetup{citecolor=cgreen,linkcolor=cblue,urlcolor=cblue}
\usepackage{amsthm}
\usepackage[foot]{amsaddr}
\usepackage{thm-restate}
\usepackage{geometry}
\usepackage{amssymb}
\usepackage{amsmath}
\usepackage{cleveref}
\usepackage{soul}
\usepackage{comment}
\usepackage{setspace}

\usepackage{caption}
\usepackage{subcaption}
\usepackage[T1]{fontenc}
\usepackage[shortlabels,inline]{enumitem}
\usepackage{listings}
\usepackage{blkarray}
\usepackage{lmodern}
\usepackage[english]{babel}
\usepackage{floatrow}
\usepackage{float}
\usepackage{tikz}
\usetikzlibrary{arrows.meta,arrows}
\usetikzlibrary{arrows,decorations.markings}
\usetikzlibrary{decorations.pathmorphing}
\usepackage{mathdots}
\usetikzlibrary{calc}
\usetikzlibrary{decorations.pathreplacing}
\usetikzlibrary{positioning,patterns}
\usetikzlibrary{arrows,shapes,positioning}
\usetikzlibrary{decorations.markings}
\def \smvx {circle[radius = .05][fill = black]}

\tikzstyle{edge}=[very thick]
\tikzstyle{diredge}=[postaction={decorate,decoration={markings,
		mark=at position .7 with {\arrow[scale = 1]{stealth};}}}]

\newcommand{\defPt}[3]{
	\def \pt {(#1, #2)}
	\coordinate [at = \pt, name = #3];
}

\newcommand{\fitellipsis}[3] 
{\draw [fill=white] let \p1=(#1), \p2=(#2), \n1={atan2(\y2-\y1,\x2-\x1)}, \n2={veclen(\y2-\y1,\x2-\x1)}
    in ($ (\p1)!0.5!(\p2) $) ellipse [ x radius=\n2/2+0.1cm, y radius=#3cm, rotate=\n1];
}

\usepackage[square,sort,comma,numbers]{natbib}

\newtheorem{theorem}{Theorem}[section]
\newtheorem{lemma}[theorem]{Lemma}
\newtheorem{corollary}[theorem]{Corollary}
\newtheorem{question}[theorem]{Question}
\newtheorem{observation}[theorem]{Observation}
\newtheorem{claim}[theorem]{Claim}
\newtheorem*{claim*}{Claim}
\newtheorem{proposition}[theorem]{Proposition}
\newtheorem{conjecture}[theorem]{Conjecture}
\newtheorem{remark}[theorem]{Remark}

\newtheorem*{theorem*}{Theorem}

\theoremstyle{definition}
\newtheorem{definition}[theorem]{Definition}
\newtheorem*{definition*}{Definition}

\newcommand{\F}{\mathbb{F}}

\newcommand{\Fn}{\F_2^n}
\newcommand{\eps}{\varepsilon}

\newcommand{\id}{\mathrm{id}}
\newcommand{\Bin}{\mathrm{Bin}}

\newcommand{\Cay}[1]{\Cayley_{\Fn}(#1)}

\DeclareMathOperator{\tr}{\mathrm{Tr}}
\DeclareMathOperator{\triv}{\operatorname{triv}}
\DeclareMathOperator{\Cayley}{Cay}

\newcommand{\floor}[1]{
    \left\lfloor #1 \right\rfloor
}

\title{On Graham's rearrangement conjecture over $\Fn$}

\author{
Benjamin Bedert\textsuperscript{1}}
\address[1]{University of Oxford, UK. Email: \href{mailto:benjamin.bedert@maths.ox.ac.uk} {\nolinkurl{benjamin.bedert@maths.ox.ac.uk}}.}
\author{ 
Matija Buci\'c\textsuperscript{2}}
\address[2]{Department of Mathematics, Princeton University, Princeton, USA. Research supported in part by NSF Award DMS-2349013. Email: \href{mailto:mb5225@princeton.edu} {\nolinkurl{mb5225@princeton.edu}}.}
\author{
Noah Kravitz\textsuperscript{3}}
\address[3]{Department of Mathematics, Princeton University, Princeton, USA. Email: \href{mailto:nkravitz@princeton.edu} {\nolinkurl{nkravitz@princeton.edu}}.}
\author{
Richard Montgomery\textsuperscript{4}}
\address[4]{University of Warwick, UK. Email: \href{mailto:Richard.Montgomery@warwick.ac.uk} {\nolinkurl{richard.montgomery@warwick.ac.uk}}. Supported by the European Research Council (ERC) under the European Union Horizon 2020 research and innovation
programme (grant agreement No. 947978).} 
\author{ 
Alp M\"uyesser\textsuperscript{5}}
\address[5]{University of Oxford, UK. Email: \href{mailto:alp.muyesser@new.ox.ac.uk} {\nolinkurl{alp.muyesser@new.ox.ac.uk}}.}

\date{}

\AtBeginDocument{%
   \def\MR#1{}
}

\begin{document}

\begin{abstract} A sequence $s_1,s_2,\ldots, s_k$ of elements of a group $G$ is called a valid ordering if the partial products $s_1, s_1 s_2, \ldots, s_1\cdots s_k$ are all distinct.
A long-standing problem in combinatorial group theory asks whether, for a given group $G$, every subset $S \subseteq G\setminus \{\mathrm{id}\}$ admits a valid ordering; the instance of the additive group $\mathbb{F}_p$ is the content of a well-known 1971 conjecture of Graham.  Most partial progress to date has concerned the edge cases where either $S$ or $G \setminus S$ is quite small.
Our main result is an essentially complete resolution of the problem for $G=\mathbb{F}_2^n$: we show that there is an absolute constant $C>0$ such that every subset $S\subseteq \mathbb{F}_2^n \setminus \{0\}$ of size at least $C$ admits a valid ordering. Our proof combines techniques from additive and probabilistic combinatorics, including the Freiman--Ruzsa theorem and the absorption method.

Along the way, we also solve the general problem for moderately large subsets: there is a constant $c>0$ such that for every group $G$ (not necessarily abelian), every subset $S \subseteq G\setminus \{\mathrm{id}\}$ of size at least $|G|^{1-c}$ admits a valid ordering.  Previous work in this direction concerned only sets of size at least $(1-o(1))|G|$.  A main ingredient in our proof is a structural result, similar in spirit to the Arithmetic Regularity Lemma, showing that every Cayley graph can be efficiently decomposed into  mildly quasirandom components.
\end{abstract}

\maketitle

\section{Introduction }

\subsection{The main problem}
A sequence $g_1, g_2, \ldots, g_n$ of elements of a (multiplicative) group $G$ is a \emph{valid ordering} if the partial products
$$g_1, \quad g_1g_2,  \quad g_1g_2g_3, \quad \ldots,  \quad g_1\cdots g_n$$
are all distinct. Which subsets of groups admit valid orderings? Variants of this natural problem have been studied in many different cases over the years.  

The first question in this direction appeared in 1961, when Gordon~\cite{gordon1961sequences}, motivated by constructions of complete Latin squares, asked for which finite groups the entire group has a valid ordering.  Gordon gave a complete characterization in the abelian case: A finite (additive), nontrivial abelian group \( G \) admits a valid ordering if and only if \( \sum_{g \in G} g \neq 0 \), this being the obvious necessary condition for the existence of such an ordering. In 1974, Ringel~\cite{ringeloldproblem} posed the closely related problem of characterising the groups $G$ whose elements can be ordered as $g_1,\ldots, g_n$ in such a way that $g_1=g_1g_2 \cdots g_n=\id$ but otherwise all partial products are distinct. The motivation for this question came from Ringel's solution~\cite{ringel2012map} of the Heawood map colouring conjecture. 

The nonabelian case of Gordon's problem is more subtle, since there are some small nonabelian groups (such as $S_3$) that for no apparent reason fail to have valid orderings.  In 1981, Keedwell \cite{keedwell1981sequenceable} posed the bold conjecture that every sufficiently large nonabelian group has a valid ordering.  M\"uyesser and Pokrovskiy \cite{muyesser2022random} recently proved Keedwell's conjecture as a consequence of their more general probabilistic analogue of the Hall--Paige Conjecture \cite{hall1955complete, asymptotichallpaige} concerning the existence of transversals in multiplication tables. This work also shows that large groups have an ordering, in the sense that Ringel asked for, if and only if the product of all group elements (in any order) is an element of the commutator subgroup\footnote{This condition is equivalent to existence of an ordering $g_1,\ldots, g_n$ such that $g_1g_2\cdots g_n=\id$.}. 

In this paper we will be concerned not only with the case when an entire group $G$ admits a valid ordering but with the more general question of when an \emph{arbitrary} subset $S$ of a given group $G$ admits a valid ordering.  Notice that when $S$ contains the identity element, every possible valid ordering of $S$ must start with the identity, since otherwise two consecutive partial products would be equal. 
Thus, if $G$ is abelian and $\sum_{g \in S}g=0$, then there cannot be a valid ordering of $S$. In order to avoid this obstruction, we restrict our attention to subsets $S$ not containing the identity, and the following is our central question.

\begin{question}\label{question:main}
For which groups $G$ does every subset $S \subseteq G \setminus \{\id\}$ admit a valid ordering?
\end{question}

It seems feasible that the answer to this question is affirmative for every finite group $G$.  At a first glance, finding valid orderings for smaller subsets $S$ might seem like an easier task, since there is more space to place the partial products without creating collisions.  However, the potential obstructions for small $S$ are at least as rich as for Gordon's setting $S=G \setminus \{\id\}$, since a small set $S$ may itself be a subgroup of $G$, or could be a complicated conglomeration of approximate subgroups and random-like sets. In the graph-theoretic formulation of these problems, which we will describe below, Gordon's setting corresponds to the complete graph case (in particular, a directed variant of a well-known conjecture of Andersen~\cite{andersen1989hamilton}), whereas \Cref{question:main} corresponds to a sparse analogue. Such sparse analogues in extremal graph theory tend to be harder and less well understood than their dense counterparts.

The simplest instance of Question~\ref{question:main} is when $G=\mathbb{F}_p$, for a prime $p$.  This problem was first posed by Graham~\cite{graham1971sums} in 1971 and later reiterated in an open problems book of Erd\H{o}s and Graham~\cite{ErdosGraham}.
\begin{conjecture}[Graham]\label{conj:graham} Let $p$ be prime. Then every subset of $\mathbb{F}_p\setminus \{0\}$ admits a valid ordering.
\end{conjecture}

Most previous work towards Conjecture~\ref{conj:graham} has concerned the edge cases where either $S$ or $\mathbb{F}_p \setminus S$ is very large.  The best result for small sets $S$ is due to Bedert and Kravitz \cite{BederdKravitz}, who showed that every set $S \subseteq \mathbb{F}_p \setminus \{0\}$ of size at most $e^{\log^{1/4}p}$ has a valid ordering. 
 For very large sets $S$, the aforementioned result of M\"uyesser and Pokrovskiy~\cite{muyesser2022random} establishes Conjecture~\ref{conj:graham} for all sets $S \subseteq \mathbb{F}_p \setminus \{0\}$ of size at least $(1-o(1))p$ (and indeed proves an analogous result for all finite groups; see Theorem~\ref{thm:ultra-dense}).  The intermediate regime remains open.
 
Various groups of authors (see, e.g.,~\cite{alspach2020strongly, hicks2019distinct, costa2018problem}) have considered instances of Question~\ref{question:main} other than $G=\mathbb{F}_p$.  In particular, Alspach~\cite{costa2020some} conjectured an affirmative answer to Question~\ref{question:main} for all finite abelian groups $G$, and Alspach and Liversidge \cite{alspach2020strongly} confirmed this for subsets of size up to $11$. For extensions of this problem to a nonabelian setting, see~\cite{ollis2019sequences, costa2025graham} and the dynamic survey of Ollis~\cite{ollis2002sequenceable}.

In a different direction, Buci\'c, Frederickson, M\"uyesser, Pokrovskiy, and Yepremyan~\cite{towards-graham} have recently provided an affirmative answer to an ``approximate'' relaxation of Question~\ref{question:main}.  They showed that every finite subset $S$ of any group $G$ has an ordering in which all but $o(|S|)$ partial products are distinct.

\subsection{Main results}
Despite the partial progress discussed above, there is no infinite class of groups $G$ for which we have a complete understanding of Question~\ref{question:main}.  Our main result remedies this situation for the family of groups $\mathbb{F}_2^n$.

\begin{theorem}\label{thm:mainthm} There is an absolute constant $C$ such that for all $n \in \mathbb{N}$, every set $S\subseteq \Fn\setminus\{0\}$ of size at least $C$ has a valid ordering. 
\end{theorem}

We remark that our methods allow us to obtain the same result for the class of finite abelian groups of exponent at most $K$ for any $K$. For example, with the same method, for $p$ fixed and $n\to \infty$, we can conclude that any subset of $\mathbb{F}_p^n\setminus\{0\}$ of size at least $C=C(p)$ has a valid ordering. For clarity of exposition, we describe only the $2$-torsion case in this paper.

One can view Theorem~\ref{thm:mainthm} as resolving the ``finite-field model'' version of Conjecture~\ref{conj:graham}. The study of additive combinatorial problems over finite-field models is a well-established topic in its own right; see the decennial surveys by Green~\cite{green2004finite}, Wolf~\cite{wolf2015finite}, and Peluse~\cite{peluse2023finite}. One of the key structural advantages of high-dimensional vector-spaces over finite fields is their rich subgroup structure. Perhaps more unexpectedly, another key advantage —crucial for our purposes— is that any moderately dense $S\subset \Fn$ contains an abundance of small subsets whose elements sum to $0$. This is surprising given that $0$-sum subsets are precisely what we need to avoid in valid orderings. We refer the interested reader to Section~\ref{sec:overview} for a high-level overview of our proof strategy.  

Although $\mathbb{F}_2^n$ has its advantages, the simplest setting for Question~\ref{question:main} turns out to be $\mathbb{Z}$, where a simple inductive argument produces a valid ordering of any finite subset of $\mathbb{Z} \setminus \{0\}$ (see~\cite{kravitz2024rearranging}). This fact plays a key role in the work of Bedert and Kravitz~\cite{BederdKravitz}, who resolve Conjecture~\ref{conj:graham} for subsets $S$ of quasipolynomial size by leveraging the fact that $\mathbb{F}_p$ looks locally like $\mathbb{Z}$. Unfortunately, there is no such ``lifting'' trick in the finite-field model.

Our proof of Theorem~\ref{thm:mainthm} treats the ``sparse $S$'' and ``dense $S$'' regimes separately.  Our argument for the sparse case makes use of the specific structure of $\Fn$, but our argument for the dense case applies to general (even nonabelian) groups.  In particular, we are able to provide an affirmative answer to Question~\ref{question:main} if one restricts attention to subsets $S$ of size at least $|G|^{1-c}$; this significantly improves on the result of M\"uyesser and Pokrovskiy~\cite{muyesser2022random}, which treats only subsets $S$ of size $(1-o(1))|G|$.

\begin{restatable}{theorem}{generaldense}
\label{thm:densecase-intro} 
There is an absolute constant $c>0$ such that for any finite (possibly nonabelian) group $G$, every subset $S\subseteq G\setminus\{\id\}$ of size at least $|G|^{1-c}$ admits a valid ordering.
\end{restatable}

\subsection{Connections to designs}

Let us say a few words about the relation between Question~\ref{question:main} and the theory of combinatorial designs.  Gordon was initially interested in groups with valid orderings because their multiplication tables can be used to construct complete Latin squares. A \emph{Latin square}, also called a \emph{quasigroup}, is a group without the axiom of associativity; equivalently, a Latin square is an $n$ by $n$ grid filled with the symbols $\{1,2,\ldots, n\}$ in such a way that each symbol appears exactly once in each row and in each column.  A Latin square is called \textit{complete} if for each pair of distinct symbols $(i,j)$, the symbol $j$ appears immediately after the symbol $i$ in exactly one row and in exactly one column. The additional degree of symmetry in complete Latin squares gives them practical uses in the design of experiments (see, e.g.,~\cite{BATE2008336}), and they have applications to the study of graph decompositions (see~\cite{ollis2002sequenceable}). We point an interested reader to a wonderful book \cite{latin-squares} on the topic with a plethora of further connections and applications.

\subsection{A weak nonabelian arithmetic regularity lemma}
The proofs of Theorems~\ref{thm:mainthm} and~\ref{thm:densecase-intro} use a combination of the absorption method and various tools from additive combinatorics.  We will give a more detailed overview in the following section, but for now we will highlight one key intermediate result which may be of independent interest.  Recall that for a subset $X$ of a group $G$, the \emph{right Cayley graph} of $G$ with respect to $X$, denoted $\Cayley_G(X)$, is the directed graph with vertex set $G$ where there is a directed edge from $g$ to $gx$ for each $g \in G$ and $x \in X$. The \textit{adjacency matrix} of a directed graph $\Gamma=(V,E)$ is the $|V| \times |V|$ matrix $M_\Gamma$ with rows and columns indexed by $V$, where the $(u,v)$-entry equals $1$ if $(u,v)$ is a directed edge and equals $0$ otherwise.  Note that $M_\Gamma$ is not necessarily symmetric, so it may have complex eigenvalues.  When every vertex of $\Gamma$ has out-degree $d$, the adjacency matrix $M_\Gamma$ always has $d$ as a \emph{trivial eigenvalue} (and in fact $d$ is the largest eigenvalue in absolute value).

\begin{restatable}{theorem}{generalexpander}    
\label{prop:RobustCayley-generalnonAbelian}
Let $\sigma\in(0,1]$ and $\varepsilon\in(0,1/2)$. Let $G$ be a finite (not necessarily abelian) group, and let $S\subseteq G$ be a subset with density $\sigma =|S|/|G|$. Then there is a subgroup $H$ of $G$ such that:
    \begin{enumerate}
        \item\label{itm:99percent} $|S\cap H|\geq (1-\varepsilon)|S|$;
        \item\label{spectralgap} all non-trivial eigenvalues of the adjacency matrix of $\Cayley_H(S\cap H)$ have real part at most $(1-\eta)|S\cap H|$, where $\eta:=\varepsilon\sigma^2/{1000}$. 
    \end{enumerate}    
\end{restatable}
Condition (\ref{spectralgap}) asserts that $\Cayley_H(S\cap H)$ has a positive spectral gap, which turns out to be a natural mild expansion condition for our purposes. In particular, this spectral condition allows us to lower bound the number of edges across any cut of $\Cayley_H(S\cap H)$. We say that an \emph{$\eta$-sparse cut} in a finite directed graph $\Gamma$ is a partition $X_1 \sqcup X_2$ of the vertex set of $\Gamma$ such that there are fewer than $\eta |X_1| \cdot |X_2|$ (directed) edges from $X_1$ to $X_2$. We will see below (Lemma~\ref{lem:spectralgapnosparsecutnonAbelian}) that (\ref{spectralgap}) implies the purely combinatorial condition that $\Cayley_H(S\cap H)$ has no $\eta\sigma$-sparse cut.

In a sense, Theorem~\ref{prop:RobustCayley-generalnonAbelian} is analogous to the more familiar Arithmetic Regularity Lemma (ARL) of Green \cite{green2005szemeredi} (see also \cite{green2010arithmetic}). Roughly speaking, the ARL offers a more refined decomposition where (\ref{spectralgap}) is strengthened by replacing $(1-\eta)|S|$ with $\eta|S|$. This stronger condition allows one to count occurrences of additive patterns such as $3$-term arithmetic progressions. Theorem~\ref{prop:RobustCayley-generalnonAbelian} is unfortunately unable to count such delicate ``local'' substructures, but in the context of Question~\ref{question:main} the mild quasirandomness condition (2) already provides sufficiently strong information, and we shall see that it has several further redeeming qualities.

One advantage of our weak ARL is that it can handle nonabelian groups.  Although there has been some prior interest in nonabelian analogues of the ARL (e.g., model-theoretic approaches \cite{conant2021structure} can be used to give structure theorems for sets with bounded VC-dimension), our weak ARL is the first such result that applies to arbitrary subsets $S$. We further note that the decomposition of $\Cayley_H(S\cap H)$ provided by Theorem \ref{prop:RobustCayley-generalnonAbelian} is particularly simple, in that it allows us to partition the vertex set $G$ into the cosets of a subgroup $H$ so that each of the induced graphs $\Cayley_{xH}(S\cap H)$ is isomorphic to the mildly quasirandom Cayley graph $\Cayley_H(S\cap H)$. The fact that these structured components are cosets makes the application in the context of Question~\ref{question:main} very convenient. It is known on the other hand that if one wants to find such a decomposition where each component is ``strongly quasirandom'' as in Green's ARL, then already in the abelian setting one has to work with more complicated components than subgroups, such as Bohr sets.

Another advantage of our Theorem~\ref{prop:RobustCayley-generalnonAbelian} lies in the quantitative aspect. The polynomial dependence of $\eta$ on $\sigma$ is ultimately the source of the polynomial improvement in Theorem~\ref{thm:densecase-intro}. By contrast, it is well-established that tower-type dependences are essential to the usual versions of regularity lemmas \cite{green2005szemeredi, hosseini2016improved}. Even for the weak graph regularity lemma of Frieze and Kannan \cite{frieze1999quick}, exponential dependencies are required \cite{conlon2012bounds}. Therefore, usual versions of regularity lemmas give useful information only about dense subsets, even in the simplest case of Cayley graphs over $\Fn$. In contrast, our Theorem~\ref{prop:RobustCayley-generalnonAbelian} gives information about polynomially sparse sets $S$. Note also that the index of $H$ in $G$ is $O(1/\sigma)$ by condition (1), so our decomposition of $G$ into $H$-cosets (with each $\Cayley_{xH}(S\cap H)$ mildly quasirandom) uses only $O(1/\sigma)$ pieces. 

We also mention that Theorem~\ref{prop:RobustCayley-generalnonAbelian} is closely related to a purely graph-theoretical result of K\"uhn, Lo, Osthus, and Staden \cite{kuhn2015robust} (see also \cite{gruslys-letzter, robust-expanders}) that provides a similar structural decomposition for \emph{dense} $d$-regular graphs. More precisely, these authors show that any regular graph of density $\sigma$ can be decomposed into clusters in such a way that there are very few edges between different clusters, and there are no $f(\sigma)$-sparse cuts within any single cluster; we refer the reader to~\cite{gruslys-letzter} for further details. Our Theorem~\ref{prop:RobustCayley-generalnonAbelian} is a more specialised result because it pertains only to Cayley graphs, but it has the dual advantages of giving group-theoretic information about the clusters, and of enjoying polynomial bounds (as contrasted with the exponential bounds in~\cite{kuhn2015robust}).

We anticipate that Theorem~\ref{prop:RobustCayley-generalnonAbelian} will find further applications in the study of Cayley graphs. For example, in upcoming work, Bedert, Dragani\'c, M\"uyesser, and Pavez-Sign\'e apply Theorem~\ref{prop:RobustCayley-generalnonAbelian} to the well-known conjecture of Lov\'asz asserting that every (connected) Cayley graph is Hamiltonian.

\subsection{Organization of the paper}

In \Cref{sec:overview} we give a high-level overview of our main ideas. The results in this section are only for expository purposes and are not used in the remainder of the paper. \Cref{sec:notation} contains notation and other preliminaries. We then turn to our weak nonabelian regularity lemma in \Cref{sec:regularity}, which is split into one subsection for the special case of $\Fn$ and one subsection for the case of general finite groups. 
In \Cref{sec:veryflex} we prove a very flexible asymptotic result for the dense setting under the assumption of a certain expansion condition (as guaranteed by the natural output of \Cref{sec:regularity}).
In \Cref{sec:absorb} we prove our absorption lemmas. This section is divided into \Cref{subsec:absorbing-path}, where we show how to build our absorbing structure, and \Cref{subsec:tails-lemma}, where we show how this structure lets us absorb a small set of leftover colours.
In \Cref{sec:f2n-dense} we establish our main result over $\Fn$ (\Cref{thm:mainthm}) in the dense case. In \Cref{sec:sparse-F2n} we complete the proof of \Cref{thm:mainthm} by analysing the sparse case. This section is split into \Cref{subsec:structured}, where we deal with the ``structured'' case, and \Cref{subsec:expanding}, where we deal with the ``random-like'' case. In \Cref{sec:general-dense} we prove \Cref{thm:densecase-intro}, which provides an affirmative answer to Question~\ref{question:main} for polynomial-density subsets of general groups. Finally, we make some concluding remarks in \Cref{sec:concluding}.

We remark that the arguments about $\Fn$ in Sections~\ref{sec:fncase} and~\ref{sec:f2n-dense} are not strictly speaking necessary since they are subsumed by the more general results in Sections~\ref{subsec:general_case} and~\ref{sec:general-dense}.  We include the analysis of these special cases separately because several of the arguments simplify, leading to a more direct and streamlined proof of \Cref{thm:mainthm}.  This case also provides an opportunity to build intuition for the more technical general results that follow.

\subsection*{Acknowledgements} We are grateful to Mira Tartarotti and Julia Wolf for remarks concerning arithmetic regularity lemmas, and we thank Zach Hunter for helpful comments on a draft of this paper. The first author gratefully acknowledges financial support from the EPSRC. The second and fifth authors were supported by the National Science Foundation under Grant No. DMS-1928930 during their Spring 2025 residence at the Simons Laufer Mathematical Sciences Institute in Berkeley, California.  The third author was supported in part by the NSF Graduate Research Fellowship
Program under grant DGE--203965.

\section{Overview}\label{sec:overview}
Our arguments combine several ideas from different parts of combinatorics, including \textit{inverse problems} and \textit{Fourier analysis} from additive combinatorics, \textit{absorption} from probabilistic combinatorics, and \textit{robust expansion} from extremal combinatorics.  In the interest of making our proofs accessible to a wide audience, we will first give a high-level overview of the main ideas in a simplified context. This purely expository section is not logically necessary for the rest of the paper. 

It is useful to recast the main problem in the language of finding rainbow paths in Cayley graphs. In general, a \emph{rainbow} subgraph of an edge-coloured graph is a subgraph all of whose edges have different colours; see~\cite{towards-graham, sudakov2024restricted, ringel} for more context on the rich study of rainbow subgraphs from a graph-theoretic perspective.   We can view the Cayley graph $\Cayley_G(S)$ (recall the definition from above) as an edge-coloured digraph with colour set $S$, where the directed edge from $g$ to $gx$ has the colour $x$ for each $g \in G$ and $x \in S$.

\begin{observation}
Let $S$ be a finite subset of a group $G$. Then, $S$ has a valid ordering if and only if $\Cayley_G(S)$ has a directed rainbow path with $|S|-1$ edges.
\end{observation}
\begin{proof}
    If $s_1,\ldots, s_{|S|}$ is a valid ordering of $S$, then $$s_1 \rightarrow s_1s_2 \rightarrow \cdots \rightarrow s_1s_2\cdots s_{|S|}$$ is a directed rainbow path in $\Cayley_G(S)$ with $|S|-1$ edges. Conversely, any directed rainbow path in $\Cayley_G(S)$ with $|S|-1$ edges is of the form
    $$g s_{\sigma(1)} \rightarrow g s_{\sigma(1)}s_{\sigma(2)} \rightarrow \cdots \rightarrow g s_{\sigma(1)}s_{\sigma(2)} \cdots s_{\sigma(|S|)}$$
for some permutation $\sigma$ of $[|S|]$ and $g \in G$, and then $s_{\sigma(1)}, s_{\sigma(2)}, \ldots, s_{\sigma(|S|)}$ is a valid ordering of $S$.
\end{proof}
Therefore, our goal is to find a rainbow path of length $|S|-1$ in $\Cayley_{G}(S)$. We use a ``$99\% \to 100\%$ framework'', more commonly known in the world of probabilistic combinatorics as the ``absorption method'' since its codification by R\"odl, Ruci\'nski, and Szemer\'edi \cite{RRSab} in 2008 (though its origins can be traced back farther to \cite{erdHos1991vertex}). The rough idea is that we first find a rainbow path of length $0.99|S|$ and then upgrade this partial rainbow path to a rainbow path of length $|S|-1$.\footnote{Of course, the constants $0.01$ and $0.99$ serve schematic purposes and should not be taken too literally.}
We carry out this upgrade using a certain ``absorbing structure'' that we set aside before finding the $99\%$ rainbow path.
We treat these two steps in the following two subsections.

\subsection{99\%-results.}\label{sec:99percent} In this subsection we will describe how to find a rainbow path of length $0.99|S|$ in $\Cayley_{\Fn}(S)$. Such an approximate result was already established recently in \cite[Theorem 1.5]{towards-graham}, but this result is not robust enough for our framework to be able to convert it into a 100\% result. The approach we use in the present paper for the 99\% part is significantly different and in particular more robust in several ways.
A key advantage of our new methods is that we can establish the existence of rainbow paths of length $0.99|S|$ in random subgraphs of $\Cayley_{\Fn}(S)$, and this flexibility is crucial for the second step of our $99\% \to 100\%$ framework.

\par A central idea is the \textit{dichotomy between structure and randomness} from additive combinatorics. We will decompose our given subset $S\subseteq \Fn$ into a ``structured'' part and a ``random-like'' part. We measure structure/randomness according to the \emph{doubling constant} $|S+S|/|S|$, where we have written $S+S:=\{x+y \colon x,y \in S\}$. Small doubling corresponds to structure; and its opposite is ``everywhere-expansion'', in the following sense.

\begin{definition}[Everywhere-expansion]
Let $\gamma,K>0$.  A subset $E \subseteq \Fn$ is \emph{$(\gamma,K)$-everywhere-expanding}
if every subset $E' \subseteq E$ of size $\gamma |E|$ satisfies $|E'+E'|\ge K|E'|$.
\end{definition}

To obtain our decomposition of $S$, we iteratively remove subsets of size at least $\gamma |S|$ and doubling at most $K$ as long as such subsets exist; the remainder is then guaranteed to be $(\gamma, K)$-everywhere-expanding.  The following lemma codifies the outcome of this procedure.

\begin{proposition}\label{prop:structure-randomness-intro}
Let $K \gg \alpha \gg \gamma >0$.  We can decompose any subset $S\subseteq \Fn$ as $S=S_1\cup S_2\cup \cdots \cup S_t \cup E$, where
   \begin{enumerate}
       \item $|S_i| \geq \gamma |S|$ and $|S_i+S_i|\leq K|S_i|$ for all $i$;
       \item $E$ is $(\gamma/\alpha, K/\alpha)$-everywhere-expanding set of size $\alpha |S|$.
   \end{enumerate}
\end{proposition}
Here, one should think of the $S_i$'s as the structured pieces of $S$ and of $E$ as the random-like piece.
Two extreme possible outcomes of the above lemma are $E=\emptyset$ and $E=S$. In the former case $S$ completely decomposes into structured pieces, while in the latter case all of $S$ is random-like; these two cases naturally require different treatments.  Our analysis of the general case splits into two cases depending on the size of $E$.

We start by illustrating how to solve the 99\% problem when the random-like part $E$ is all of $S$. For this we will need the following standard additive-combinatorial tool (see \cite[Lemma 2.6]{tao2006additive}).
\begin{lemma}[Ruzsa triangle inequality]\label{lem:triangle-ineq-over}
    For subsets $V,S$ of an abelian group, we have $|V+S|^2\geq |V|\cdot |S+S|$.
\end{lemma}
We can now establish a $99\%$-result for the model case of an everywhere-expanding set $S$.  This case per se does not figure in our main argument, but it serves as an excellent illustration of the ideas involved.  The strategy is that we will build a long rainbow path two vertices at a time, and at each step we will make sure that we have enough options to continue extending the path at the subsequent step.  Extending two vertices at a time instead of one vertex at a time is what allows us to make use of the everywhere-expanding hypothesis (which guarantees that sumsets of large subsets of $S$ grow).  
\begin{proposition}\label{prop:modelexpansion}
Let $0<\gamma<1/10$ and $K>0$ satisfy $K>10/\gamma^{4}$.  Suppose that $S \subseteq \Fn \setminus \{0\}$ is a $(\gamma, K)$-everywhere-expanding set of size $|S| \geq 2/\gamma$. Then, $\Cayley_{\Fn}(S)$ has a rainbow path of length $(1-2\gamma)|S|$. 
\end{proposition}
\begin{proof}
For each $t=0,1,2,\ldots, (1/2-\gamma)|S|$, we will build a rainbow path $$P_t=(v_0 \rightarrow v_1 \rightarrow \cdots \rightarrow v_{2t})$$
in $\Cayley_{\Fn}(S)$ such that $v_{2t}$ has at most $\gamma |S|$ neighbours in $P_t$, i.e., $$|(v_{2t}+S) \cap \{v_0, \ldots, v_{2t}\}| \leq \gamma |S|.$$
For $t=0$, we can take $v_0$ to be any element of $\Fn$.  Suppose that we have already constructed $P_t$ and we want to extend it to $P_{t+1}$. Since $v_{2t}$ has at most $\gamma |S|$ neighbours in $P_t$, among the $|S|-2t>2\gamma |S|$ colours not appearing in $P_t$, there is a set $S' \subseteq S$ consisting of $2\gamma |S|-\gamma |S|=\gamma |S|$ colours such that
\begin{equation}\label{eq:miss-the-path}
    (v_{2t}+S') \cap \{v_0, \ldots, v_{2t}\}=\emptyset;
\end{equation}
let $S''$ consist of some $\gamma |S|$ of the remaining colours not appearing in $P_t$.
The Ruzsa triangle inequality and the $(\gamma,K)$-everywhere-expanding hypothesis give
\begin{equation}\label{eq:many-choices}
    |v_{2t}+S'+S''| \geq \sqrt{|S'| \cdot |S''+S''|} \geq \sqrt{\gamma |S| \cdot K\gamma |S|}=\sqrt{K} \cdot \gamma |S|.
\end{equation}
We will obtain the path $P_{t+1}$ by setting $$v_{2t+1}:=v_{2t}+s', \quad v_{2t+2}:=v_{2t}+s'+s''$$ for suitable $s' \in S'$, $s'' \in S''$. Our definitions of the sets $S',S''$ guarantee that $P_{t+1}$ is a rainbow walk; we show that we can choose $s',s''$ so that this walk is in fact a path. Note that $v_{2t+1}$ is disjoint from $P_t$ by \eqref{eq:miss-the-path} for all choices of $s' \in S'$. We must check that $v_{2t+2}$ does not lie on $P_t$ and that $v_{2t+2}$ has at most $\gamma |S|$ neighbours in $P_t \cup \{v_{2t+1}\}$.

Say that a vertex $v \in \Fn$ is \emph{bad} if it either lies on $P_t$ or has at least $(\gamma/2)|S|$ neighbours in $P_t$.  Since there are at most $|S|$ vertices on $P_t$ and each is incident to $|S|$ edges, the number of bad vertices is at most $$(2t+1)+\frac{|S|\cdot |S|}{(\gamma/2) |S|}\leq |S|+(2/\gamma)|S|<\sqrt{K} \cdot \gamma|S|.$$ So by \eqref{eq:many-choices}, we can choose $s' \in S'$, $s'' \in S''$ so that $v_{2t+2}$ is not bad. It follows that $v_{2t+2}$ has at most $(\gamma/2)|S|+1 \leq \gamma |S|$ neighbours in $P_t \cup \{v_{2t+1}, v_{2t+2}\}$, as desired.
\end{proof}

This proof has a fair bit of flexibility. For example, we had plenty of viable choices, say, at least $\frac12\sqrt{K} \gamma |S|$ choices, for $v_{2t+2}$ at each step. 
Now, if $P'$ is a fixed rainbow path of length $1000$ (say) with colours not appearing in $P_t$, then we can append a translate of $P'$ to one of our viable choices for $v_{2t+2}$ in such a way that we still get a path, and that the final vertex of the resulting long rainbow path has few neighbours on the new path itself.
In other words, at the cost of using two colours from the given everywhere-expanding set, we can incorporate $1000$ \emph{arbitrary} colours into our rainbow path.  A careful implementation of this idea leads to the following proposition ensuring a $99\%$ rainbow path in $\Cayley_{\Fn}(S)$ whenever the unstructured piece of $S$ has size at least $0.01|S|$ (see Theorem~\ref{thm:expandingcase} for more details).

\begin{proposition} 
Let $S\subseteq \Fn$, and suppose that there is a $(0.001,10^{20})$-everywhere-expanding subset $E \subseteq S$ of size at least $0.01|S|$.  Then $\Cayley_{\Fn}(S)$ has a rainbow path of length $0.99|S|$.
\end{proposition}

In order to show that $\Cayley_{\Fn}(S)$ has a rainbow path of length $0.99|S|$ for all choices of $S$, it remains only to handle the case where at least $99\%$ of $S$ is structured, in the sense of Proposition~\ref{prop:structure-randomness-intro}.  To this end, suppose that at least $99\%$ of $S$ can be expressed as the union of sets $S_1, \ldots, S_t$ each with size at least $\gamma |S|$ and doubling at most $K$.  Notice that $t \leq 1/\gamma$ is bounded.  Provided that we can (somewhat flexibly) find a $99\%$ rainbow path in each $\Cayley_{\Fn}(S_i)$ individually, we will be able to concatenate translates of these paths 
using ideas similar to those sketched above (see \Cref{lem:cosets} for more details).

With this in mind, let us turn our attention to the $99\%$ problem for a single structured piece.
Our analysis of this case starts with the celebrated Freiman--Ruzsa Theorem, which provides a description of sets of small doubling.  Green and Tao \cite{green-tao} proved a strong result of this type in $\Fn$, and we will use the following slight improvement later formulated in \cite{finite-field-freiman}. 

\begin{theorem}\label{lem:freiman--Ruzsa-overview} Let $K \geq 1$.  If $S\subseteq \Fn$ satisfies $|S+S| \le K |S|$, then there is a subspace $H$ of $\Fn$ such that $S \subseteq H$ and $|H|\leq 2^{2K} |S|$.
\end{theorem}

The recently proven Polynomial Freiman--Ruzsa Conjecture over $\Fn$~\cite{gowers2025conjecture} provides the additional information that any subset $S \subseteq \Fn$ of doubling at most $K$ can be covered by $K^{O(1)}$ translates of a ``small'' subspace of $\Fn$.  Using this result in place of Theorem~\ref{lem:freiman--Ruzsa-overview} would improve the quantitative dependencies among the various parameters in our proof, but such an improvement would be inconsequential for the final result Theorem \ref{thm:mainthm}. Hence, we prefer to work with the conceptually simpler Theorem~\ref{lem:freiman--Ruzsa-overview} despite its quantitative inefficiency.

Theorem~\ref{lem:freiman--Ruzsa-overview} effectively reduces the structured case to the case of dense subsets of subspaces of $\Fn$, which, of course, are isomorphic to $\mathbb{F}_2^{m}$ for $m\leq n$. Such a reduction is useful because it gives us access to so-called ``robust expansion'' tools, as in the work of Lo, K\"uhn, Osthus, and Staden~\cite{kuhn2015robust} mentioned above, which generally apply only in the setting of \emph{dense} graphs.  We will return to this theme in Section~\ref{sec:veryflex}; in the meantime we refer the reader to~\cite[Sections 4 and 5]{towards-graham} and~\cite{kuhn2015robust, gruslys-letzter} for more context.

Once we reduce to the dense case, we can apply a result from \cite{towards-graham} (based on robust expansion tools) to obtain a $99\%$ path in each $\Cayley_{\Fn}(S_i)$.  This is not sufficient, however: For other parts of our argument (concatenating the paths for different $S_i$'s and carrying out the later absorption step), we need additional flexibility in prescribing \emph{where} within $\Cayley_{\Fn}(S_i)$ the $99\%$ path lives. It is here that \Cref{prop:RobustCayley-generalnonAbelian} comes to the rescue
by allowing us to pass from the Cayley graph of a dense set to a robust expander whose vertex set corresponds to a subgroup of $\Fn$.

We will prove Theorem~\ref{prop:RobustCayley-generalnonAbelian} in full generality in Section~\ref{sec:regularity}. The proof of Theorem~\ref{thm:mainthm} requires only the special case of Cayley graphs on $\Fn$, where the following slightly stronger result holds. 

\begin{restatable}{lemma}{fnexpander}
\label{cor:robexpander}
    Let $\eps\in(0,1/2)$ and write $N=2^n$. Let $S \subseteq \Fn$ have size $|S| \ge \sigma N$. Then, there is a subspace $H$ of $\Fn$ satisfying
            \begin{itemize}
                \item [(1)] $|S\cap H|\geq (1-\eps)|S|$;
                \item [(2)] $\Cayley_{H}(S \cap H)$ has no $\eps\sigma/2$-sparse cuts.
                \end{itemize} 
\end{restatable}

The proof of Theorem~\ref{prop:RobustCayley-generalnonAbelian} simplifies considerably in the special setting of $\Fn$, and we include a separate proof of this case in Section~\ref{sec:regularity} since it is all that is needed for the proof of Theorem~\ref{thm:mainthm}.

In particular, the reader who wishes only to see a proof of Theorem~\ref{thm:mainthm} need not bother with our nonabelian Fourier-analytic arguments for general groups.

\Cref{cor:robexpander} tells us that by sacrificing a tiny proportion of the structured set $S$, we may assume that $S$ generates a Cayley graph with good expansion properties within the subspace generated by $S$. This subspace property will later prove useful since we will be able to ``jump'' among cosets when linking up translates of various paths; see \Cref{lem:asymptoticinrandom-dense} below.

\subsection{99\% to 100\%-results}\label{99problemsbut1percentaintone} In this subsection we will discuss how to upgrade a $99\%$ result to a $100\%$ result.  The main framework has three steps:

\setlength{\fboxsep}{10pt}  
\setlength{\fboxrule}{1pt}  
\fbox{
\begin{minipage}{0.945\linewidth}
\begin{itemize}

    \item[\textbf{Step 1.}] Build a flexible ``absorbing'' structure within $\Cayley_{G}(S)$.
    \item[\textbf{Step 2.}] Run the $99\%$ strategy to obtain a rainbow path using $99\%$ of the colours in $S$. 
    \item[\textbf{Step 3.}] Use the absorbing structure to integrate the remaining $1\%$ of colours of $S$ into the rainbow path. 
\end{itemize}
\end{minipage}}

Let us break this down step by step.
\par \textbf{Step 1.} The main idea for building our flexible structure is exploiting \textit{popular sums}.  For simplicity, consider the case where the group $G$ is abelian.  Suppose $S \subseteq G$ contains elements $a,b,c$ summing to $0$, and let $d$ be some other element of $S$.  Then, for any $v \in G$ we can build a path from $v$ to $v+d$ either directly as $v \rightarrow v+d$ (using only the colour $d$) or as $$v \,\,\,\rightarrow\,\,\, v+a\,\,\, \rightarrow\,\,\, v+a+d \,\,\,\rightarrow \,\,\,v+a+d+b \,\,\,\rightarrow \,\,\,v+a+d+b+c=v+d$$ (using the colours $a,b,c,d$). See Figure~\ref{fig:absorbing-gadget}. We note that for the latter case, some mild conditions on $a,b,c,d$ are required in order for this to be an actual path rather than a walk.  
Thus, if we have a rainbow path containing an edge of colour $d$, and the above alternative route does not intersect the path elsewhere, then we may choose whether or not to add the colours $a,b,c$ in addition to $d$. 

We will see in Section~\ref{sec:absorb} that with some minor caveats (including using $6$-tuples instead of triples), we can find not only a single quadruple $(a,b,c,d)$ as above but rather many disjoint such quadruples $(a_i,b_i,c_i,d_i)$ for $1 \leq i \leq |S|/10$ (say) with $a_i+b_i+c_i=0$.  This is possible in $\Fn$ because $0$ is a ``popular sum'' for any sufficiently large subset $S\subset \Fn$.  
In Lemma~\ref{lem:absorbing-path}, we will see how to string together the gadgets from the previous paragraph to obtain a long rainbow absorbing path in which for each $i$, there is a shortcut that avoids precisely $a_i,b_i,c_i$. (When we refer to an absorbing path, we mean the path that takes the long route through each gadget.) Thus we may choose, independently for each $i$, whether or not to take the colours $a_i,b_i,c_i$ out of our rainbow path.  See Figure~\ref{fig:absorbingpath}.  The benefit of this manoeuvre is that we may later flexibly use the freed-up triples $a_i,b_i,c_i$ elsewhere, and below in Step 3 we will see how this flexibility will turn out to be very useful. 

For nonabelian groups $G$, our absorbing structure is more delicate because we cannot rely on an abundance of small subsets of $S$ with the same product.  We will instead use a variant of the so-called ``distributive absorption'' strategy, first introduced in~\cite{randomspanningtree}.  We defer further explanation to Section~\ref{sec:general-dense}.

\begin{figure}[h]
    \centering

\tikzset{every picture/.style={line width=0.75pt}} 

\RawFloats
\begin{minipage}[t]{0.35\textwidth}
\centering 
\captionsetup{width=0.9\linewidth}
\begin{tikzpicture}[x=0.75pt,y=0.75pt,yscale=-0.8,xscale=0.8]

\draw [color={rgb, 255:red, 155; green, 155; blue, 155 }  ,draw opacity=1 ][line width=1.5]    (332.01,9.86) -- (293.51,39.8) ;
\draw [color={rgb, 255:red, 74; green, 144; blue, 226 }  ,draw opacity=1 ][line width=1.5]    (332.01,9.86) -- (371.29,38.66) ;
\draw [color={rgb, 255:red, 208; green, 2; blue, 27 }  ,draw opacity=1 ][line width=1.5]    (371.29,38.66) -- (359.24,86.32) ;
\draw [color={rgb, 255:red, 155; green, 155; blue, 155 }  ,draw opacity=1 ][line width=1.5]    (306.57,86.32) -- (357.29,86.32) ;
\draw [color={rgb, 255:red, 144; green, 19; blue, 254 }  ,draw opacity=1 ][line width=1.5]    (306.57,86.32) -- (293.51,39.8) ;
\draw  [fill={rgb, 255:red, 0; green, 0; blue, 0 }  ,fill opacity=1 ] (369.34,38.66) .. controls (369.34,37.54) and (370.22,36.63) .. (371.29,36.63) .. controls (372.37,36.63) and (373.25,37.54) .. (373.25,38.66) .. controls (373.25,39.78) and (372.37,40.69) .. (371.29,40.69) .. controls (370.22,40.69) and (369.34,39.78) .. (369.34,38.66) -- cycle ;
\draw  [fill={rgb, 255:red, 0; green, 0; blue, 0 }  ,fill opacity=1 ] (330.05,9.86) .. controls (330.05,8.74) and (330.93,7.83) .. (332.01,7.83) .. controls (333.08,7.83) and (333.96,8.74) .. (333.96,9.86) .. controls (333.96,10.98) and (333.08,11.88) .. (332.01,11.88) .. controls (330.93,11.88) and (330.05,10.98) .. (330.05,9.86) -- cycle ;
\draw  [fill={rgb, 255:red, 0; green, 0; blue, 0 }  ,fill opacity=1 ] (357.29,86.32) .. controls (357.29,85.2) and (358.16,84.3) .. (359.24,84.3) .. controls (360.32,84.3) and (361.19,85.2) .. (361.19,86.32) .. controls (361.19,87.44) and (360.32,88.35) .. (359.24,88.35) .. controls (358.16,88.35) and (357.29,87.44) .. (357.29,86.32) -- cycle ;
\draw  [fill={rgb, 255:red, 0; green, 0; blue, 0 }  ,fill opacity=1 ] (304.62,86.32) .. controls (304.62,85.2) and (305.49,84.3) .. (306.57,84.3) .. controls (307.64,84.3) and (308.52,85.2) .. (308.52,86.32) .. controls (308.52,87.44) and (307.64,88.35) .. (306.57,88.35) .. controls (305.49,88.35) and (304.62,87.44) .. (304.62,86.32) -- cycle ;
\draw  [fill={rgb, 255:red, 0; green, 0; blue, 0 }  ,fill opacity=1 ] (291.56,39.8) .. controls (291.56,38.68) and (292.44,37.77) .. (293.51,37.77) .. controls (294.59,37.77) and (295.47,38.68) .. (295.47,39.8) .. controls (295.47,40.92) and (294.59,41.83) .. (293.51,41.83) .. controls (292.44,41.83) and (291.56,40.92) .. (291.56,39.8) -- cycle ;

\draw (278,54.91) node [anchor=north west][inner sep=0.75pt]  [font=\small]  {$a$};
\draw (360,7.4) node [anchor=north west][inner sep=0.75pt]  [font=\small]  {$b$};
\draw (373,56.4) node [anchor=north west][inner sep=0.75pt]  [font=\small]  {$c$};
\draw (323,92.4) node [anchor=north west][inner sep=0.75pt]  [font=\small]  {$d$};
\draw (294,8.4) node [anchor=north west][inner sep=0.75pt]  [font=\small]  {$d$};
\draw (291,91.4) node [anchor=north west][inner sep=0.75pt]  [font=\small]  {$v$};
\draw (352,92.4) node [anchor=north west][inner sep=0.75pt]  [font=\small]  {$v+d$};

\end{tikzpicture}

    \caption{
    Two paths from $v$ to $v+d$, one using only the colour $d$, and the other using the colours $a,b,c,d$.}
    \label{fig:absorbing-gadget}
\end{minipage}\hfill
\begin{minipage}[t]{0.64\textwidth}
    
    \centering

\tikzset{every picture/.style={line width=0.75pt}} 

\begin{tikzpicture}[x=0.75pt,y=0.75pt,yscale=-1,xscale=0.93]

\draw [color={rgb, 255:red, 155; green, 155; blue, 155 }  ,draw opacity=1 ][line width=1.5]    (189.9,5.14) -- (159.65,28.35) ;
\draw [color={rgb, 255:red, 74; green, 144; blue, 226 }  ,draw opacity=1 ][line width=1.5]    (189.9,5.14) -- (220.78,27.47) ;
\draw [color={rgb, 255:red, 208; green, 2; blue, 27 }  ,draw opacity=1 ][line width=1.5]    (220.78,27.47) -- (211.31,64.43) ;
\draw [color={rgb, 255:red, 155; green, 155; blue, 155 }  ,draw opacity=1 ][line width=1.5]    (169.91,64.43) -- (209.77,64.43) ;
\draw [color={rgb, 255:red, 144; green, 19; blue, 254 }  ,draw opacity=1 ][line width=1.5]    (169.91,64.43) -- (159.65,28.35) ;
\draw  [fill={rgb, 255:red, 0; green, 0; blue, 0 }  ,fill opacity=1 ] (219.25,27.47) .. controls (219.25,26.6) and (219.93,25.9) .. (220.78,25.9) .. controls (221.63,25.9) and (222.31,26.6) .. (222.31,27.47) .. controls (222.31,28.34) and (221.63,29.04) .. (220.78,29.04) .. controls (219.93,29.04) and (219.25,28.34) .. (219.25,27.47) -- cycle ;
\draw  [fill={rgb, 255:red, 0; green, 0; blue, 0 }  ,fill opacity=1 ] (188.37,5.14) .. controls (188.37,4.27) and (189.05,3.56) .. (189.9,3.56) .. controls (190.75,3.56) and (191.43,4.27) .. (191.43,5.14) .. controls (191.43,6) and (190.75,6.71) .. (189.9,6.71) .. controls (189.05,6.71) and (188.37,6) .. (188.37,5.14) -- cycle ;
\draw  [fill={rgb, 255:red, 0; green, 0; blue, 0 }  ,fill opacity=1 ] (209.77,64.43) .. controls (209.77,63.56) and (210.46,62.85) .. (211.31,62.85) .. controls (212.15,62.85) and (212.84,63.56) .. (212.84,64.43) .. controls (212.84,65.3) and (212.15,66) .. (211.31,66) .. controls (210.46,66) and (209.77,65.3) .. (209.77,64.43) -- cycle ;
\draw  [fill={rgb, 255:red, 0; green, 0; blue, 0 }  ,fill opacity=1 ] (168.37,64.43) .. controls (168.37,63.56) and (169.06,62.85) .. (169.91,62.85) .. controls (170.75,62.85) and (171.44,63.56) .. (171.44,64.43) .. controls (171.44,65.3) and (170.75,66) .. (169.91,66) .. controls (169.06,66) and (168.37,65.3) .. (168.37,64.43) -- cycle ;
\draw  [fill={rgb, 255:red, 0; green, 0; blue, 0 }  ,fill opacity=1 ] (158.11,28.35) .. controls (158.11,27.49) and (158.8,26.78) .. (159.65,26.78) .. controls (160.49,26.78) and (161.18,27.49) .. (161.18,28.35) .. controls (161.18,29.22) and (160.49,29.93) .. (159.65,29.93) .. controls (158.8,29.93) and (158.11,29.22) .. (158.11,28.35) -- cycle ;
\draw [color={rgb, 255:red, 155; green, 155; blue, 155 }  ,draw opacity=1 ][line width=1.5]    (296.9,5.14) -- (266.65,28.35) ;
\draw [color={rgb, 255:red, 245; green, 166; blue, 35 }  ,draw opacity=1 ][line width=1.5]    (296.9,5.14) -- (327.78,27.47) ;
\draw [color={rgb, 255:red, 184; green, 233; blue, 134 }  ,draw opacity=1 ][line width=1.5]    (327.78,27.47) -- (318.31,64.43) ;
\draw [color={rgb, 255:red, 155; green, 155; blue, 155 }  ,draw opacity=1 ][line width=1.5]    (276.91,64.43) -- (316.77,64.43) ;
\draw [color={rgb, 255:red, 80; green, 227; blue, 194 }  ,draw opacity=1 ][line width=1.5]    (276.91,64.43) -- (266.65,28.35) ;
\draw  [fill={rgb, 255:red, 0; green, 0; blue, 0 }  ,fill opacity=1 ] (326.25,27.47) .. controls (326.25,26.6) and (326.93,25.9) .. (327.78,25.9) .. controls (328.63,25.9) and (329.31,26.6) .. (329.31,27.47) .. controls (329.31,28.34) and (328.63,29.04) .. (327.78,29.04) .. controls (326.93,29.04) and (326.25,28.34) .. (326.25,27.47) -- cycle ;
\draw  [fill={rgb, 255:red, 0; green, 0; blue, 0 }  ,fill opacity=1 ] (295.37,5.14) .. controls (295.37,4.27) and (296.05,3.56) .. (296.9,3.56) .. controls (297.75,3.56) and (298.43,4.27) .. (298.43,5.14) .. controls (298.43,6) and (297.75,6.71) .. (296.9,6.71) .. controls (296.05,6.71) and (295.37,6) .. (295.37,5.14) -- cycle ;
\draw  [fill={rgb, 255:red, 0; green, 0; blue, 0 }  ,fill opacity=1 ] (316.77,64.43) .. controls (316.77,63.56) and (317.46,62.85) .. (318.31,62.85) .. controls (319.15,62.85) and (319.84,63.56) .. (319.84,64.43) .. controls (319.84,65.3) and (319.15,66) .. (318.31,66) .. controls (317.46,66) and (316.77,65.3) .. (316.77,64.43) -- cycle ;
\draw  [fill={rgb, 255:red, 0; green, 0; blue, 0 }  ,fill opacity=1 ] (275.37,64.43) .. controls (275.37,63.56) and (276.06,62.85) .. (276.91,62.85) .. controls (277.75,62.85) and (278.44,63.56) .. (278.44,64.43) .. controls (278.44,65.3) and (277.75,66) .. (276.91,66) .. controls (276.06,66) and (275.37,65.3) .. (275.37,64.43) -- cycle ;
\draw  [fill={rgb, 255:red, 0; green, 0; blue, 0 }  ,fill opacity=1 ] (265.11,28.35) .. controls (265.11,27.49) and (265.8,26.78) .. (266.65,26.78) .. controls (267.49,26.78) and (268.18,27.49) .. (268.18,28.35) .. controls (268.18,29.22) and (267.49,29.93) .. (266.65,29.93) .. controls (265.8,29.93) and (265.11,29.22) .. (265.11,28.35) -- cycle ;
\draw [color={rgb, 255:red, 155; green, 155; blue, 155 }  ,draw opacity=1 ][line width=1.5]    (403.9,5.14) -- (373.65,28.35) ;
\draw [color={rgb, 255:red, 117; green, 5; blue, 44 }  ,draw opacity=1 ][line width=1.5]    (403.9,5.14) -- (434.78,27.47) ;
\draw [color={rgb, 255:red, 8; green, 185; blue, 0 }  ,draw opacity=1 ][line width=1.5]    (434.78,27.47) -- (425.31,64.43) ;
\draw [color={rgb, 255:red, 155; green, 155; blue, 155 }  ,draw opacity=1 ][line width=1.5]    (383.91,64.43) -- (423.77,64.43) ;
\draw [color={rgb, 255:red, 233; green, 155; blue, 134 }  ,draw opacity=1 ][line width=1.5]    (383.91,64.43) -- (373.65,28.35) ;
\draw  [fill={rgb, 255:red, 0; green, 0; blue, 0 }  ,fill opacity=1 ] (433.25,27.47) .. controls (433.25,26.6) and (433.93,25.9) .. (434.78,25.9) .. controls (435.63,25.9) and (436.31,26.6) .. (436.31,27.47) .. controls (436.31,28.34) and (435.63,29.04) .. (434.78,29.04) .. controls (433.93,29.04) and (433.25,28.34) .. (433.25,27.47) -- cycle ;
\draw  [fill={rgb, 255:red, 0; green, 0; blue, 0 }  ,fill opacity=1 ] (402.37,5.14) .. controls (402.37,4.27) and (403.05,3.56) .. (403.9,3.56) .. controls (404.75,3.56) and (405.43,4.27) .. (405.43,5.14) .. controls (405.43,6) and (404.75,6.71) .. (403.9,6.71) .. controls (403.05,6.71) and (402.37,6) .. (402.37,5.14) -- cycle ;
\draw  [fill={rgb, 255:red, 0; green, 0; blue, 0 }  ,fill opacity=1 ] (423.77,64.43) .. controls (423.77,63.56) and (424.46,62.85) .. (425.31,62.85) .. controls (426.15,62.85) and (426.84,63.56) .. (426.84,64.43) .. controls (426.84,65.3) and (426.15,66) .. (425.31,66) .. controls (424.46,66) and (423.77,65.3) .. (423.77,64.43) -- cycle ;
\draw  [fill={rgb, 255:red, 0; green, 0; blue, 0 }  ,fill opacity=1 ] (382.37,64.43) .. controls (382.37,63.56) and (383.06,62.85) .. (383.91,62.85) .. controls (384.75,62.85) and (385.44,63.56) .. (385.44,64.43) .. controls (385.44,65.3) and (384.75,66) .. (383.91,66) .. controls (383.06,66) and (382.37,65.3) .. (382.37,64.43) -- cycle ;
\draw  [fill={rgb, 255:red, 0; green, 0; blue, 0 }  ,fill opacity=1 ] (372.11,28.35) .. controls (372.11,27.49) and (372.8,26.78) .. (373.65,26.78) .. controls (374.49,26.78) and (375.18,27.49) .. (375.18,28.35) .. controls (375.18,29.22) and (374.49,29.93) .. (373.65,29.93) .. controls (372.8,29.93) and (372.11,29.22) .. (372.11,28.35) -- cycle ;
\draw [color={rgb, 255:red, 155; green, 155; blue, 155 }  ,draw opacity=1 ][line width=1.5]    (510.9,5.14) -- (480.65,28.35) ;
\draw [color={rgb, 255:red, 148; green, 199; blue, 255 }  ,draw opacity=1 ][line width=1.5]    (510.9,5.14) -- (541.78,27.47) ;
\draw [color={rgb, 255:red, 255; green, 75; blue, 98 }  ,draw opacity=1 ][line width=1.5]    (541.78,27.47) -- (532.31,64.43) ;
\draw [color={rgb, 255:red, 155; green, 155; blue, 155 }  ,draw opacity=1 ][line width=1.5]    (490.91,64.43) -- (530.77,64.43) ;
\draw [color={rgb, 255:red, 202; green, 149; blue, 255 }  ,draw opacity=1 ][line width=1.5]    (490.91,64.43) -- (480.65,28.35) ;
\draw  [fill={rgb, 255:red, 0; green, 0; blue, 0 }  ,fill opacity=1 ] (540.25,27.47) .. controls (540.25,26.6) and (540.93,25.9) .. (541.78,25.9) .. controls (542.63,25.9) and (543.31,26.6) .. (543.31,27.47) .. controls (543.31,28.34) and (542.63,29.04) .. (541.78,29.04) .. controls (540.93,29.04) and (540.25,28.34) .. (540.25,27.47) -- cycle ;
\draw  [fill={rgb, 255:red, 0; green, 0; blue, 0 }  ,fill opacity=1 ] (509.37,5.14) .. controls (509.37,4.27) and (510.05,3.56) .. (510.9,3.56) .. controls (511.75,3.56) and (512.43,4.27) .. (512.43,5.14) .. controls (512.43,6) and (511.75,6.71) .. (510.9,6.71) .. controls (510.05,6.71) and (509.37,6) .. (509.37,5.14) -- cycle ;
\draw  [fill={rgb, 255:red, 0; green, 0; blue, 0 }  ,fill opacity=1 ] (530.77,64.43) .. controls (530.77,63.56) and (531.46,62.85) .. (532.31,62.85) .. controls (533.15,62.85) and (533.84,63.56) .. (533.84,64.43) .. controls (533.84,65.3) and (533.15,66) .. (532.31,66) .. controls (531.46,66) and (530.77,65.3) .. (530.77,64.43) -- cycle ;
\draw  [fill={rgb, 255:red, 0; green, 0; blue, 0 }  ,fill opacity=1 ] (489.37,64.43) .. controls (489.37,63.56) and (490.06,62.85) .. (490.91,62.85) .. controls (491.75,62.85) and (492.44,63.56) .. (492.44,64.43) .. controls (492.44,65.3) and (491.75,66) .. (490.91,66) .. controls (490.06,66) and (489.37,65.3) .. (489.37,64.43) -- cycle ;
\draw  [fill={rgb, 255:red, 0; green, 0; blue, 0 }  ,fill opacity=1 ] (479.11,28.35) .. controls (479.11,27.49) and (479.8,26.78) .. (480.65,26.78) .. controls (481.49,26.78) and (482.18,27.49) .. (482.18,28.35) .. controls (482.18,29.22) and (481.49,29.93) .. (480.65,29.93) .. controls (479.8,29.93) and (479.11,29.22) .. (479.11,28.35) -- cycle ;
\draw [color={rgb, 255:red, 0; green, 0; blue, 0 }  ,draw opacity=0.73 ][line width=1.5]    (211.31,64.43) -- (275.37,64.43) ;
\draw [color={rgb, 255:red, 0; green, 0; blue, 0 }  ,draw opacity=0.77 ][line width=1.5]    (319.84,64.43) -- (383.91,64.43) ;
\draw [color={rgb, 255:red, 0; green, 0; blue, 0 }  ,draw opacity=0.77 ][line width=1.5]    (425.31,64.43) -- (489.37,64.43) ;
\draw [color={rgb, 255:red, 144; green, 19; blue, 254 }  ,draw opacity=0.23 ][line width=4.5]    (136,66) -- (169.91,64.43) -- (158.11,28.35) -- (189.9,3.56) -- (220.78,27.47) -- (211.31,64.43) -- (276.91,64.43) -- (383.91,64.43) -- (373.65,28.35) -- (403.9,5.14) -- (434.78,29.04) -- (425.31,64.43) -- (556,65.91) ;
\draw [shift={(564,66)}, rotate = 180.65] [fill={rgb, 255:red, 144; green, 19; blue, 254 }  ,fill opacity=0.23 ][line width=0.08]  [draw opacity=0] (24.11,-11.58) -- (0,0) -- (24.11,11.58) -- cycle    ;

\end{tikzpicture}

    \caption{
    An absorbing path of gadgets.  The path indicated in purple shows a subpath that uses some triples of colours $a_i,b_i,c_i$, but not others.}
    \label{fig:absorbingpath}
    \end{minipage}
\end{figure}

\par \textbf{Step 2.} We take the last vertex of the absorbing path from Step 1 and use it as the first vertex for a $99\%$ rainbow path as described in the previous subsection.  See Figure~\ref{fig:absorbing_plus_99}.  (More precisely, the $99\%$ path will use $99\%$ of the colours not already used in the absorbing path.)  We need to ensure that the absorbing path is vertex-disjoint from the $99\%$ path. In the everywhere-expanding case from the previous subsection, this is not too difficult since we always have enough choices to avoid an absorbing path fixed from the outset. In the structured case, however, we do not have such freedom, so instead we will build the absorbing path and the $99\%$ path in disjoint random subsets of $G$; this introduces several technical difficulties that we will gloss over for now. 
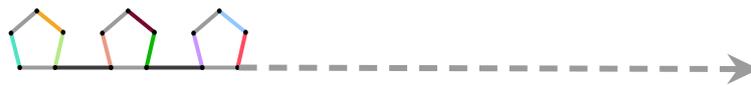
\begin{figure}[h]
\tikzset{every picture/.style={line width=0.75pt}} 

\begin{tikzpicture}[x=0.75pt,y=0.75pt,yscale=-1,xscale=1]

\draw [color={rgb, 255:red, 155; green, 155; blue, 155 }  ,draw opacity=1 ][line width=1.5]    (164.44,15.76) -- (151.56,26.91) ;
\draw [color={rgb, 255:red, 245; green, 166; blue, 35 }  ,draw opacity=1 ][line width=1.5]    (164.44,15.76) -- (177.58,26.49) ;
\draw [color={rgb, 255:red, 184; green, 233; blue, 134 }  ,draw opacity=1 ][line width=1.5]    (177.58,26.49) -- (173.55,44.24) ;
\draw [color={rgb, 255:red, 155; green, 155; blue, 155 }  ,draw opacity=1 ][line width=1.5]    (155.92,44.24) -- (172.9,44.24) ;
\draw [color={rgb, 255:red, 80; green, 227; blue, 194 }  ,draw opacity=1 ][line width=1.5]    (155.92,44.24) -- (151.56,26.91) ;
\draw  [fill={rgb, 255:red, 0; green, 0; blue, 0 }  ,fill opacity=1 ] (176.93,26.49) .. controls (176.93,26.07) and (177.22,25.73) .. (177.58,25.73) .. controls (177.94,25.73) and (178.24,26.07) .. (178.24,26.49) .. controls (178.24,26.9) and (177.94,27.24) .. (177.58,27.24) .. controls (177.22,27.24) and (176.93,26.9) .. (176.93,26.49) -- cycle ;
\draw  [fill={rgb, 255:red, 0; green, 0; blue, 0 }  ,fill opacity=1 ] (163.78,15.76) .. controls (163.78,15.34) and (164.08,15) .. (164.44,15) .. controls (164.8,15) and (165.09,15.34) .. (165.09,15.76) .. controls (165.09,16.17) and (164.8,16.51) .. (164.44,16.51) .. controls (164.08,16.51) and (163.78,16.17) .. (163.78,15.76) -- cycle ;
\draw  [fill={rgb, 255:red, 0; green, 0; blue, 0 }  ,fill opacity=1 ] (172.9,44.24) .. controls (172.9,43.83) and (173.19,43.49) .. (173.55,43.49) .. controls (173.91,43.49) and (174.2,43.83) .. (174.2,44.24) .. controls (174.2,44.66) and (173.91,45) .. (173.55,45) .. controls (173.19,45) and (172.9,44.66) .. (172.9,44.24) -- cycle ;
\draw  [fill={rgb, 255:red, 0; green, 0; blue, 0 }  ,fill opacity=1 ] (155.27,44.24) .. controls (155.27,43.83) and (155.56,43.49) .. (155.92,43.49) .. controls (156.28,43.49) and (156.58,43.83) .. (156.58,44.24) .. controls (156.58,44.66) and (156.28,45) .. (155.92,45) .. controls (155.56,45) and (155.27,44.66) .. (155.27,44.24) -- cycle ;
\draw  [fill={rgb, 255:red, 0; green, 0; blue, 0 }  ,fill opacity=1 ] (150.9,26.91) .. controls (150.9,26.49) and (151.2,26.16) .. (151.56,26.16) .. controls (151.92,26.16) and (152.21,26.49) .. (152.21,26.91) .. controls (152.21,27.33) and (151.92,27.67) .. (151.56,27.67) .. controls (151.2,27.67) and (150.9,27.33) .. (150.9,26.91) -- cycle ;
\draw [color={rgb, 255:red, 155; green, 155; blue, 155 }  ,draw opacity=1 ][line width=1.5]    (209.99,15.76) -- (197.11,26.91) ;
\draw [color={rgb, 255:red, 117; green, 5; blue, 44 }  ,draw opacity=1 ][line width=1.5]    (209.99,15.76) -- (223.14,26.49) ;
\draw [color={rgb, 255:red, 8; green, 185; blue, 0 }  ,draw opacity=1 ][line width=1.5]    (223.14,26.49) -- (219.11,44.24) ;
\draw [color={rgb, 255:red, 155; green, 155; blue, 155 }  ,draw opacity=1 ][line width=1.5]    (201.48,44.24) -- (218.45,44.24) ;
\draw [color={rgb, 255:red, 233; green, 155; blue, 134 }  ,draw opacity=1 ][line width=1.5]    (201.48,44.24) -- (197.11,26.91) ;
\draw  [fill={rgb, 255:red, 0; green, 0; blue, 0 }  ,fill opacity=1 ] (222.49,26.49) .. controls (222.49,26.07) and (222.78,25.73) .. (223.14,25.73) .. controls (223.5,25.73) and (223.79,26.07) .. (223.79,26.49) .. controls (223.79,26.9) and (223.5,27.24) .. (223.14,27.24) .. controls (222.78,27.24) and (222.49,26.9) .. (222.49,26.49) -- cycle ;
\draw  [fill={rgb, 255:red, 0; green, 0; blue, 0 }  ,fill opacity=1 ] (209.34,15.76) .. controls (209.34,15.34) and (209.63,15) .. (209.99,15) .. controls (210.35,15) and (210.64,15.34) .. (210.64,15.76) .. controls (210.64,16.17) and (210.35,16.51) .. (209.99,16.51) .. controls (209.63,16.51) and (209.34,16.17) .. (209.34,15.76) -- cycle ;
\draw  [fill={rgb, 255:red, 0; green, 0; blue, 0 }  ,fill opacity=1 ] (218.45,44.24) .. controls (218.45,43.83) and (218.74,43.49) .. (219.11,43.49) .. controls (219.47,43.49) and (219.76,43.83) .. (219.76,44.24) .. controls (219.76,44.66) and (219.47,45) .. (219.11,45) .. controls (218.74,45) and (218.45,44.66) .. (218.45,44.24) -- cycle ;
\draw  [fill={rgb, 255:red, 0; green, 0; blue, 0 }  ,fill opacity=1 ] (200.83,44.24) .. controls (200.83,43.83) and (201.12,43.49) .. (201.48,43.49) .. controls (201.84,43.49) and (202.13,43.83) .. (202.13,44.24) .. controls (202.13,44.66) and (201.84,45) .. (201.48,45) .. controls (201.12,45) and (200.83,44.66) .. (200.83,44.24) -- cycle ;
\draw  [fill={rgb, 255:red, 0; green, 0; blue, 0 }  ,fill opacity=1 ] (196.46,26.91) .. controls (196.46,26.49) and (196.75,26.16) .. (197.11,26.16) .. controls (197.47,26.16) and (197.76,26.49) .. (197.76,26.91) .. controls (197.76,27.33) and (197.47,27.67) .. (197.11,27.67) .. controls (196.75,27.67) and (196.46,27.33) .. (196.46,26.91) -- cycle ;
\draw [color={rgb, 255:red, 155; green, 155; blue, 155 }  ,draw opacity=1 ][line width=1.5]    (255.55,15.76) -- (242.67,26.91) ;
\draw [color={rgb, 255:red, 148; green, 199; blue, 255 }  ,draw opacity=1 ][line width=1.5]    (255.55,15.76) -- (268.69,26.49) ;
\draw [color={rgb, 255:red, 255; green, 75; blue, 98 }  ,draw opacity=1 ][line width=1.5]    (268.69,26.49) -- (264.66,44.24) ;
\draw [color={rgb, 255:red, 155; green, 155; blue, 155 }  ,draw opacity=1 ][line width=1.5]    (247.03,44.24) -- (264.01,44.24) ;
\draw [color={rgb, 255:red, 202; green, 149; blue, 255 }  ,draw opacity=1 ][line width=1.5]    (247.03,44.24) -- (242.67,26.91) ;
\draw  [fill={rgb, 255:red, 0; green, 0; blue, 0 }  ,fill opacity=1 ] (268.04,26.49) .. controls (268.04,26.07) and (268.33,25.73) .. (268.69,25.73) .. controls (269.05,25.73) and (269.35,26.07) .. (269.35,26.49) .. controls (269.35,26.9) and (269.05,27.24) .. (268.69,27.24) .. controls (268.33,27.24) and (268.04,26.9) .. (268.04,26.49) -- cycle ;
\draw  [fill={rgb, 255:red, 0; green, 0; blue, 0 }  ,fill opacity=1 ] (254.89,15.76) .. controls (254.89,15.34) and (255.19,15) .. (255.55,15) .. controls (255.91,15) and (256.2,15.34) .. (256.2,15.76) .. controls (256.2,16.17) and (255.91,16.51) .. (255.55,16.51) .. controls (255.19,16.51) and (254.89,16.17) .. (254.89,15.76) -- cycle ;
\draw  [fill={rgb, 255:red, 0; green, 0; blue, 0 }  ,fill opacity=1 ] (264.01,44.24) .. controls (264.01,43.83) and (264.3,43.49) .. (264.66,43.49) .. controls (265.02,43.49) and (265.31,43.83) .. (265.31,44.24) .. controls (265.31,44.66) and (265.02,45) .. (264.66,45) .. controls (264.3,45) and (264.01,44.66) .. (264.01,44.24) -- cycle ;
\draw  [fill={rgb, 255:red, 0; green, 0; blue, 0 }  ,fill opacity=1 ] (246.38,44.24) .. controls (246.38,43.83) and (246.67,43.49) .. (247.03,43.49) .. controls (247.4,43.49) and (247.69,43.83) .. (247.69,44.24) .. controls (247.69,44.66) and (247.4,45) .. (247.03,45) .. controls (246.67,45) and (246.38,44.66) .. (246.38,44.24) -- cycle ;
\draw  [fill={rgb, 255:red, 0; green, 0; blue, 0 }  ,fill opacity=1 ] (242.01,26.91) .. controls (242.01,26.49) and (242.31,26.16) .. (242.67,26.16) .. controls (243.03,26.16) and (243.32,26.49) .. (243.32,26.91) .. controls (243.32,27.33) and (243.03,27.67) .. (242.67,27.67) .. controls (242.31,27.67) and (242.01,27.33) .. (242.01,26.91) -- cycle ;
\draw [color={rgb, 255:red, 0; green, 0; blue, 0 }  ,draw opacity=0.77 ][line width=1.5]    (174.2,44.24) -- (201.48,44.24) ;
\draw [color={rgb, 255:red, 0; green, 0; blue, 0 }  ,draw opacity=0.77 ][line width=1.5]    (219.11,44.24) -- (246.38,44.24) ;
\draw [color={rgb, 255:red, 155; green, 155; blue, 155 }  ,draw opacity=1 ][fill={rgb, 255:red, 0; green, 0; blue, 0 }  ,fill opacity=0.86 ][line width=2.25]  [dash pattern={on 6.75pt off 4.5pt}]  (265.31,44.24) -- (520,44.99) ;
\draw [shift={(525,45)}, rotate = 180.17] [fill={rgb, 255:red, 155; green, 155; blue, 155 }  ,fill opacity=1 ][line width=0.08]  [draw opacity=0] (16.07,-7.72) -- (0,0) -- (16.07,7.72) -- (10.67,0) -- cycle    ;
\end{tikzpicture}

\caption{
    An absorbing path connected to a $99\%$ path (drawn dashed).}
    \label{fig:absorbing_plus_99}
\end{figure}

\par \textbf{Step 3.} We have now built a rainbow path $P$ that uses $99\%$ of the colours of $S$ and contains a long absorbing path.  The heart of the matter is using the flexibility of our absorbing path to integrate the remaining $1\%$ of the colours.  Let $\mathcal{L}$ denote the set of ``leftover'' colours not yet used.  The key insight is that we can iteratively reduce the size of $\mathcal{L}$ by ``activating'' an absorbing gadget $a_i,b_i,c_i,d_i$ and using the freed-up colours $a_i,b_i,c_i$ elsewhere.

As long as $|\mathcal{L}| \geq 3$, choose some three elements $\ell_1, \ell_2, \ell_3 \in \mathcal{L}$.  Consider all of the $4$-edge extensions of $P$ using the colours $a_i, \ell_1, \ell_2, \ell_3$ in this order, for $i$ ranging over the indices of the absorbing gadgets that have not yet been activated.  See \Cref{fig:out-spider-extension}.  This figure is a bit misleading since the $4$-edge paths may intersect one another or earlier parts of $P$, but let us suppose for the moment that we can find some $4$-edge path, corresponding to the index $i_0$, which does not intersect $P$.  Then, we modify $P$ as follows: we ``activate'' the gadget indexed by $i_0$ and free up the colours $a_{i_0}, b_{i_0}, c_{i_0}$ by taking the shortcut along the colour $d_{i_0}$; and we extend $P$ by adding the length-$4$ path with index $i_0$.  We then update the leftover set $\mathcal{L}$ by removing $\ell_1,\ell_2,\ell_3$ and adding $b_{i_0},c_{i_0}$.  In total, we have succeeded in reducing the size of $\mathcal{L}$ by $1$.

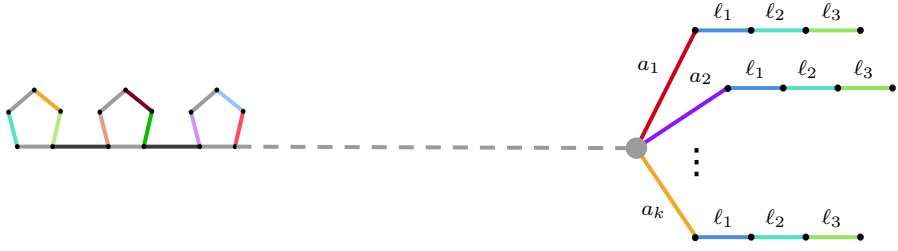
\begin{figure}[h]
    \centering
\tikzset{every picture/.style={line width=0.75pt}} 

\begin{tikzpicture}[x=0.75pt,y=0.75pt,yscale=-1,xscale=1]

\draw [color={rgb, 255:red, 74; green, 144; blue, 226 }  ,draw opacity=1 ][line width=1.5]    (420,15.76) -- (447.28,15.76) ;
\draw [color={rgb, 255:red, 80; green, 227; blue, 194 }  ,draw opacity=1 ][line width=1.5]    (447.93,15.76) -- (475.21,15.76) ;
\draw  [fill={rgb, 255:red, 0; green, 0; blue, 0 }  ,fill opacity=1 ][line width=1.5]  (446.62,15.76) .. controls (446.62,15.34) and (446.92,15) .. (447.28,15) .. controls (447.64,15) and (447.93,15.34) .. (447.93,15.76) .. controls (447.93,16.17) and (447.64,16.51) .. (447.28,16.51) .. controls (446.92,16.51) and (446.62,16.17) .. (446.62,15.76) -- cycle ;
\draw [color={rgb, 255:red, 138; green, 227; blue, 80 }  ,draw opacity=0.92 ][line width=1.5]    (474.55,15.76) -- (501.83,15.76) ;
\draw  [fill={rgb, 255:red, 0; green, 0; blue, 0 }  ,fill opacity=1 ][line width=1.5]  (473.9,15.76) .. controls (473.9,15.34) and (474.19,15) .. (474.55,15) .. controls (474.91,15) and (475.21,15.34) .. (475.21,15.76) .. controls (475.21,16.17) and (474.91,16.51) .. (474.55,16.51) .. controls (474.19,16.51) and (473.9,16.17) .. (473.9,15.76) -- cycle ;
\draw  [fill={rgb, 255:red, 0; green, 0; blue, 0 }  ,fill opacity=1 ][line width=1.5]  (501.18,15.76) .. controls (501.18,15.34) and (501.47,15) .. (501.83,15) .. controls (502.19,15) and (502.48,15.34) .. (502.48,15.76) .. controls (502.48,16.17) and (502.19,16.51) .. (501.83,16.51) .. controls (501.47,16.51) and (501.18,16.17) .. (501.18,15.76) -- cycle ;
\draw [color={rgb, 255:red, 74; green, 144; blue, 226 }  ,draw opacity=1 ][line width=1.5]    (436,44.76) -- (463.28,44.76) ;
\draw [color={rgb, 255:red, 80; green, 227; blue, 194 }  ,draw opacity=1 ][line width=1.5]    (463.93,44.76) -- (491.21,44.76) ;
\draw  [fill={rgb, 255:red, 0; green, 0; blue, 0 }  ,fill opacity=1 ][line width=1.5]  (462.62,44.76) .. controls (462.62,44.34) and (462.92,44) .. (463.28,44) .. controls (463.64,44) and (463.93,44.34) .. (463.93,44.76) .. controls (463.93,45.17) and (463.64,45.51) .. (463.28,45.51) .. controls (462.92,45.51) and (462.62,45.17) .. (462.62,44.76) -- cycle ;
\draw [color={rgb, 255:red, 138; green, 227; blue, 80 }  ,draw opacity=0.92 ][line width=1.5]    (490.55,44.76) -- (517.83,44.76) ;
\draw  [fill={rgb, 255:red, 0; green, 0; blue, 0 }  ,fill opacity=1 ][line width=1.5]  (489.9,44.76) .. controls (489.9,44.34) and (490.19,44) .. (490.55,44) .. controls (490.91,44) and (491.21,44.34) .. (491.21,44.76) .. controls (491.21,45.17) and (490.91,45.51) .. (490.55,45.51) .. controls (490.19,45.51) and (489.9,45.17) .. (489.9,44.76) -- cycle ;
\draw  [fill={rgb, 255:red, 0; green, 0; blue, 0 }  ,fill opacity=1 ][line width=1.5]  (517.18,44.76) .. controls (517.18,44.34) and (517.47,44) .. (517.83,44) .. controls (518.19,44) and (518.48,44.34) .. (518.48,44.76) .. controls (518.48,45.17) and (518.19,45.51) .. (517.83,45.51) .. controls (517.47,45.51) and (517.18,45.17) .. (517.18,44.76) -- cycle ;
\draw [color={rgb, 255:red, 74; green, 144; blue, 226 }  ,draw opacity=1 ][line width=1.5]    (420,119.76) -- (447.28,119.76) ;
\draw [color={rgb, 255:red, 80; green, 227; blue, 194 }  ,draw opacity=1 ][line width=1.5]    (447.93,119.76) -- (475.21,119.76) ;
\draw  [fill={rgb, 255:red, 0; green, 0; blue, 0 }  ,fill opacity=1 ][line width=1.5]  (446.62,119.76) .. controls (446.62,119.34) and (446.92,119) .. (447.28,119) .. controls (447.64,119) and (447.93,119.34) .. (447.93,119.76) .. controls (447.93,120.17) and (447.64,120.51) .. (447.28,120.51) .. controls (446.92,120.51) and (446.62,120.17) .. (446.62,119.76) -- cycle ;
\draw [color={rgb, 255:red, 138; green, 227; blue, 80 }  ,draw opacity=0.92 ][line width=1.5]    (474.55,119.76) -- (501.83,119.76) ;
\draw  [fill={rgb, 255:red, 0; green, 0; blue, 0 }  ,fill opacity=1 ][line width=1.5]  (473.9,119.76) .. controls (473.9,119.34) and (474.19,119) .. (474.55,119) .. controls (474.91,119) and (475.21,119.34) .. (475.21,119.76) .. controls (475.21,120.17) and (474.91,120.51) .. (474.55,120.51) .. controls (474.19,120.51) and (473.9,120.17) .. (473.9,119.76) -- cycle ;
\draw  [fill={rgb, 255:red, 0; green, 0; blue, 0 }  ,fill opacity=1 ][line width=1.5]  (501.18,119.76) .. controls (501.18,119.34) and (501.47,119) .. (501.83,119) .. controls (502.19,119) and (502.48,119.34) .. (502.48,119.76) .. controls (502.48,120.17) and (502.19,120.51) .. (501.83,120.51) .. controls (501.47,120.51) and (501.18,120.17) .. (501.18,119.76) -- cycle ;

\draw [color={rgb, 255:red, 155; green, 155; blue, 155 }  ,draw opacity=1 ][line width=1.5]    (89.44,45.76) -- (76.56,56.91) ;
\draw [color={rgb, 255:red, 245; green, 166; blue, 35 }  ,draw opacity=1 ][line width=1.5]    (89.44,45.76) -- (102.58,56.49) ;
\draw [color={rgb, 255:red, 184; green, 233; blue, 134 }  ,draw opacity=1 ][line width=1.5]    (102.58,56.49) -- (98.55,74.24) ;
\draw [color={rgb, 255:red, 155; green, 155; blue, 155 }  ,draw opacity=1 ][line width=1.5]    (80.92,74.24) -- (97.9,74.24) ;
\draw [color={rgb, 255:red, 80; green, 227; blue, 194 }  ,draw opacity=1 ][line width=1.5]    (80.92,74.24) -- (76.56,56.91) ;
\draw  [fill={rgb, 255:red, 0; green, 0; blue, 0 }  ,fill opacity=1 ] (101.93,56.49) .. controls (101.93,56.07) and (102.22,55.73) .. (102.58,55.73) .. controls (102.94,55.73) and (103.24,56.07) .. (103.24,56.49) .. controls (103.24,56.9) and (102.94,57.24) .. (102.58,57.24) .. controls (102.22,57.24) and (101.93,56.9) .. (101.93,56.49) -- cycle ;
\draw  [fill={rgb, 255:red, 0; green, 0; blue, 0 }  ,fill opacity=1 ] (88.78,45.76) .. controls (88.78,45.34) and (89.08,45) .. (89.44,45) .. controls (89.8,45) and (90.09,45.34) .. (90.09,45.76) .. controls (90.09,46.17) and (89.8,46.51) .. (89.44,46.51) .. controls (89.08,46.51) and (88.78,46.17) .. (88.78,45.76) -- cycle ;
\draw  [fill={rgb, 255:red, 0; green, 0; blue, 0 }  ,fill opacity=1 ] (97.9,74.24) .. controls (97.9,73.83) and (98.19,73.49) .. (98.55,73.49) .. controls (98.91,73.49) and (99.2,73.83) .. (99.2,74.24) .. controls (99.2,74.66) and (98.91,75) .. (98.55,75) .. controls (98.19,75) and (97.9,74.66) .. (97.9,74.24) -- cycle ;
\draw  [fill={rgb, 255:red, 0; green, 0; blue, 0 }  ,fill opacity=1 ] (80.27,74.24) .. controls (80.27,73.83) and (80.56,73.49) .. (80.92,73.49) .. controls (81.28,73.49) and (81.58,73.83) .. (81.58,74.24) .. controls (81.58,74.66) and (81.28,75) .. (80.92,75) .. controls (80.56,75) and (80.27,74.66) .. (80.27,74.24) -- cycle ;
\draw  [fill={rgb, 255:red, 0; green, 0; blue, 0 }  ,fill opacity=1 ] (75.9,56.91) .. controls (75.9,56.49) and (76.2,56.16) .. (76.56,56.16) .. controls (76.92,56.16) and (77.21,56.49) .. (77.21,56.91) .. controls (77.21,57.33) and (76.92,57.67) .. (76.56,57.67) .. controls (76.2,57.67) and (75.9,57.33) .. (75.9,56.91) -- cycle ;
\draw [color={rgb, 255:red, 155; green, 155; blue, 155 }  ,draw opacity=1 ][line width=1.5]    (134.99,45.76) -- (122.11,56.91) ;
\draw [color={rgb, 255:red, 117; green, 5; blue, 44 }  ,draw opacity=1 ][line width=1.5]    (134.99,45.76) -- (148.14,56.49) ;
\draw [color={rgb, 255:red, 8; green, 185; blue, 0 }  ,draw opacity=1 ][line width=1.5]    (148.14,56.49) -- (144.11,74.24) ;
\draw [color={rgb, 255:red, 155; green, 155; blue, 155 }  ,draw opacity=1 ][line width=1.5]    (126.48,74.24) -- (143.45,74.24) ;
\draw [color={rgb, 255:red, 233; green, 155; blue, 134 }  ,draw opacity=1 ][line width=1.5]    (126.48,74.24) -- (122.11,56.91) ;
\draw  [fill={rgb, 255:red, 0; green, 0; blue, 0 }  ,fill opacity=1 ] (147.49,56.49) .. controls (147.49,56.07) and (147.78,55.73) .. (148.14,55.73) .. controls (148.5,55.73) and (148.79,56.07) .. (148.79,56.49) .. controls (148.79,56.9) and (148.5,57.24) .. (148.14,57.24) .. controls (147.78,57.24) and (147.49,56.9) .. (147.49,56.49) -- cycle ;
\draw  [fill={rgb, 255:red, 0; green, 0; blue, 0 }  ,fill opacity=1 ] (134.34,45.76) .. controls (134.34,45.34) and (134.63,45) .. (134.99,45) .. controls (135.35,45) and (135.64,45.34) .. (135.64,45.76) .. controls (135.64,46.17) and (135.35,46.51) .. (134.99,46.51) .. controls (134.63,46.51) and (134.34,46.17) .. (134.34,45.76) -- cycle ;
\draw  [fill={rgb, 255:red, 0; green, 0; blue, 0 }  ,fill opacity=1 ] (143.45,74.24) .. controls (143.45,73.83) and (143.74,73.49) .. (144.11,73.49) .. controls (144.47,73.49) and (144.76,73.83) .. (144.76,74.24) .. controls (144.76,74.66) and (144.47,75) .. (144.11,75) .. controls (143.74,75) and (143.45,74.66) .. (143.45,74.24) -- cycle ;
\draw  [fill={rgb, 255:red, 0; green, 0; blue, 0 }  ,fill opacity=1 ] (125.83,74.24) .. controls (125.83,73.83) and (126.12,73.49) .. (126.48,73.49) .. controls (126.84,73.49) and (127.13,73.83) .. (127.13,74.24) .. controls (127.13,74.66) and (126.84,75) .. (126.48,75) .. controls (126.12,75) and (125.83,74.66) .. (125.83,74.24) -- cycle ;
\draw  [fill={rgb, 255:red, 0; green, 0; blue, 0 }  ,fill opacity=1 ] (121.46,56.91) .. controls (121.46,56.49) and (121.75,56.16) .. (122.11,56.16) .. controls (122.47,56.16) and (122.76,56.49) .. (122.76,56.91) .. controls (122.76,57.33) and (122.47,57.67) .. (122.11,57.67) .. controls (121.75,57.67) and (121.46,57.33) .. (121.46,56.91) -- cycle ;
\draw [color={rgb, 255:red, 155; green, 155; blue, 155 }  ,draw opacity=1 ][line width=1.5]    (180.55,45.76) -- (167.67,56.91) ;
\draw [color={rgb, 255:red, 148; green, 199; blue, 255 }  ,draw opacity=1 ][line width=1.5]    (180.55,45.76) -- (193.69,56.49) ;
\draw [color={rgb, 255:red, 255; green, 75; blue, 98 }  ,draw opacity=1 ][line width=1.5]    (193.69,56.49) -- (189.66,74.24) ;
\draw [color={rgb, 255:red, 155; green, 155; blue, 155 }  ,draw opacity=1 ][line width=1.5]    (172.03,74.24) -- (189.01,74.24) ;
\draw [color={rgb, 255:red, 202; green, 149; blue, 255 }  ,draw opacity=1 ][line width=1.5]    (172.03,74.24) -- (167.67,56.91) ;
\draw  [fill={rgb, 255:red, 0; green, 0; blue, 0 }  ,fill opacity=1 ] (193.04,56.49) .. controls (193.04,56.07) and (193.33,55.73) .. (193.69,55.73) .. controls (194.05,55.73) and (194.35,56.07) .. (194.35,56.49) .. controls (194.35,56.9) and (194.05,57.24) .. (193.69,57.24) .. controls (193.33,57.24) and (193.04,56.9) .. (193.04,56.49) -- cycle ;
\draw  [fill={rgb, 255:red, 0; green, 0; blue, 0 }  ,fill opacity=1 ] (179.89,45.76) .. controls (179.89,45.34) and (180.19,45) .. (180.55,45) .. controls (180.91,45) and (181.2,45.34) .. (181.2,45.76) .. controls (181.2,46.17) and (180.91,46.51) .. (180.55,46.51) .. controls (180.19,46.51) and (179.89,46.17) .. (179.89,45.76) -- cycle ;
\draw  [fill={rgb, 255:red, 0; green, 0; blue, 0 }  ,fill opacity=1 ] (189.01,74.24) .. controls (189.01,73.83) and (189.3,73.49) .. (189.66,73.49) .. controls (190.02,73.49) and (190.31,73.83) .. (190.31,74.24) .. controls (190.31,74.66) and (190.02,75) .. (189.66,75) .. controls (189.3,75) and (189.01,74.66) .. (189.01,74.24) -- cycle ;
\draw  [fill={rgb, 255:red, 0; green, 0; blue, 0 }  ,fill opacity=1 ] (171.38,74.24) .. controls (171.38,73.83) and (171.67,73.49) .. (172.03,73.49) .. controls (172.4,73.49) and (172.69,73.83) .. (172.69,74.24) .. controls (172.69,74.66) and (172.4,75) .. (172.03,75) .. controls (171.67,75) and (171.38,74.66) .. (171.38,74.24) -- cycle ;
\draw  [fill={rgb, 255:red, 0; green, 0; blue, 0 }  ,fill opacity=1 ] (167.01,56.91) .. controls (167.01,56.49) and (167.31,56.16) .. (167.67,56.16) .. controls (168.03,56.16) and (168.32,56.49) .. (168.32,56.91) .. controls (168.32,57.33) and (168.03,57.67) .. (167.67,57.67) .. controls (167.31,57.67) and (167.01,57.33) .. (167.01,56.91) -- cycle ;
\draw [color={rgb, 255:red, 0; green, 0; blue, 0 }  ,draw opacity=0.77 ][line width=1.5]    (99.2,74.24) -- (126.48,74.24) ;
\draw [color={rgb, 255:red, 0; green, 0; blue, 0 }  ,draw opacity=0.77 ][line width=1.5]    (144.11,74.24) -- (171.38,74.24) ;
\draw [color={rgb, 255:red, 208; green, 2; blue, 27 }  ,draw opacity=1 ][line width=1.5]    (390,75) -- (420,15) ;
\draw [color={rgb, 255:red, 144; green, 19; blue, 254 }  ,draw opacity=1 ][line width=1.5]    (390,75) -- (435,45) ;
\draw [color={rgb, 255:red, 0; green, 0; blue, 0 }  ,draw opacity=1 ][fill={rgb, 255:red, 0; green, 0; blue, 0 }  ,fill opacity=1 ][line width=1.5]  [dash pattern={on 1.69pt off 2.76pt}]  (420,75) -- (420,90) ;
\draw [color={rgb, 255:red, 245; green, 166; blue, 35 }  ,draw opacity=1 ][line width=1.5]    (390,75) -- (420,120) ;
\draw  [fill={rgb, 255:red, 0; green, 0; blue, 0 }  ,fill opacity=1 ][line width=1.5]  (418.69,15.76) .. controls (418.69,15.34) and (418.99,15) .. (419.35,15) .. controls (419.71,15) and (420,15.34) .. (420,15.76) .. controls (420,16.17) and (419.71,16.51) .. (419.35,16.51) .. controls (418.99,16.51) and (418.69,16.17) .. (418.69,15.76) -- cycle ;
\draw  [fill={rgb, 255:red, 0; green, 0; blue, 0 }  ,fill opacity=1 ][line width=1.5]  (435,45) .. controls (435,44.58) and (435.29,44.24) .. (435.65,44.24) .. controls (436.01,44.24) and (436.31,44.58) .. (436.31,45) .. controls (436.31,45.42) and (436.01,45.76) .. (435.65,45.76) .. controls (435.29,45.76) and (435,45.42) .. (435,45) -- cycle ;
\draw  [fill={rgb, 255:red, 0; green, 0; blue, 0 }  ,fill opacity=1 ][line width=1.5]  (418.69,120) .. controls (418.69,119.58) and (418.99,119.24) .. (419.35,119.24) .. controls (419.71,119.24) and (420,119.58) .. (420,120) .. controls (420,120.42) and (419.71,120.76) .. (419.35,120.76) .. controls (418.99,120.76) and (418.69,120.42) .. (418.69,120) -- cycle ;
\draw [color={rgb, 255:red, 155; green, 155; blue, 155 }  ,draw opacity=1 ][fill={rgb, 255:red, 0; green, 0; blue, 0 }  ,fill opacity=0.86 ][line width=1.5]  [dash pattern={on 5.63pt off 4.5pt}]  (190.31,74.24) -- (390,75) ;
\draw [shift={(390,75)}, rotate = 0.22] [color={rgb, 255:red, 155; green, 155; blue, 155 }  ,draw opacity=1 ][fill={rgb, 255:red, 155; green, 155; blue, 155 }  ,fill opacity=1 ][line width=1.5]      (0, 0) circle [x radius= 4.36, y radius= 4.36]   ;

\draw (389,29.4) node [anchor=north west][inner sep=0.75pt]  [font=\footnotesize]  {$a_{1}$};
\draw (415,35.4) node [anchor=north west][inner sep=0.75pt]  [font=\footnotesize]  {$a_{2}$};
\draw (391,102.4) node [anchor=north west][inner sep=0.75pt]  [font=\footnotesize]  {$a_{k}$};
\draw (427,0.4) node [anchor=north west][inner sep=0.75pt]  [font=\footnotesize]  {$\ell _{1}$};
\draw (453,0.4) node [anchor=north west][inner sep=0.75pt]  [font=\footnotesize]  {$\ell _{2}$};
\draw (481,0.4) node [anchor=north west][inner sep=0.75pt]  [font=\footnotesize]  {$\ell _{3}$};
\draw (443,29.4) node [anchor=north west][inner sep=0.75pt]  [font=\footnotesize]  {$\ell _{1}$};
\draw (469,29.4) node [anchor=north west][inner sep=0.75pt]  [font=\footnotesize]  {$\ell _{2}$};
\draw (497,29.4) node [anchor=north west][inner sep=0.75pt]  [font=\footnotesize]  {$\ell _{3}$};
\draw (427,104.4) node [anchor=north west][inner sep=0.75pt]  [font=\footnotesize]  {$\ell _{1}$};
\draw (453,104.4) node [anchor=north west][inner sep=0.75pt]  [font=\footnotesize]  {$\ell _{2}$};
\draw (481,104.4) node [anchor=north west][inner sep=0.75pt]  [font=\footnotesize]  {$\ell _{3}$};

\end{tikzpicture}
\caption{
    Various options for extending our long rainbow path by a $4$-edge path.  To append one of the $4$-edge paths, we activate the corresponding gadget.}
    
    \label{fig:out-spider-extension}
\end{figure}

We have omitted two important technical points from this discussion.  The first point concerns ensuring that we can always find a length-$4$ path that does not intersect other parts of our structure.  How we ensure this depends on whether or not $S$ has an everywhere-expanding part. If $S$ does have an everywhere-expanding part, then we can use this expansion to make our candidate length-$4$ extensions ``spread out''.  We will show that in the remaining case of structured $S$, we can carry out Steps 1, 2, and 3 in disjoint random vertex subsets of $\Cayley_G(S)$, effectively avoiding this issue altogether.  The second point is that our iterative procedure terminates once $|\mathcal{L}|$ drops below $3$.  Saturating the remaining $2$ colours is a delicate matter that we will discuss later in the paper.

\section{Notation and Prerequisites}\label{sec:notation}

Given a set $X$, a \emph{$p$-random subset} of $X$ is obtained by sampling each element of $X$ with probability $p$, independently of all other elements.

We will need the following basic concentration bound.
\begin{lemma}[Chernoff's inequality]\label{chernoff}
    Let $X$ be a sum of independent Bernoulli random variables with $\mathbb E(X) = \mu$. Then, for every $t > 0$,
    \begin{itemize}
        \item $\mathbb P(X \leq \mu - t) \leq \exp(-t^2/(2\mu))$;
        \item $\mathbb P(X \geq \mu + t) \leq \exp(-t^2/(2\mu + t))$.
    \end{itemize}
\end{lemma}

The digraphs we consider are loopless, and for each pair $(u,v)$ of distinct vertices, we allow at most one edge from $u$ to $v$, which we denote by $(u,v)$. We do, however, allow both edges $(u,v)$ and $(v,u)$ to appear in the same digraph. If $G$ is a (possibly edge-coloured) digraph, then for $U, V \subseteq V(G)$, we write $e_G(U,V)$ to denote the number of edges $(u,v)$ with $u \in U$ and $v \in V$. As special cases, for a vertex $v \in V(G)$, we denote the \emph{out-degree} of $v$ by $\deg^+_G(v) := e_G(\{v\}, V(G))$ and the \emph{in-degree} of $v$ by $\deg^-_G(v) := e_G(V(G), \{v\})$. We denote the minimum out-degree and in-degree of $G$ by $\delta^+(G)$ and $\delta^-(G)$, respectively, and we write $\delta^{\pm}(G) := \min \{\delta^+(G), \delta^-(G)\}$ for the \emph{minimum semi-degree} of $G$.

\medskip

\textbf{Nonabelian Fourier analysis.} We shall make use of some nonabelian Fourier analysis for finite groups in order to prove the regularity result in Theorem \ref{prop:RobustCayley-generalnonAbelian}; we record all the basic properties that we need here. Again, we mention that to prove Lemma \ref{cor:robexpander}, it suffices to use Fourier analysis over $\mathbb{F}_2^n$. The reader who wishes to focus on this result may skip ahead to the end of this section, where we separately state the basic properties of Fourier analysis over $\Fn$.

\medskip

Let \( G \) be a finite (possibly nonabelian) group. We use the following standard notation:
\begin{itemize}
  \item \( |G| \): the order of \( G \),
  \item \( \widehat{G} \): the set of irreducible complex representations of \( G \),
  \item \( \rho \in \widehat{G} \): a representation \( \rho: G \to \mathrm{GL}(V_\rho) \), which means that $\rho$ is a group homomorphism from $G$ to $\mathrm{GL}(V_\rho)$,
  \item \( d_\rho = \dim V_\rho \) is the \emph{degree} of the representation $\rho$, and note that $\sum_{\rho\in\hat{G}}d_\rho^2=|G|$.
\end{itemize}
We will write $\triv$ for the trivial irreducible representation. For a vector $v\in V_\rho$, we will write $\lVert v\rVert_{V_\rho}^2=\langle v,v\rangle_{V_\rho}$ where we take $\langle \cdot ,\cdot\rangle_{V_\rho}$ to be a Hermitian inner product in each of the vector spaces $V_\rho$, for each irreducible representation $\rho$. By Weyl's unitary trick, we can and will always assume that each of the representations $\rho$ is unitary with respect to $\langle \cdot,\cdot\rangle_{V_\rho}$, meaning that all the matrices $\rho(g), g\in G$ are unitary so that $\langle \rho(g)v,\rho(g)v\rangle_{V_\rho} =\langle v,v\rangle_{V_\rho}$ for all $v\in V_\rho$ and $g\in G$. Recall also that a matrix $A$ is said to be \emph{unitary} if it satisfies $\overline{A}^T=A^{-1}$.

\par
For a function \( f : G \to \mathbb{C} \), the Fourier transform of $f$ at \( \rho \in \widehat{G} \) is given by

\[
\widehat{f}(\rho) = \sum_{x \in G} f(x) \, \rho(x) \in \mathbb{C}^{d_\rho \times d_\rho}.
\]
We shall exclusively consider Fourier transforms $\hat{1}_T(\rho),\rho\in\hat{G}$ of indicator functions of sets $T\subset G$ in this paper. We note that the eigenvalues of the Fourier coefficient matrices $\hat{1}_T(\rho),\rho\in\hat{G}$ are precisely the eigenvalues of the adjacency matrix of the directed graph $\Cayley_G(T)$. In fact, each of the $d_\rho$ eigenvalues of $\hat{1}_T(\rho)$ appears with multiplicity $d_\rho$ as an eigenvalue of the adjacency matrix of $\Cayley_G(T)$ (which has $|G|=\sum_{\rho\in \hat{G}}d_\rho^2$ eigenvalues in total). The value of the function \( f \) at $y\in G$ can be recovered from its Fourier transform via the inversion formula

\[
f(y) = \frac{1}{|G|} \sum_{\rho \in \widehat{G}} d_\rho \cdot \mathrm{Tr} \left( \widehat{f}(\rho) \cdot \rho(y)^* \right)
\]
where \( \rho(x)^* =\overline{\rho(x)}^T\) denotes the conjugate transpose of \( \rho(x) \), and where $\mathrm{Tr}$ is the trace. Parseval's identity states that for $f:G\to\mathbb{C}$ we have
\[
\sum_{x \in G} |f(x)|^2 = \frac{1}{|G|} \sum_{\rho \in \widehat{G}} d_\rho \cdot \| \widehat{f}(\rho) \|_F^2
\]

where \( \| \cdot \|_F \) is the Frobenius norm
$\| A \|_F^2 = \sum_{i,j} |A_{ij}|^2
=\mathrm{Tr}( AA^*)$.

We shall also make use of the following important property for the Fourier transform of the convolution of two functions \( f, g : G \to \mathbb{C} \), which is defined as
$(f * g)(x) = \sum_{y \in G} f(y) g(y^{-1}x)$. The Fourier transform of the convolution satisfies
$\widehat{f * g}(\rho) = \widehat{f}(\rho) \cdot \widehat{g}(\rho)$. 

\par An important observation is that, for sets $X,Y,T\subset G$, the number of solutions $(x,y,t)\in X\times Y\times T$ to $xyt=\id$ can be written using convolutions as $1_X*1_Y*1_T(\id)$, and hence the formula for the Fourier transform of a convolution and Fourier inversion allow us to express this count as $\frac{1}{|G|} \sum_{\rho \in \widehat{G}} d_\rho \cdot \mathrm{Tr} \left( \widehat{1}_X(\rho) \hat{1}_Y(\rho)\hat{1}_T(\rho) \right)$. Let us state one more elementary fact, namely that the Fourier transform of the indicator function of the whole group $G$ is given by $$\hat{1}_G(\rho)=\begin{cases}
    |G|, &\text{ if $\rho=\triv$}\\
    (0)_{d_\rho\times d_\rho}, &\text{ if $\rho\in\hat{G}$ is non-trivial,}
\end{cases}$$ 
where $(0)_{d_\rho\times d_\rho}$ denotes the $d_\rho$ by $d_\rho$ zero matrix. These are all the properties that we require in this paper, for a more extensive overview of the basics of nonabelian Fourier analysis, we refer the reader to \cite{Folland2016}.

\medskip

\textbf{Fourier analysis over $\Fn$.} For the convenience of the reader who wants a streamlined proof of Theorem \ref{thm:mainthm}, we briefly discuss the results above in the specialised setting where $G=\Fn$. The dual group $\hat{G}=\hat{\mathbb{F}}_2^n$ of characters on $\Fn$ is isomorphic to $\Fn$, with each character $\gamma\in \hat{G}$ being of the form $$
  \gamma_\xi:\Fn\to\mathbb{R}:x\mapsto (-1)^{\langle \xi, x \rangle}$$ for a $\xi \in \F_2^n$, where $\langle \xi, x \rangle = \sum_{i=1}^n \xi_i x_i$ is the standard dot product over $\Fn$. For a function $f : \F_2^n \to \mathbb{C}$, its Fourier transform is
  \[
  \widehat{f}(\gamma) = \sum_{x \in \F_2^n}f(x)\gamma(x).
  \]
  We have the inversion formula $f(x) = \frac{1}{|G|}\sum_{\gamma\in \hat{G}} \widehat{f}(\gamma) \gamma(x)$, and Parseval's formula states that 
  \[
  \sum_{x\in\Fn} |f(x)|^2 = \frac{1}{|G|}\sum_{\gamma \in \hat{G}} |\widehat{f}(\gamma)|^2.
  \]
For $f, g : \F_2^n \to \mathbb{C}$, their convolution $f*g$ is defined as $(f * g)(x) = \sum_{y \in \F_2^n}f(y)g(x + y)$, and the Fourier transform of a convolution satisfies $\widehat{f * g}(\gamma) = \widehat{f}(\gamma) \widehat{g}(\gamma)$. Finally, we note that the Fourier transform of the indicator function of the whole group $G=\Fn$ is given by $$\hat{1}_G(\gamma)=\begin{cases}
    |G|, &\text{ if $\gamma=0$ is the trivial character}\\
    0, &\text{ if $\gamma\in\hat{G}\setminus\{0\}$.}
\end{cases}$$

\section{A weak nonabelian regularity lemma for finding expander Cayley subgraphs}\label{sec:regularity}
The goal of this section is to prove the nonabelian regularity lemma Theorem~\ref{prop:RobustCayley-generalnonAbelian}, which says that given a subset $S$ of a finite group $G$, we can find a subgroup $H$ of $G$ such that $H$ contains most of $S$ and $\Cayley_H(H \cap S)$ is mildly quasirandom.  We restate Theorem~\ref{prop:RobustCayley-generalnonAbelian} now for the reader's convenience.

\generalexpander*

Let us digress and make a few remarks.  First, the quadratic dependence of $\eta$ on $\sigma$ is optimal, as illustrated by the case where $S$ is an arithmetic progression in $\mathbb{F}_p$.  Second, our arguments can be modified to produce a subgroup $H$ and an element $x\in G$ such that $|S\cap (x^{-1}H)|\geq (1-\varepsilon)|S|$, and all non-trivial eigenvalues of the adjacency matrix of $\Cayley_H(H \cap (xS))$ have \emph{absolute value} (rather than real part) at most $(1-\eta)|(xS)\cap H|$.  Third, another minor variation of the proof shows that there is a subgroup $H$ such that $|S\cap H|\geq (1-\varepsilon)|S|$, and all non-trivial eigenvalues of the adjacency matrix of $\Cayley_H(S\cap H)$ have real part at most $(1-\eta)|S\cap H|$, where $\eta:=\frac{\varepsilon\delta^2}{1000|\log \sigma|}$ and $\delta:=|S\cap H|/|H|$ is the density of $S$ within $H$ (rather than the density $\sigma$ of $S$ within $G$).

Recall from \Cref{sec:notation} that the spectrum of the adjacency matrix of $\Cayley_G(S)$ is precisely the union of the spectra of the Fourier coefficient matrices $\hat{1}_S(\rho)$ for $\rho \in \hat{G}$.  Thus, the second condition in Theorem~\ref{prop:RobustCayley-generalnonAbelian} can be formulated in terms of $S \cap H \subseteq H$ having a spectral gap bounded away from $0$, in the sense of Definition~\ref{def:specgap} below.  The connection between spectral gaps, quasirandomness, and edge-expansion is by now a standard theme in spectral graph theory; see, e.g., the survey \cite{KrivelevichSudakov2006}.  Our formulation of this connection, encapsulated in Lemma~\ref{lem:spectralgapnosparsecutnonAbelian} below, relies on the notion of an $\eta$-sparse cut (as defined in the introduction following the statement of Theorem~\ref{prop:RobustCayley-generalnonAbelian}) and leads to the following corollary of Theorem~\ref{prop:RobustCayley-generalnonAbelian} that will be convenient for our later applications.

\begin{corollary}\label{cor:nosparsecutreg}
    Let $\sigma\in(0,1]$ and $\varepsilon\in(0,1/2)$. Let $G$ be a finite group (not necessarily abelian), and let $S\subseteq G$ be a subset with density $\sigma =|S|/|G|$. Then there is a subgroup $H$ of $G$ such that
    \begin{enumerate}
        \item $|S\cap H|\geq (1-\varepsilon)|S|$;
        \item $\Cayley_H(S\cap H)$ has no $\varepsilon\sigma^3/1000$-sparse cuts.
    \end{enumerate}    
\end{corollary}
We remind the reader that $\Cayley_H(S\cap H)$ has no $\eta$-sparse cuts if for every partition $H=X_1\sqcup X_2$ we have $$\#\{(x_1,x_2)\in X_1\times X_2: x_1^{-1}x_2\in S\}\geq \eta|X_1||X_2|.$$
Our proof of Theorem~\ref{prop:RobustCayley-generalnonAbelian} uses nonabelian Fourier analysis.  As a warm-up, we will start by proving Theorem~\ref{prop:RobustCayley-generalnonAbelian} for the group $\Fn$, where the argument simplifies considerably due to the nature of the Fourier transform on $\Fn$.  This simplified argument also yields a somewhat better (in fact, optimal) quantitative dependence of the sparse-cut parameter on the density $\sigma$.  We remark that only this special case is necessary for the proof of Theorem~\ref{thm:mainthm}, so the reader who is interested only in that result may safely skip our treatment of the general case of Theorem~\ref{prop:RobustCayley-generalnonAbelian}.

\subsection{The \texorpdfstring{$\Fn$}{F2n} case}\label{sec:fncase}
Here is Lemma~\ref{cor:robexpander} restated for the reader's convenience.

\fnexpander*

Before proving this lemma, we will need the following auxiliary result which states that the Cayley graph of a subset $T\subset \Fn$ has no sparse cuts provided that it has a spectral gap. As we discussed at the start of this section, results of this flavour are well-known. 
\begin{lemma}\label{lem:spectralgapnosparsecutFn}
    Let $T\subset H=\Fn$ have a spectral gap $$\max_{\gamma\in\hat{H}:\gamma\neq 0} \hat{1}_T(\gamma)\leq (1-\beta)|T|,$$ then $\Cayley_{H} (T)$ has no $\beta\tau$-sparse cuts, where $\tau=|T|/|H|$.
\end{lemma}

\begin{proof}[Proof of \Cref{lem:spectralgapnosparsecutFn}]
    Let $H=X_1\sqcup X_2$ be a partition of $H=\Fn$, so that our goal is to show that $\#\{(x_1,x_2)\in X_1\times X_2: x_1+x_2\in T\}\geqslant \beta \tau|X_1||X_2|$. By the formula for the Fourier transform of a convolution and Fourier inversion, we can write
    \begin{align*}
    \#\{(x_1,x_2)\in X_1\times X_2: x_1+x_2\in T\}&=1_{X_1}*1_{X_2}*1_T(0)=\frac{1}{|H|}\sum_{\gamma\in\hat{H}}\hat{1}_{X_1}(\gamma)\hat{1}_{X_2}(\gamma)\hat{1}_T(\gamma).    
    \end{align*}
    Since the contribution from the trivial character $\gamma=0$ to the right hand side is $|X_1||X_2||T|/|H|=\tau|X_1||X_2|$, it suffices to show that
    \begin{align}\label{eq:specgapeqFn}
        \frac{-1}{|H|}\sum_{\gamma\in\hat{H}:\gamma\neq 0}\hat{1}_{X_1}(\gamma)\hat{1}_{X_2}(\gamma)\hat{1}_T(\gamma)\leq (1-\beta)\tau|X_1||X_2|
    \end{align}
    as this would show precisely that
    $\#\{(x_1,x_2)\in X_1\times X_2: x_1+x_2\in T\}\geq \beta\tau|X_1||X_2|$.
    Note that as $X_1,X_2$ partition $H$, we have that $$\hat{1}_{X_1}+\hat{1}_{X_2}=\hat{1}_{H}=\begin{cases}
        |H|, &\text{if $\gamma=0$}\\
        0, &\text{otherwise.}
    \end{cases}$$
    So $\hat{1}_{X_2}(\gamma)=-\hat{1}_{X_1}(\gamma)$ at all non-zero characters $\gamma$. Hence,
    \begin{align*}
    \frac{-1}{|H|}\sum_{\gamma\in\hat{H}:\gamma\neq 0}\hat{1}_{X_1}(\gamma)\hat{1}_{X_2}(\gamma)\hat{1}_T(\gamma)= \frac{1}{|H|}\sum_{\gamma\in\hat{H}:\gamma\neq 0}|\hat{1}_{X_1}(\gamma)|^2\hat{1}_T(\gamma).
    \end{align*}
    Now we simply invoke the spectral gap assumption to bound this by
    \begin{align*}
        \frac{(1-\beta)|T|}{|H|}\sum_{\gamma\neq 0}|\hat{1}_{X_1}(\gamma)|^2&=(1-\beta)|T|\left(|X_1|-\frac{|X_1|^2}{|H|}\right)\\
    &= (1-\beta)\tau|X_1||X_2|
    \end{align*}
    where we used Parseval to calculate $\frac{1}{|H|}\sum_{\gamma\neq 0}|\hat{1}_{X_1}(\gamma)|^2=|X_1|-\hat{1}_{X_1}(0)^2/|H|=|X_1|-|X_1|^2/|H|=|X_1||X_2|/|H|$ as $|X_2|=|H|-|X_1|$ since $X_1,X_2$ partition $H$. This establishes \eqref{eq:specgapeqFn} and hence completes the proof. \end{proof}
It is now a simple matter to prove our result over $\Fn$ due to the nature of the Fourier transform in this group. Namely, the characters $\gamma\in \hat{\mathbb{F}}_2^n$ are precisely those functions of the form $\gamma_\xi:x\in \Fn\mapsto (-1)^{\langle x,\xi\rangle}$ for some vector $\xi\in \mathbb{F}_2^n$, where $\langle x,\xi\rangle={\sum_{j=1}^nx_j\xi_j}$ is the standard dot product in $\Fn$. Hence, if we write $$\langle \gamma\rangle^\perp=\{x\in \mathbb{F}_2^n:\gamma(x)=1\}=\{x\in \mathbb{F}_2^n: \langle x,\xi\rangle=0\}$$ for the codimension one subspace defined by $\gamma$, then for any subset $T\subset \mathbb{F}_2^n$, the Fourier transform of $T$ at $\gamma$ is simply given by $$\hat{1}_T(\gamma)=\sum_{x\in T}\gamma(x)=\sum_{x\in T}(-1)^{\langle x,\xi\rangle}=|T\cap \langle \gamma\rangle^\perp|-|T\cap(x_0+\langle\gamma\rangle^\perp)|, $$
where $x_0+\langle \gamma\rangle^\perp$ is the non-trivial coset of $\langle\gamma\rangle^\perp$ in $\mathbb{F}_2^n$. In particular, if $T$ has \textbf{no} spectral gap (in the sense of Lemma \ref{lem:spectralgapnosparsecutFn}), one immediately sees that most of $T$ is contained in the proper subspace $\langle \gamma\rangle^\perp$ of $\Fn$.
\begin{proof}[Proof of \Cref{cor:robexpander}] 
Let $\varepsilon\in(0,1/2)$ be given. Let $S\subset\Fn$ and we define $\sigma=|S|/N$ and $\delta=\varepsilon\sigma/2$. We proceed by a basic density increment argument, starting with $S_0=S$ and $H_0=\Fn$. We will iteratively construct subgroups $H_j< H_{j-1}$ and sets $S_j:=S_{j-1}\cap H_j$ satisfying for all $j$ that:
\begin{equation}
    |S_{j+1}|\geq (1-\frac{\varepsilon}{2^{j+1}})|S_{j}|.\label{eq:S_jlargeFn}
\end{equation}
Suppose now that we have constructed $H_j<H_{j-1}<\dots<H_0$ and $S_i=S\cap H_i$, for $i\leq j$, satisfying \eqref{eq:S_jlargeFn}. Then we certainly have 
\begin{equation}\label{eq:item(i)fine}
|S_j|\geq |S|\prod_{i=0}^\infty(1-\frac{\varepsilon}{2^{i+1}})\geq (1-\varepsilon)|S|.\end{equation}
So item (1) from the conclusion of \Cref{cor:robexpander} is satisfied for all $S_j$. If there is no $\delta$-sparse cut in $\Cayley_{H_j}(S_j)$, then item (2) is also satisfied and we are done. Else, by \Cref{lem:spectralgapnosparsecutFn} there is a non-trivial character $\gamma\in\hat{H}_j$ satisfying $$\hat{1}_{S_j}(\gamma)\geq \left(1-\delta\frac{|H_j|}{|S_j|}\right)|S_j|\geq \left(1-\frac{\delta}{2^{j-1}\sigma}\right)|S_j|$$
since we noted that $S_j$ has size at least $(1-\varepsilon)|S|\geq |S|/2$ as $\varepsilon<1/2$, and hence $S_j$ has density at least $2^{j-1}\sigma$ in $H_j$ because $H_j$ is a subgroup of $H_0$ of codimension $j$ (note that at each stage $i$ we find $H_i$ which is a proper subgroup of $H_{i-1}$). Now define $H_{j+1}=\langle \gamma\rangle^\perp$ which is a subspace of $H_j$ of codimension 1. As $\hat{1}_{S_j}(\gamma)=|S_j\cap H_{j+1}|-|S_j\cap(\text{non-trivial coset of }H_{j+1})|$, we get from the bound on the Fourier coefficient at $\gamma$ from above that $S_{j+1}=S_j\cap H_{j+1}$ has size  $$|S_{j+1}|\geq  \left(1-\frac{\delta}{2^{j}\sigma}\right)|S_j|\geq \left(1-\frac{\varepsilon}{2^{j+1}}\right)|S_j|,$$ as $\delta=\varepsilon\sigma/2$. Hence, we have shown that if $\Cayley_{H_j}(S_j)$ has a $\delta$-sparse cut, then we can continue and find a proper subgroup $H_{j+1}<H_j$ such that the set $S_{j+1}=S\cap H_{j+1}$ still satisfies \eqref{eq:S_jlargeFn}.

\smallskip

Observe that the process must trivially halt after a finite number of steps, since there is no infinite chain of subgroups $H_j$ in the finite group $\Fn$. In fact, since each $S_j$ has density at least $2^{j-1}\sigma$ in $H_j$, the process terminates after $O(\log 1/\sigma)$ steps. The final set $S_j$ in this process then has the property that $\Cayley_{H_j}(S_j)$ has no $\delta$-sparse cuts, and moreover it satisfies \eqref{eq:item(i)fine}, so $H_j$ and $S_j=S\cap H_j$ are the desired sets from the conclusion of the lemma. 
\end{proof}
\subsection{The general case.} \label{subsec:general_case}
Next, we prove the result in Theorem \ref{prop:RobustCayley-generalnonAbelian} in full generality, i.e. for a general finite group. We emphasise again that this general result is needed to establish Theorem \ref{thm:densecase-intro}, but that it is not required for our resolution of the rearrangement problem over $\Fn$ (Theorem \ref{thm:mainthm}). We begin by defining the correct notion, at least for our application, of a spectral gap for subsets of a general (not necessarily abelian) group. Note that for a subset $T$ of a nonabelian group $G$, its Fourier coefficients $\hat{1}_T(\rho)$ are matrices, rather than scalars as is the case when $G$ is abelian (such as $G=\Fn$ in the previous subsection). \begin{definition}[Spectral gap]\label{def:specgap}
    Let $H$ be a finite possibly nonabelian group, and let $T\subset H$. We say that $T$ has a \emph{$\beta$-spectral gap} if for every non-trivial irreducible representation $\rho\in\hat{H}$ and every unit vector $v\in V_\rho$ the following holds: $$\Re\,\langle \hat{1}_T(\rho)v,v\rangle_{V_\rho}\leq (1-\beta)|T|.$$
\end{definition}
\begin{remark}
    \normalfont We note that this definition of the spectral gap is equivalent to the condition that all eigenvalues of the matrices $\hat{1}_T(\rho)$ have real part at most $(1-\beta)|T|$, for all non-trivial irreducible representations $\rho$. In particular, recalling from \Cref{sec:notation} that the adjacency matrix of the directed graph $\Cayley_H(T)$ has precisely the same eigenvalues as the matrices $\hat{1}_T(\rho)$ as $\rho$ ranges over $\hat{H}$, we further note $T$ has a $\beta$-spectral gap if and only if all non-trivial eigenvalues of the adjacency matrix of $\Cayley_H(T)$ have real part at most $(1-\beta)|T|$. 
\end{remark}
Intuitively speaking, $T$ has no (or only a small) spectral gap if there is some non-trivial irreducible representation $\rho$ and some vector $v$ such that $\hat{1}_T(\rho)v\approx |T|v$. We also remark that $\hat{1}_T(\rho)=\sum_{t\in T}\rho(t)$ is a sum of $|T|$ unitary matrices, and hence we always have the trivial bound $\Re\, \langle\hat{1}_T(\rho)v,v\rangle_{V_\rho}\leq \lVert \hat{1}_T(\rho)v\rVert_{V_\rho}\leq |T|$ for unit vectors $v$. The following lemma generalises Lemma \ref{lem:spectralgapnosparsecutFn}, showing that if a subset $T$ of a finite group $H$ has a spectral gap, then $\Cayley_H(T)$ has no sparse cut. In particular, it immediately shows that Corollary \ref{cor:nosparsecutreg} follows from Theorem \ref{prop:RobustCayley-generalnonAbelian}.

\begin{lemma}\label{lem:spectralgapnosparsecutnonAbelian}
    Let $H$ be a finite group, and suppose that $T\subset H$ has a $\beta$-spectral gap in the sense of Definition \ref{def:specgap}: $$\sup_{\rho\in\hat{H}:\rho\neq \triv}\sup_{\substack{v\in V_\rho\\ \lVert v\rVert_{V_\rho}=1}}\Re\,\langle \hat{1}_T(\rho)v,v\rangle_{V_\rho}\leq (1-\beta)|T|.$$ Then $\Cayley_H(T)$ has no $\beta\tau$-sparse cuts, where $\tau=|T|/|H|$.
\end{lemma}

\begin{proof}[Proof of \Cref{lem:spectralgapnosparsecutnonAbelian}]
    Let $H=X_1\sqcup X_2$ be a partition of $H$, so that our goal is to show that $\#\{(x_1,x_2)\in X_1\times X_2: x_1^{-1}x_2\in T\}\geqslant \beta\tau |X_1||X_2|$. By the formula for the Fourier transform of a convolution and Fourier inversion, we can write
    \begin{align*}
    \#\{(x_1,x_2)\in X_1\times X_2: x_1^{-1}x_2\in T\}&=1_{X_1}*1_T*1_{X_2^{-1}}(\id_H)=\frac{1}{|H|}\sum_{\rho\in\hat{H}}d_\rho\tr\left(\hat{1}_{X_1}(\rho)\hat{1}_T(\rho)\overline{\hat{1}_{X_2}(\rho)}^T\right),    
    \end{align*}
    where we used that all irreducible representations are unitary to note that $\hat{1}_{X_2^{-1}}(\rho)=\sum_{x_2\in X_2}\rho(x_2)^{-1}=\sum_{x_2\in X_2}\overline{\rho(x_2)}^T=\overline{\hat{1}_{X_2}(\rho)}^T$. Since the contribution from the trivial representation to the right hand side is $|X_1||X_2||T|/|H|=\tau|X_1||X_2|$, and by using that the trace is invariant under cyclic shifts, it suffices to show that
    \begin{align}\label{eq:specgapeqnonAbelian}
        \Re\, \frac{-1}{|H|}\sum_{\rho\in\hat{H}:\rho\neq \triv}d_\rho\tr(\overline{\hat{1}_{X_2}(\rho)}^T\hat{1}_{X_1}(\rho)\hat{1}_T(\rho))\leq (1-\beta)\tau|X_1||X_2|,
    \end{align}
    as plugging this into the first equation would show precisely that
    $\#\{(x_1,x_2)\in X_1\times X_2: x_1^{-1}x_2\in T\}\geq \beta\tau|X_1||X_2|$.
    Note that as $X_1,X_2$ partition $H$, we have that $$\hat{1}_{X_1}+\hat{1}_{X_2}=\hat{1}_{H}=\begin{cases}
        |H|, &\text{if $\rho=\triv$}\\
        0, &\text{otherwise.}
    \end{cases}$$
    So $\hat{1}_{X_2}(\rho)=-\hat{1}_{X_1}(\rho)$ at all non-trivial representations $\rho$. Hence,
    \begin{align}\label{eq:X_1only}
    \Re\, \frac{-1}{|H|}\sum_{\rho\in\hat{H}:\rho\neq \triv}d_\rho\tr(\overline{\hat{1}_{X_2}(\rho)}^T\hat{1}_{X_1}(\rho)\hat{1}_T(\rho))= \frac{1}{|H|}\sum_{\rho\in\hat{H}:\rho\neq \triv}d_\rho\Re\tr (\overline{\hat{1}_{X_1}(\rho)}^T\hat{1}_{X_1}(\rho)\hat{1}_T(\rho)).
    \end{align}
    The matrix $\overline{\hat{1}_{X_1}(\rho)}^T\hat{1}_{X_1}(\rho)$ is conjugate symmetric (so self-adjoint with respect to $\langle \cdot,\cdot\rangle_{V_\rho}$) and positive semi-definite, so there is an orthonormal basis of vectors $v_1,v_2,\dots,v_{d_\rho}$ of $V_\rho$ which are eigenvectors, and with real non-negative eigenvalues $\lambda_1,\lambda_2,\dots,\lambda_{d_\rho}\geqslant 0$ whose sum is equal to $\tr(\hat{1}_{X_1}(\rho)\overline{\hat{1}_{X_1}(\rho)}^T)$. By noting that for any linear map $A:V_\rho\to V_\rho$ and any orthonormal basis $w_j$ we have that $\tr(A)=\sum_j \langle Aw_j, w_j\rangle_{V_\rho}$, we get for any non-trivial irreducible representation $\rho$ that \begin{align*}
\Re\,\tr(\overline{\hat{1}_{X_1}(\rho)}^T\hat{1}_{X_1}(\rho)\hat{1}_T(\rho))&=\Re\,\sum_{j=1}^{d_\rho}\langle \overline{\hat{1}_{X_1}(\rho)}^T\hat{1}_{X_1}(\rho)\hat{1}_T(\rho)v_j,v_j\rangle_{V_\rho}\\&=\Re \,\sum_{j=1}^{d_\rho} \langle \hat{1}_T(\rho)v_j,\overline{\hat{1}_{X_1}(\rho)}^T\hat{1}_{X_1}(\rho)v_j\rangle_{V_\rho}\\
&=\sum_{j=1}^{d_\rho}\lambda_j\Re\,\langle\hat{1}_T(\rho)v_j,v_j\rangle_{V_\rho},
    \end{align*} 
where we used that $\overline{\hat{1}_{X_1}(\rho)}^T\hat{1}_{X_1}(\rho)$ is self-adjoint in the second line. The spectral gap assumption states that for any unit vector $v$ we have an upper bound $\Re\, \langle \hat{1}_T(\rho)v,v\rangle_{V_\rho}\leqslant (1-\beta)|T|$. Using this spectral gap bound in the equation above, we get the following upper bound for every non-trivial irreducible representation $\rho$: \begin{align*}\Re\, \tr(\overline{{1}_{X_1}(\rho)}^T\hat{1}_{X_1}(\rho)\hat{1}_T(\rho))&\leqslant
(1-\beta)|T|\sum_{j=1}^{d_\rho} \lambda_j\\
&=(1-\beta)|T|\tr(\overline{\hat{1}_{X_1}(\rho)}^T\hat{1}_{X_1}(\rho)),\end{align*} as $\tr(\overline{\hat{1}_{X_1}(\rho)}^T\hat{1}_{X_1}(\rho))=\sum_{j=1}^{d_\rho} \lambda_j$. Finally, we can plug these trace bounds in the right hand side of \eqref{eq:X_1only} and bound this by
    \begin{align*}
        \frac{(1-\beta)|T|}{|H|}\sum_{\rho\neq 0}d_\rho\tr(\overline{\hat{1}_{X_1}(\rho)}^T\hat{1}_{X_1}(\rho))&=(1-\beta)|T|\left(|X_1|-\frac{|X_1|^2}{|H|}\right)\\
    &= (1-\beta)\tau|X_1||X_2|
    \end{align*}
    where we used Parseval to calculate $\frac{1}{|H|}\sum_{\rho\neq \triv}d_\rho\tr(\overline{\hat{1}_{X_1}(\gamma)}^T\hat{1}_{X_1}(\rho))=|X_1|-\hat{1}_{X_1}(\triv)^2/|H|=|X_1|-|X_1|^2/|H|=|X_1||X_2|/|H|$ as $|X_2|=|H|-|X_1|$ since $X_1,X_2$ partition $H$. This establishes \eqref{eq:specgapeqnonAbelian} and hence completes the proof of the lemma.
    \end{proof}
Recall that over $\Fn$, a lemma of the type that we just proved could immediately be combined with a density increment argument to conclude Lemma \ref{cor:robexpander}, basically because a subset $T$ of $\Fn$ having no spectral gap is trivially equivalent to most of $T$ being contained in a proper subgroup. Such a statement is more delicate in general groups, and in fact only true in a weaker sense. The next auxiliary lemma is a result of this type that is true in general groups. It states that if a set $T\subset G$ has no $\beta$-spectral gap for some $\beta$ which is sufficiently small in terms of the density of $T$ in $G$, then one can again conclude that most of $T$ lies in a proper subgroup. 
\begin{lemma}\label{lem:gammaordernonAbelian}
    Let $H$ be a finite group, and suppose that $T\subset H$ is a subset for which there exists a non-trivial irreducible representation $\rho$ and a unit vector $v\in V_\rho$ such that $$\Re\,\langle\hat{1}_T(\rho)v,v\rangle\geqslant (1-\beta)|T|.$$ Let $\tau=|T|/|H|$ and assume that $\beta\leq\tau^2/1000$ (say). Then there is a proper subgroup $H'$ of $H$ which contains at least $|T\cap H'|\geq (1-50\beta/\tau^2)|T|$
    of the elements of $T$.
\end{lemma}
\begin{proof}
Suppose that there exists a non-trivial irreducible representation $\rho$ and a unit vector $v\in V_\rho$ such that $\Re\,\langle\hat{1}_T(\rho)v,v\rangle_{V_\rho}\geqslant (1-\beta)|T|$. Throughout this proof, $\rho$ will be fixed and hence we will simply write $\lVert \cdot \rVert$ and $\langle\cdot,\cdot\rangle$ for $\lVert \cdot\rVert_{V_\rho},\langle \cdot,\cdot\rangle_{V_\rho}$. We will consider the Bohr sets $$B(\eta):=\{h\in H:\lVert \rho(h)v-v\rVert\leq \eta\}$$
for $\eta\in[0,2]$ (these sets also depend on $v$, but we consider $v$ to be fixed in this proof). Recall that $\Re\,\langle \rho(t)v,v\rangle\leqslant 1$ for any $t\in T$, as $\rho(t)$ is unitary. From the assumption that $\Re\, \langle\hat{1}_T(\rho)v,v\rangle\geq (1-\beta)|T|$ we thus deduce that the set $$T_m:=\{t\in T: \Re\, \langle\rho(t)v,v\rangle\geq 1-m\beta\}$$ satisfies $$(1-\beta)|T|\leq \Re\,\langle \hat{1}_T(\rho)v,v\rangle=\sum_{t\in T}\Re\,\langle\rho(t)v,v\rangle\leq |T_m|+(1-m\beta)|T\setminus T_m|=|T|-m\beta|T\setminus T_m|$$ and hence has size $|T_m|\geq (1-1/m)|T|$. $T_m$ is not (necessarily) a Bohr set, but we show that it is efficiently contained in a Bohr set. Since $v$ and hence $\rho(t)v$ are unit vectors, we have for $t\in T_m$, i.e. for $t$ satisfying $\Re\langle \rho(t)v,v\rangle\geqslant (1- m\beta)$, that $\lVert \rho(t)v-v\rVert^2 =2-2\Re\langle \rho(t)v,v\rangle\leqslant 2m\beta$. We conclude that the set $$T^{(m)}:=T\cap B(\sqrt{2m\beta})=\{t\in T: \lVert\rho(t)v-v\rVert\leqslant \sqrt{2m\beta}\}$$ contains $T_m$ and thus has size $|T^{(m)}|\geq (1-1/m)|T|$. In other words, we have shown that $T^{(m)}$ contains `most' of $T$ and is contained in a Bohr set $B(\sqrt{2m\beta})$ of rather short width. 

\par We make the following basic observation about the Bohr sets $B(\eta)$: if $x\in B(\eta)$ then $xB(\eta')\subset B(\eta+\eta')$. The proof of this is simply observing that if $\lVert \rho(x)v-v\rVert \leqslant \eta$ and $\rVert \rho(y)v-v\rVert \leqslant \eta'$, then 
\begin{align*}
    \lVert \rho(xy)v-v\rVert&=\rVert\rho(x)\rho(y)v-v\rVert\\
    &= \lVert \rho(x)\rho(y)v-\rho(x)v+\rho(x)v-v\rVert\\
    &\leqslant \lVert \rho(y)v-v\rVert+\lVert \rho(x)v-v\rVert\leqslant \eta+\eta',
\end{align*} by the triangle inequality and as $\rho(x)$ is unitary.

\par We now choose $m = \tau^2/(50\beta)$ and we will write $\eta:= \tau/5$, so we have shown that $T^{(m)}:=T\cap B(\sqrt{2m\beta})=T\cap B(\eta)$ contains at least $(1-1/m)|T|\geqslant (1-50\beta/\tau^2)|T|$ elements of $T$. In particular, as we are assuming that $\beta\leq\tau^2/1000$, we certainly have that $|T^{(m)}|\geqslant 0.9|T|$. Note also that $T^{(m)}$ satisfies the size requirement from the conclusion of the lemma, so to complete the proof of this lemma, it only remains to show that $T^{(m)}$ is contained in a proper subgroup of $H$. It thus suffices to show that $B(\eta)$ is contained in a proper subgroup, and to do this we will establish the following claim.
\begin{claim}
    There exists some integer $k<4/\tau$ such that $B(k\eta)\setminus B((k-1)\eta)=\emptyset$, where $\eta=\tau/5$.
\end{claim} 
First let us see how, assuming this claim, we can easily deduce the desired conclusion that $B(\tau/5)=B(\eta)$ is contained in a proper subgroup of $H$. Indeed, we have the basic fact that $xB((k-1)\eta)\subset B(k\eta)$ for any $x\in B(\eta)$, and hence the claim that $B(k\eta)\setminus B((k-1)\eta)=\emptyset$ implies that $B(\eta)\cdot X\subset X$ where $X=B((k-1)\eta)$. Iterating this, we see that $B(\eta)^j\cdot X\subset X$ for all $j$ and as $X=B((k-1)\eta)$ clearly contains the identity element, $X$ must therefore contain the subgroup generated by $B(\eta)$. Finally, to see that this subgroup is proper we may simply note that $X$ is not the whole of $H$ since $$X=B((k-1)\eta)\subset B(k\eta)\subset B((4/\tau)\tau/5)\subset B(4/5)=\{h\in H:\lVert \rho(h)v-v\rVert\leqslant 4/5\}$$ cannot contain the whole of $H$: recall the orthogonality relation $\sum_{h\in H}\rho(h)=0$ which holds as $\rho$ is a non-trivial irreducible representation, so $\sum_{h\in H}\rho(h)v=0$ and hence $$|G|=\left\lVert |G|v-\sum_{g\in G}\rho(g)v\right\rVert \leqslant \frac45|B(4/5)|+2|G\setminus B(4/5)|$$ which shows that $|G\setminus B(4/5)|\geqslant |G|/10$.

\par

It only remains to prove the claim. Suppose for a contradiction that it is not true, then for each integer $j\leqslant 4/\tau$ we can find an element $h_j\in B(j\eta)\setminus B((j-1)\eta)$. Consider the elements $h_1,h_4,\dots,h_{3r+1}$ with indices which are $1\pmod 3$ up to $4/\tau$. We claim that the sets $h_1B(\eta),h_4B(\eta),\dots,h_{3r+1}B(\eta)$ are pairwise disjoint subsets of $H$. Indeed, pick $x\in h_{3i+1}B(\eta)$ and $y\in h_{3j+1}B(\eta)$ for some $i>j$. So $x=h_{3i+1}x_0$ for some $x_0\in B(\eta)$. Then we calculate
\begin{align*}
\lVert \rho(x)v-\rho(y)v\rVert&\geqslant \lVert \rho(x)v-v\rVert-\lVert \rho(y)v-v\rVert\\
&\geq \lVert (\rho(h_{3i+1})v-v)+\rho(h_{3i+1})(\rho(x_0)v-v)\rVert- (3j+2)\eta, 
\end{align*} where we used that $y\in h_{3j+1}B(\eta)\subset B((3j+2)\eta)$. Hence, using that $\rho(h_{3i+1})$ is unitary (so distance preserving) and that $x_0\in B(\eta)$, we get
\begin{align*}
    \lVert \rho(x)v-\rho(y)v\rVert &\geqslant \lVert \rho(h_{3i+1})v-v\rVert - \lVert \rho(x_0)v-v\rVert-(3j+2)\eta\\
    &> 3i\eta - (3j+3)\eta
\end{align*} where the strictness in the final inequality holds as $h_{3i+1}\in B((3i+1)\eta)\setminus B(3i\eta)$. As $i>j$ this implies that $x\neq y$ so that indeed the sets $h_{3i+1}B(\eta),h_{3j+1}B(\eta)$ are disjoint as we claimed. Finally, we note that this gives us the required contradiction since we showed above that $T^{(m)}\subset B(\tau/5)=B(\eta)$ and that $T^{(m)}$ contains at least $(1-50\beta/\tau^2)|T|\geq0.9|T|=0.9\tau |H|$ elements, by the assumption of the lemma that $\beta\leq\tau^2/1000$. Hence the sets $h_{3j+1}B(\tau/5)$ for $1\leqslant 3j+1\leqslant 4/\tau$ would give us $4/(3\tau)$ disjoints sets of size at least $0.9\tau|H|$ inside $H$. This is of course absurd, and we deduce that there must be some $j\leqslant 4/\tau$ for which $B(j\eta)\setminus B((j-1)\eta)=\emptyset$, proving the claim. 
\end{proof}

We can now prove Theorem \ref{prop:RobustCayley-generalnonAbelian} by repeatedly applying Lemma \ref{lem:gammaordernonAbelian}.

\begin{proof}
Let $\varepsilon\in(0,1/2)$ be given. Let $S\subset G$ and we define $\sigma=|S|/|G|$ and $\eta=\varepsilon\sigma^2/1000$. We proceed by a density increment argument, starting with $S_0=S$ and $H_0=G$. We will iteratively construct subgroups $H_j< H_{j-1}$ and sets $S_j:=S_{j-1}\cap H_j$ satisfying for all $j$ that:
\begin{equation}
    |S_{j+1}|\geq \left(1-\frac{\varepsilon}{2^{j+1}}\right)|S_{j}|.\label{eq:S_jlargenonAbelian}
\end{equation}
Suppose now that we have constructed $H_j<H_{j-1}<\dots<H_0$ and $S_i=S\cap H_i$, for $i\leq j$, satisfying \eqref{eq:S_jlargenonAbelian}. Then we certainly have 
\begin{equation}\label{eq:item(i)finegeneralnonAbelian}
|S_j|\geq |S|\prod_{i=0}^\infty\left(1-\frac{\varepsilon}{2^{i+1}}\right)\geq (1-\varepsilon)|S|.\end{equation}
So item (1) is satisfied for all $S_j$. Hence, either item (2) is also satisfied in which case the desired conclusion from Theorem \ref{prop:RobustCayley-generalnonAbelian} holds, or the adjacency matrix of $\Cayley_{H_j}(S_j)$ has a non-trivial eigenvalue with real part at least $(1-\eta)|S_j|$, where $\eta=\varepsilon\sigma^2/1000$. Following the remark after Definition \ref{def:specgap}, this is equivalent to $S_j\subset H_j$ having no $\eta$-spectral gap meaning that there are a non-trivial irreducible representation $\rho\in\hat{H}_j$ and a unit vector $v\in V_\rho$ satisfying $$\Re\,\langle \hat{1}_{S_j}(\rho)v,v\rangle_{V_\rho}\geqslant \left(1-\eta\right)|S_j|.$$
By \eqref{eq:item(i)finegeneralnonAbelian}, we have that $S_j$ has size at least $(1-\varepsilon)|S|\geq |S|/2$ as $\varepsilon<1/2$, and hence $S_j$ has density at least $2^{j-1}\sigma$ in $|H_j|$ because $H_j$ is a subgroup of $H_0$ of index at least $2^j$ (note that at each stage $i$ we find $H_i$ which is a proper subgroup of $H_{i-1}$). In particular, as $\eta=\varepsilon\sigma^2/1000$, we see that $\eta\leq(|S_j|/|H_j|)^2/1000$ so that the assumption of Lemma \ref{lem:gammaordernonAbelian} is satisfied. This lemma concludes that there is a proper subgroup $H_{j+1}<H_j$ such that $S_{j+1}=S\cap H_{j+1}$ has size at least $$|S_{j+1}|\geq  \left(1-\eta\cdot\frac{50}{(|S_j|/|H_j|)^2}\right)|S_j|\geq\left(1-\eta\cdot\frac{50}{4^{j-1}\sigma^2}\right)|S_j|\geq \left(1-\frac{\varepsilon}{2^{j+1}}\right)|S_j|,$$ as $\eta=\varepsilon\sigma^2/1000$. Hence, we have shown that if $\Cayley_{H_j}(S_j)$ does not satisfy condition (2), then we can continue and find a proper subgroup $H_{j+1}<H_j$ such that $S_{j+1}=S\cap H_{j+1}$ still satisfies \eqref{eq:S_jlargenonAbelian}.

\smallskip

Observe that the process must trivially halt after a finite number of steps, since there is no infinite chain of subgroups $H_j$ in the finite group $G$. The final set $S_j$ in this process then has the property that $\Cayley_{H_j}(S_j)$ satisfies condition (2), and moreover it satisfies \eqref{eq:item(i)finegeneralnonAbelian}, so $H_j$ and $S_j=S\cap H_j$ are the desired sets from the conclusion of Theorem \ref{prop:RobustCayley-generalnonAbelian}.
\end{proof}

\section{A flexible 99\% result}\label{sec:veryflex}
The main result of this section is Lemma~\ref{lem:asymptoticinrandom-dense}, a flexible asymptotic statement about finding rainbow paths in dense robust expander digraphs. This Lemma~\ref{lem:asymptoticinrandom-dense} will play a crucial role at several stages in the remainder of the paper. 

As we mentioned in the introduction, it is shown in ~\cite{towards-graham} that if $\Cayley_G(S)$ is a robust expander, then it contains a rainbow path of length $(1-o(1))|S|$. This result is insufficient for our applications because we will need to obtain the same conclusion even if we restrict to random vertex subsets of $\Cayley_G(S)$ and forbid a small number of colours from $S$. Due to this additional flexibility requirement, our proof of Lemma~\ref{lem:asymptoticinrandom-dense} diverges significantly from the approach in~\cite{towards-graham}.

\subsection{Tools}
In this section we introduce some notation and previous results. 

We start with the following lemma (\cite[Lemma 3.8]{muyesser2022random}), which combines Thomason's jumbledness criterion with the R\"odl nibble. 
The power of the lemma is that the set $C'$ can be chosen completely arbitrarily \emph{after} the random sets $A',B'$ are revealed.  
We say that a tripartite $3$-uniform hypergraph is $(\gamma,p,n)$-\emph{typical} if each partite set has $(1 \pm \gamma)n$ vertices;  each vertex has degree $(1 \pm \gamma)pn$; and for each pair of vertices $u,v$ in the same partite set, there are $(1 \pm \gamma)p^2n$ vertices $w$ in each other partite set such that $(u,w,x), (v,w,y)$ are edges for some $x,y$ in the third partite set. A hypergraph is linear if through every two vertices there exists at most one edge.

\begin{lemma}[\cite{muyesser2022random}, Lemma 3.8]\label{Lemma_2_random_1_deterministic}
Let $H=(A,B,C)$ be a $(0,1,n)$-typical tripartite linear hypergraph, and let $p\geq n^{-1/600}$. Let $A'\subseteq A$ and $B' \subseteq B$ be (not necessarily independent) $p$-random subsets. Then with probability at least $1-n^{-2}$, the following holds: For any $C'\subseteq C$ of size $(1\pm n^{-0.2})pn$, there is a matching covering all but $2n^{1-1/500}$ vertices in $A'\cup B'\cup C'$.
\end{lemma}

We will always use the above lemma in the form of the following corollary, which picks out the special case of the multiplication hypergraph of a group $G$ (which is always $(0,1,n)$-typical).

\begin{corollary}
\label{Cor_2_random_1_deterministic}
Let $G$ be a group on $n$ elements, and let $p\geq n^{-1/600}$.
Let $A,B\subseteq G$ be disjoint $p$-random subsets. Then with probability at least $1-n^{-2}$, the following holds: For any $C \subseteq G$ of size $(1\pm n^{-0.2})pn$, there is a rainbow matching in $\Cayley_G(C)$ from $A$ to $B$ covering all but $2n^{1-1/500}$ vertices in $A \cup B$ and using all but at most $2n^{1-1/500}$ colours from $C$.
\end{corollary}

We next introduce the notion of robust expansion (following \cite{robust-expanders}). As we will see shortly, robust expansion is implied\footnote{The two notions are in fact equivalent up to a constant factor loss in parameters.} by the absence of sparse cuts. We will use robust expansion only through our invocation of \Cref{lem:connecting lemma} below (from \cite{towards-graham}); the notion will not otherwise figure in the paper.

\begin{definition}[Robust expansion]\label{def:robustexpander}
    Let $G$ be a directed graph on $n$ vertices. For $U \subseteq V(G)$ and $\nu > 0$, the \emph{$\nu$-robust out-neighbourhood} of $U$ in $G$ is the set
    \[RN_{\nu,G}^+(U) := \{v \in V(G) : |N^-(v) \cap U| \geq \nu n\}.\]
    We say that $G$ is a \emph{robust $(\nu, \tau)$-out-expander} if every $U \subseteq V(G)$ with $\tau n \leq |U| \leq (1-\tau)n$ satisfies
    \[|RN_{\nu, G}^+(U) \setminus U| \geq \nu n.\]
    Similarly, the \emph{$\nu$-robust in-neighbourhood} of $U$ in $G$ is the set
    \[RN_{\nu,G}^-(U) := \{v \in V(G) : |N^+(v) \cap U| \geq \nu n\},\]
    and we say that $G$ is a \emph{robust $(\nu, \tau)$-in-expander} if every $U \subseteq V(G)$ with $\tau n \leq |U| \leq (1-\tau)n$ satisfies
    \[|RN_{\nu, G}^-(U) \setminus U| \geq \nu n.\]
\end{definition}

We say that an undirected graph $G$ is a \emph{robust $(\nu, \tau)$-expander} if the directed graph obtained by replacing each edge with two directed edges (one in each direction) is a robust $(\nu, \tau)$-out-expander (or, equivalently, a robust $(\nu, \tau)$-in-expander).

The following elementary proposition shows that a graph with no sparse cuts, as in the definition following \Cref{prop:RobustCayley-generalnonAbelian}, is a robust expander.  After quoting \Cref{lem:connecting lemma} from \cite{towards-graham}, we will work with only the no-sparse-cuts property in the rest of this paper.

\begin{proposition}\label{prop:no-sparse-cuts-to-robust-expansion}
    Let $0\le \tau \le 1/2$ and $0 \leq \zeta \leq 1$.  Then any digraph $H$ with no $\zeta$-sparse cuts is a $(\zeta \tau/8,\tau)$-robust-out-expander. 
\end{proposition}
\begin{proof}
Set $n:=|H|$. Let $U$ be any subset of $V(H)$ of size $\tau n \le |U| \le (1-\tau)n$  Since $U \sqcup (H \setminus U)$ is not a $\zeta$-sparse cut,  there must be at least $\zeta \tau (1-\tau)n^2$ edges from $U$ to $H \setminus U$. The $\zeta \tau/8$-non-robust out-neighbourhood of $U$ can pick up at most $(\zeta \tau/8) n^2$ of these edges, so the $\zeta \tau/8$-robust out-neighbourhood of $U$ has size at least $(\zeta \tau(1-\tau)-\zeta \tau/8)n \geq (\zeta \tau/8)n$. 
\end{proof}

In a sufficiently dense robust expander, one can find short paths connecting any two given vertices, and one can moreover guarantee that all vertices and edge-colours in the connecting path come from specified random subsets. The following lemma makes this precise (here we we state only one of the several properties in the lemma from \cite{towards-graham}).  The proof consists of elementary applications of Chernoff's bound and an application of the definition of robust expansion. %

\begin{lemma}[\cite{towards-graham}, Lemma 4.3]\label{lem:connecting lemma}
Let $\nu, \tau, p \leq 1$ be positive constants. Let $G$ be a properly edge-coloured directed graph on $n$ vertices, where $p^3 \nu^2 n \geq 144\log n$. Suppose that $G$ is a robust $(\nu, \tau)$-out-expander,
    with $\delta^\pm(G) \geq  (\nu + \tau)n$. Let $V_0 \subseteq V(G), C_0 \subseteq C(G)$ be independent $p$-random subsets. Then with probability at least $1 - 5/n$, the following holds:

For any distinct vertices $u,v \in V(G)$, and for any vertex subset $V_1 \subseteq V_0$ and colour subset $C_1 \subseteq C_0$ with $|V_1|, |C_1| \leq (p^3 \nu/100) n$, there exists a rainbow directed path of length at most $\nu^{-1} + 1$ from $u$ to $v$ in $G$ whose internal vertices lie in $V_0 \setminus V_1$ and whose colours lie in $C_0 \setminus C_1$.
\end{lemma}

Iterative applications of this lemma yield the following corollary.

\begin{corollary}\label{cor:vxdisjointpaths}  Let $0<\nu, \tau, p \leq 1$. Let $G$ be a properly edge-coloured directed graph on $n$ vertices, where $p^3 \nu^2 n \geq 144\log n$.  Suppose that $G$ is a robust $(\nu,\tau)$-out-expander with $\delta^{\pm}(G) \ge (\nu+\tau)n$. Let $V_0 \subseteq V, C_0 \subseteq C(G)$ be independent $p$-random subsets. Then with probability at least $1-5/n$, the following holds:

For any collection $(v_i,w_i)_{i\in [k]}$ of $k\leq \frac{p^3 \nu^2}{300}n$ disjoint pairs of vertices, we can find a rainbow collection of vertex-disjoint paths $P_1,\ldots,P_k$ (meaning that the union of the $P_i$'s is rainbow), where each $P_i$ goes from $v_i$ to $w_i$, and the vertices of the $P_i$'s lie in $V_0$ and use colours from $C_0$.
\end{corollary}
\begin{proof}
With probability $1-5/n$, the conclusion of Lemma~\ref{lem:connecting lemma} holds for $V_0,C_0$.  We construct the paths $P_i$ one at a time.  Suppose we have already constructed $P_1, \ldots, P_\ell$ for some $\ell<k$.  Let $V_1$ denote the union of the internal vertices in $P_1, \ldots, P_\ell$, and let $C_1$ denote the set of colours in $P_1, \ldots, P_\ell$.  Notice that $$|C_1|,|V_1| \leq (\nu^{-1}+2)\ell\leq (\nu^{-1}+2)k \leq (p^3 \nu/100) n.$$  Then Lemma~\ref{lem:connecting lemma} with this choice of $V_1,C_1$ produces the desired path $P_{\ell+1}$ from $v_{\ell+1}$ to $w_{\ell+1}$.\end{proof}

\subsection{The $99\%$ lemma}
We have nearly arrived at the main lemma, which establishes a very flexible asymptotic result in the dense setting. This lemma allows us to find a rainbow path of length $(1-o(1))|S|$ inside a (large) random vertex subset of $\Cay{S}$ with high probability.  We can in fact guarantee a bit more: For $S_F$ contained in a random $S' \subseteq S$, we want to find a rainbow path in $\Cayley_{\Fn}({S\setminus S_F})$ of length $(1-o(1))|S\setminus S_F|$; our lemma guarantees that with high probability, the restriction of $\Cayley_{\Fn}({S\setminus S_F})$ to our random vertex set contains such a path for \emph{all} eligible choices of $S_F$ simultaneously.  This flexibility will be useful later in the argument, for instance when we want our $99\%$ path to avoid the absorbing structure that we set aside initially.

The statement of our lemma involves many different parameters, objects, and quantifiers.  To help the reader get their bearing, we gloss over some of the characters involved.  The main thrust of the lemma is that a nicely expanding Cayley graph with a generating set $S$ of size at least $8n^{1-1/9500}$ has a rainbow path which uses all but a few colours from $S$.
For our later applications, we will need to be able to impose further restrictions on this long rainbow path:

\begin{itemize}
    \item If $M$ is a randomly sampled vertex subset, then with high probability for any two vertices $u,v$ we can require the long rainbow path to start at $u$, end at $v$, and have all of its internal vertices lying in $M$.
    \item We require the path to avoid a small (adversarially-chosen) deterministic vertex set $J$.
    \item We also require the colour set of the path to avoid an adversarially-chosen subset $S_F$ of a randomly sampled subsets $S' \subseteq S$.
    \item Our path should use all but a small fraction of the colours in $S \setminus S_F$.
\end{itemize}

We now give the formal statement of our flexible $99\%$ lemma.
\begin{lemma}\label{lem:asymptoticinrandom-dense}
    Let $G$ be an $N$-element group. Let $8N^{-1/9500}\le  \zeta, \mu, \eps, q \le 1.$ 
    \begin{itemize}
        \item Let $S\subseteq G$ have  $|S| \ge \eps N$, and suppose that $\Cayley_G(S)$ has no $\zeta$-sparse cuts.
        \item Let $J \subseteq G$ have $|J| \le 2^{-28}q^3\mu^3 \eps^2 \zeta^2  N$. 
        \item Let $M\subseteq G$ be a $q$-random subset of $G$, with $q \ge (1+\mu)|S|/N.$
        \item Let $S'\subseteq S $ be a $q'$-random subset of $S$, with $q'\le 1-\mu q/4$.
    \end{itemize}
      Then with probability at least $1-7/N$,
    the following holds for every choice of $S_F\subseteq S'$ and every pair of distinct vertices $u,v \in G$: There exists a rainbow path from $u$ to $v$ in $\Cayley_G(S \setminus S_F)$, using all but at most $\mu q N$ colours of $S \setminus S_F$, such that all of the internal vertices of the path lie in $M \setminus J$. 
\end{lemma}\begin{proof}
Let $H:=\Cayley_G(S)$ and set $\tau:=\frac{3}{4}\varepsilon$.  Due to \Cref{prop:no-sparse-cuts-to-robust-expansion}, the no $\zeta$-sparse cuts hypothesis implies that $H$ is a $(\nu,\tau)$-robust out-expander for $\nu:=\zeta \tau/8$. 

Since $|J| \le \zeta \eps/32 \cdot N\le \zeta \tau/16 \cdot N$, the graph $H \setminus J$ is still a $(\nu/2,\tau)$-robust out-expander with minimum degree at least $|S|-|J| \ge \frac78\eps N$. 
Note that $\frac78\eps\ge \zeta \tau/16+\tau$, so $H\setminus J$ satisfies the minimum-degree requirement of \Cref{cor:vxdisjointpaths}.

Let $t:= \frac{2^{28}}{q^2\mu^3 \zeta^2\eps^2}$.  We now randomly partition $G$ (the vertex set of $H$) into sets $R,M_1,\ldots, M_t$ and a junk set by placing each vertex into $R$ with probability $\tilde p:= \mu q/4$, into each $M_i$ with probability $p:=(q-\tilde p)/t$, and into the junk set otherwise (independently for each vertex).  Hence $R \subseteq G$ is a $\tilde p$-random subset, each $M_i \subseteq G$ is a $p$-random subset, and
$$M:=R \cup M_1 \cup \ldots \cup M_t \subseteq G$$ is a $q$-random subset of $G$; of course the sets $R,M_1,\ldots, M_t$ are all disjoint.  (We discard the junk set.)
Our choice of $t$ guarantees that $$p \ge \frac{q}{2t} \ge q^3\mu^3\zeta^2 \eps^2/2^{29}\ge  N^{-1/600}.$$

We can now describe the plan for the proof, depicted schematically in \Cref{fig:flexible-proof}.  We will use \Cref{Cor_2_random_1_deterministic} to obtain an almost-complete rainbow matching between $M_i$ and $M_{i+1}$ for each $i$ (using a fresh set of colours for each new pair); this produces a large rainbow path forest with few components. We will then use \Cref{cor:vxdisjointpaths} to find rainbow paths in $R$ (depicted in gray in \Cref{fig:flexible-proof}) linking together the components of the path forest; this step will use colours from a reserved random set $C_R'$.

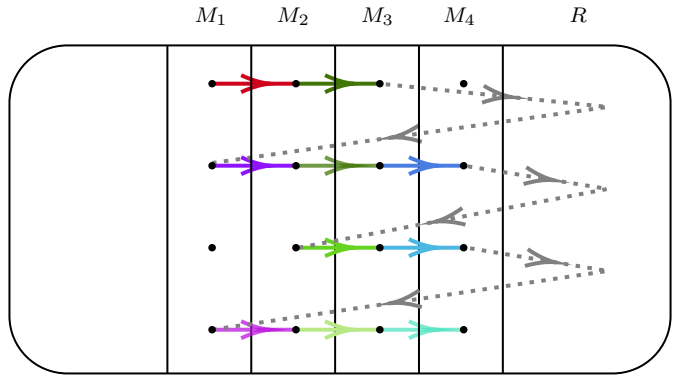
\begin{figure}[h]
\centering

\tikzset{every picture/.style={line width=0.75pt}} 

\begin{tikzpicture}[x=0.75pt,y=0.75pt,yscale=-1,xscale=1]

\draw [color={rgb, 255:red, 80; green, 227; blue, 194 }  ,draw opacity=0.73 ][line width=1.5]    (364.88,173) -- (406.74,173) ;
\draw [shift={(393.61,173)}, rotate = 180] [color={rgb, 255:red, 80; green, 227; blue, 194 }  ,draw opacity=0.73 ][line width=1.5]    (14.21,-4.28) .. controls (9.04,-1.82) and (4.3,-0.39) .. (0,0) .. controls (4.3,0.39) and (9.04,1.82) .. (14.21,4.28)   ;
\draw [color={rgb, 255:red, 65; green, 117; blue, 5 }  ,draw opacity=0.74 ][line width=1.5]    (323.02,90.5) -- (364.88,90.5) ;
\draw [shift={(351.75,90.5)}, rotate = 180] [color={rgb, 255:red, 65; green, 117; blue, 5 }  ,draw opacity=0.74 ][line width=1.5]    (14.21,-4.28) .. controls (9.04,-1.82) and (4.3,-0.39) .. (0,0) .. controls (4.3,0.39) and (9.04,1.82) .. (14.21,4.28)   ;
\draw [color={rgb, 255:red, 65; green, 117; blue, 5 }  ,draw opacity=1 ][line width=1.5]    (323.02,49.25) -- (364.88,49.25) ;
\draw [shift={(351.75,49.25)}, rotate = 180] [color={rgb, 255:red, 65; green, 117; blue, 5 }  ,draw opacity=1 ][line width=1.5]    (14.21,-4.28) .. controls (9.04,-1.82) and (4.3,-0.39) .. (0,0) .. controls (4.3,0.39) and (9.04,1.82) .. (14.21,4.28)   ;
\draw [color={rgb, 255:red, 144; green, 19; blue, 254 }  ,draw opacity=1 ][line width=1.5]    (281.16,90.5) -- (323.02,90.5) ;
\draw [shift={(309.89,90.5)}, rotate = 180] [color={rgb, 255:red, 144; green, 19; blue, 254 }  ,draw opacity=1 ][line width=1.5]    (14.21,-4.28) .. controls (9.04,-1.82) and (4.3,-0.39) .. (0,0) .. controls (4.3,0.39) and (9.04,1.82) .. (14.21,4.28)   ;
\draw [color={rgb, 255:red, 128; green, 128; blue, 128 }  ,draw opacity=1 ][line width=1.5]  [dash pattern={on 1.69pt off 2.76pt}]  (281.16,89.13) -- (478.6,60.94) -- (364.88,49.25) ;
\draw [shift={(371.17,76.27)}, rotate = 351.88] [color={rgb, 255:red, 128; green, 128; blue, 128 }  ,draw opacity=1 ][line width=1.5]    (14.21,-4.28) .. controls (9.04,-1.82) and (4.3,-0.39) .. (0,0) .. controls (4.3,0.39) and (9.04,1.82) .. (14.21,4.28)   ;
\draw [shift={(430.5,55.99)}, rotate = 185.87] [color={rgb, 255:red, 128; green, 128; blue, 128 }  ,draw opacity=1 ][line width=1.5]    (14.21,-4.28) .. controls (9.04,-1.82) and (4.3,-0.39) .. (0,0) .. controls (4.3,0.39) and (9.04,1.82) .. (14.21,4.28)   ;
\draw [color={rgb, 255:red, 208; green, 2; blue, 27 }  ,draw opacity=1 ][line width=1.5]    (281.16,49.25) -- (323.02,49.25) ;
\draw [shift={(309.89,49.25)}, rotate = 180] [color={rgb, 255:red, 208; green, 2; blue, 27 }  ,draw opacity=1 ][line width=1.5]    (14.21,-4.28) .. controls (9.04,-1.82) and (4.3,-0.39) .. (0,0) .. controls (4.3,0.39) and (9.04,1.82) .. (14.21,4.28)   ;
\draw   (180,58.88) .. controls (180,42.93) and (192.93,30) .. (208.88,30) -- (481.13,30) .. controls (497.07,30) and (510,42.93) .. (510,58.88) -- (510,166.13) .. controls (510,182.07) and (497.07,195) .. (481.13,195) -- (208.88,195) .. controls (192.93,195) and (180,182.07) .. (180,166.13) -- cycle ;
\draw    (258.84,30) -- (258.84,195) ;
\draw    (300.7,30) -- (300.7,195) ;
\draw    (342.56,30) -- (342.56,195) ;
\draw    (384.42,30) -- (384.42,195) ;
\draw    (426.28,30) -- (426.28,195) ;
\draw [color={rgb, 255:red, 189; green, 16; blue, 224 }  ,draw opacity=0.74 ][line width=1.5]    (281.16,173) -- (323.02,173) ;
\draw [shift={(309.89,173)}, rotate = 180] [color={rgb, 255:red, 189; green, 16; blue, 224 }  ,draw opacity=0.74 ][line width=1.5]    (14.21,-4.28) .. controls (9.04,-1.82) and (4.3,-0.39) .. (0,0) .. controls (4.3,0.39) and (9.04,1.82) .. (14.21,4.28)   ;
\draw [color={rgb, 255:red, 103; green, 211; blue, 33 }  ,draw opacity=1 ][line width=1.5]    (323.02,131.75) -- (364.88,131.75) ;
\draw [shift={(351.75,131.75)}, rotate = 180] [color={rgb, 255:red, 103; green, 211; blue, 33 }  ,draw opacity=1 ][line width=1.5]    (14.21,-4.28) .. controls (9.04,-1.82) and (4.3,-0.39) .. (0,0) .. controls (4.3,0.39) and (9.04,1.82) .. (14.21,4.28)   ;
\draw [color={rgb, 255:red, 184; green, 233; blue, 134 }  ,draw opacity=1 ][line width=1.5]    (323.02,173) -- (364.88,173) ;
\draw [shift={(351.75,173)}, rotate = 180] [color={rgb, 255:red, 184; green, 233; blue, 134 }  ,draw opacity=1 ][line width=1.5]    (14.21,-4.28) .. controls (9.04,-1.82) and (4.3,-0.39) .. (0,0) .. controls (4.3,0.39) and (9.04,1.82) .. (14.21,4.28)   ;
\draw [color={rgb, 255:red, 74; green, 184; blue, 226 }  ,draw opacity=1 ][line width=1.5]    (364.88,131.75) -- (406.74,131.75) ;
\draw [shift={(393.61,131.75)}, rotate = 180] [color={rgb, 255:red, 74; green, 184; blue, 226 }  ,draw opacity=1 ][line width=1.5]    (14.21,-4.28) .. controls (9.04,-1.82) and (4.3,-0.39) .. (0,0) .. controls (4.3,0.39) and (9.04,1.82) .. (14.21,4.28)   ;
\draw [color={rgb, 255:red, 128; green, 128; blue, 128 }  ,draw opacity=1 ][line width=1.5]  [dash pattern={on 1.69pt off 2.76pt}]  (324.42,131.75) -- (478.6,102.19) -- (406.74,90.5) ;
\draw [shift={(392.87,118.63)}, rotate = 349.15] [color={rgb, 255:red, 128; green, 128; blue, 128 }  ,draw opacity=1 ][line width=1.5]    (14.21,-4.28) .. controls (9.04,-1.82) and (4.3,-0.39) .. (0,0) .. controls (4.3,0.39) and (9.04,1.82) .. (14.21,4.28)   ;
\draw [shift={(451.36,97.76)}, rotate = 189.24] [color={rgb, 255:red, 128; green, 128; blue, 128 }  ,draw opacity=1 ][line width=1.5]    (14.21,-4.28) .. controls (9.04,-1.82) and (4.3,-0.39) .. (0,0) .. controls (4.3,0.39) and (9.04,1.82) .. (14.21,4.28)   ;
\draw [color={rgb, 255:red, 128; green, 128; blue, 128 }  ,draw opacity=1 ][line width=1.5]  [dash pattern={on 1.69pt off 2.76pt}]  (281.16,173) -- (478.6,143.44) -- (408.14,131.75) ;
\draw [shift={(371.18,159.52)}, rotate = 351.48] [color={rgb, 255:red, 128; green, 128; blue, 128 }  ,draw opacity=1 ][line width=1.5]    (14.21,-4.28) .. controls (9.04,-1.82) and (4.3,-0.39) .. (0,0) .. controls (4.3,0.39) and (9.04,1.82) .. (14.21,4.28)   ;
\draw [shift={(452.05,139.03)}, rotate = 189.42] [color={rgb, 255:red, 128; green, 128; blue, 128 }  ,draw opacity=1 ][line width=1.5]    (14.21,-4.28) .. controls (9.04,-1.82) and (4.3,-0.39) .. (0,0) .. controls (4.3,0.39) and (9.04,1.82) .. (14.21,4.28)   ;
\draw [color={rgb, 255:red, 74; green, 125; blue, 226 }  ,draw opacity=1 ][line width=1.5]    (364.88,90.5) -- (406.74,90.5) ;
\draw [shift={(393.61,90.5)}, rotate = 180] [color={rgb, 255:red, 74; green, 125; blue, 226 }  ,draw opacity=1 ][line width=1.5]    (14.21,-4.28) .. controls (9.04,-1.82) and (4.3,-0.39) .. (0,0) .. controls (4.3,0.39) and (9.04,1.82) .. (14.21,4.28)   ;
\draw  [fill={rgb, 255:red, 0; green, 0; blue, 0 }  ,fill opacity=1 ] (363.49,49.25) .. controls (363.49,50.01) and (364.11,50.63) .. (364.88,50.63) .. controls (365.65,50.63) and (366.28,50.01) .. (366.28,49.25) .. controls (366.28,48.49) and (365.65,47.88) .. (364.88,47.88) .. controls (364.11,47.88) and (363.49,48.49) .. (363.49,49.25) -- cycle ;
\draw  [fill={rgb, 255:red, 0; green, 0; blue, 0 }  ,fill opacity=1 ] (363.49,90.5) .. controls (363.49,91.26) and (364.11,91.88) .. (364.88,91.88) .. controls (365.65,91.88) and (366.28,91.26) .. (366.28,90.5) .. controls (366.28,89.74) and (365.65,89.13) .. (364.88,89.13) .. controls (364.11,89.13) and (363.49,89.74) .. (363.49,90.5) -- cycle ;
\draw  [fill={rgb, 255:red, 0; green, 0; blue, 0 }  ,fill opacity=1 ] (363.49,131.75) .. controls (363.49,132.51) and (364.11,133.13) .. (364.88,133.13) .. controls (365.65,133.13) and (366.28,132.51) .. (366.28,131.75) .. controls (366.28,130.99) and (365.65,130.38) .. (364.88,130.38) .. controls (364.11,130.38) and (363.49,130.99) .. (363.49,131.75) -- cycle ;
\draw  [fill={rgb, 255:red, 0; green, 0; blue, 0 }  ,fill opacity=1 ] (363.49,173) .. controls (363.49,173.76) and (364.11,174.38) .. (364.88,174.38) .. controls (365.65,174.38) and (366.28,173.76) .. (366.28,173) .. controls (366.28,172.24) and (365.65,171.63) .. (364.88,171.63) .. controls (364.11,171.63) and (363.49,172.24) .. (363.49,173) -- cycle ;

\draw  [fill={rgb, 255:red, 0; green, 0; blue, 0 }  ,fill opacity=1 ] (405.35,49.25) .. controls (405.35,50.01) and (405.97,50.63) .. (406.74,50.63) .. controls (407.51,50.63) and (408.14,50.01) .. (408.14,49.25) .. controls (408.14,48.49) and (407.51,47.88) .. (406.74,47.88) .. controls (405.97,47.88) and (405.35,48.49) .. (405.35,49.25) -- cycle ;
\draw  [fill={rgb, 255:red, 0; green, 0; blue, 0 }  ,fill opacity=1 ] (405.35,90.5) .. controls (405.35,91.26) and (405.97,91.88) .. (406.74,91.88) .. controls (407.51,91.88) and (408.14,91.26) .. (408.14,90.5) .. controls (408.14,89.74) and (407.51,89.13) .. (406.74,89.13) .. controls (405.97,89.13) and (405.35,89.74) .. (405.35,90.5) -- cycle ;
\draw  [fill={rgb, 255:red, 0; green, 0; blue, 0 }  ,fill opacity=1 ] (405.35,131.75) .. controls (405.35,132.51) and (405.97,133.13) .. (406.74,133.13) .. controls (407.51,133.13) and (408.14,132.51) .. (408.14,131.75) .. controls (408.14,130.99) and (407.51,130.38) .. (406.74,130.38) .. controls (405.97,130.38) and (405.35,130.99) .. (405.35,131.75) -- cycle ;
\draw  [fill={rgb, 255:red, 0; green, 0; blue, 0 }  ,fill opacity=1 ] (405.35,173) .. controls (405.35,173.76) and (405.97,174.38) .. (406.74,174.38) .. controls (407.51,174.38) and (408.14,173.76) .. (408.14,173) .. controls (408.14,172.24) and (407.51,171.63) .. (406.74,171.63) .. controls (405.97,171.63) and (405.35,172.24) .. (405.35,173) -- cycle ;

\draw  [fill={rgb, 255:red, 0; green, 0; blue, 0 }  ,fill opacity=1 ] (321.63,49.25) .. controls (321.63,50.01) and (322.25,50.63) .. (323.02,50.63) .. controls (323.79,50.63) and (324.42,50.01) .. (324.42,49.25) .. controls (324.42,48.49) and (323.79,47.88) .. (323.02,47.88) .. controls (322.25,47.88) and (321.63,48.49) .. (321.63,49.25) -- cycle ;
\draw  [fill={rgb, 255:red, 0; green, 0; blue, 0 }  ,fill opacity=1 ] (321.63,90.5) .. controls (321.63,91.26) and (322.25,91.88) .. (323.02,91.88) .. controls (323.79,91.88) and (324.42,91.26) .. (324.42,90.5) .. controls (324.42,89.74) and (323.79,89.13) .. (323.02,89.13) .. controls (322.25,89.13) and (321.63,89.74) .. (321.63,90.5) -- cycle ;
\draw  [fill={rgb, 255:red, 0; green, 0; blue, 0 }  ,fill opacity=1 ] (321.63,131.75) .. controls (321.63,132.51) and (322.25,133.13) .. (323.02,133.13) .. controls (323.79,133.13) and (324.42,132.51) .. (324.42,131.75) .. controls (324.42,130.99) and (323.79,130.38) .. (323.02,130.38) .. controls (322.25,130.38) and (321.63,130.99) .. (321.63,131.75) -- cycle ;
\draw  [fill={rgb, 255:red, 0; green, 0; blue, 0 }  ,fill opacity=1 ] (321.63,173) .. controls (321.63,173.76) and (322.25,174.38) .. (323.02,174.38) .. controls (323.79,174.38) and (324.42,173.76) .. (324.42,173) .. controls (324.42,172.24) and (323.79,171.63) .. (323.02,171.63) .. controls (322.25,171.63) and (321.63,172.24) .. (321.63,173) -- cycle ;

\draw  [fill={rgb, 255:red, 0; green, 0; blue, 0 }  ,fill opacity=1 ] (279.77,49.25) .. controls (279.77,50.01) and (280.39,50.63) .. (281.16,50.63) .. controls (281.93,50.63) and (282.56,50.01) .. (282.56,49.25) .. controls (282.56,48.49) and (281.93,47.88) .. (281.16,47.88) .. controls (280.39,47.88) and (279.77,48.49) .. (279.77,49.25) -- cycle ;
\draw  [fill={rgb, 255:red, 0; green, 0; blue, 0 }  ,fill opacity=1 ] (279.77,90.5) .. controls (279.77,91.26) and (280.39,91.88) .. (281.16,91.88) .. controls (281.93,91.88) and (282.56,91.26) .. (282.56,90.5) .. controls (282.56,89.74) and (281.93,89.13) .. (281.16,89.13) .. controls (280.39,89.13) and (279.77,89.74) .. (279.77,90.5) -- cycle ;
\draw  [fill={rgb, 255:red, 0; green, 0; blue, 0 }  ,fill opacity=1 ] (279.77,131.75) .. controls (279.77,132.51) and (280.39,133.13) .. (281.16,133.13) .. controls (281.93,133.13) and (282.56,132.51) .. (282.56,131.75) .. controls (282.56,130.99) and (281.93,130.38) .. (281.16,130.38) .. controls (280.39,130.38) and (279.77,130.99) .. (279.77,131.75) -- cycle ;
\draw  [fill={rgb, 255:red, 0; green, 0; blue, 0 }  ,fill opacity=1 ] (279.77,173) .. controls (279.77,173.76) and (280.39,174.38) .. (281.16,174.38) .. controls (281.93,174.38) and (282.56,173.76) .. (282.56,173) .. controls (282.56,172.24) and (281.93,171.63) .. (281.16,171.63) .. controls (280.39,171.63) and (279.77,172.24) .. (279.77,173) -- cycle ;

\draw (271,9.4) node [anchor=north west][inner sep=0.75pt]  [font=\footnotesize]  {$M_{1}$};
\draw (312,9.4) node [anchor=north west][inner sep=0.75pt]  [font=\footnotesize]  {$M_{2}$};
\draw (354,9.4) node [anchor=north west][inner sep=0.75pt]  [font=\footnotesize]  {$M_{3}$};
\draw (395,9.4) node [anchor=north west][inner sep=0.75pt]  [font=\footnotesize]  {$M_{4}$};
\draw (458,9.4) node [anchor=north west][inner sep=0.75pt]  [font=\footnotesize]  {$R$};

\end{tikzpicture}

\captionsetup{width=0.98\linewidth}
\caption{An illustration of the argument in \Cref{lem:asymptoticinrandom-dense}. 
}
\label{fig:flexible-proof}
\end{figure}

In order to carry out this strategy, we need to upper-bound the probability of failure in our applications of Corollaries~\ref{Cor_2_random_1_deterministic} and~\ref{cor:vxdisjointpaths} to various random sets.  Let us start with the latter.  Let $C_R \subseteq S$ be a $\frac{\tilde p}{1-q'}$-random subset, and define $C_R':=C_R \setminus S'$, which is a $\tilde p$-random subset of $S$.  Now \Cref{cor:vxdisjointpaths} applied to $H \setminus J$ tells us that with probability at least $1-5/(N-|J|)$, the following property holds: We can link any collection of up to 
\begin{equation}\label{eq:number-of-pairs}
    \frac{\tilde p^3 (\nu/2)^2}{300} \cdot (N-|J|)\ge \frac{\tilde p^3\zeta^2 \eps^2}{2^{19}} \cdot N
\end{equation} 
disjoint pairs of vertices with a rainbow path forest such that the paths use colours only from $C_R'$ and all of their internal vertices lie in $R\setminus J$.  Note that the final hypothesis of \Cref{cor:vxdisjointpaths} is satisfied since the right-hand side of \eqref{eq:number-of-pairs} is much larger than $\log N$.

The second desirable property is that for each $1 \le i\le t-1$ and \emph{every} subset $C \subseteq G$ of $pN$ colours, the graph $\Cayley_G(C)$ contains a rainbow matching between $M_i$ and $M_{i+1}$ covering all but at most $2N^{1-1/500}$ vertices of $M_i \cup M_{i+1}$ and using all but at most $2N^{1-1/500}$ colours of $C$. This happens for each fixed $i$ with probability at least $1-1/N^{2}$ by \Cref{Cor_2_random_1_deterministic} applied to the whole of $\Cayley_G(G)$ (which is $(0,1,n)$-typical) since $M_i,M_{i+1}$ are both $p$-random subsets with $p \ge N^{-1/600}$.  Notice that independence of $M_i, M_{i+1}$ is \emph{not} required for the application of \Cref{Cor_2_random_1_deterministic}. 

By a union bound we can ensure that with probability at least $1-7/N$, the properties from the previous two paragraphs simultaneously hold, and we have $|C_R'| \le 2\tilde pN$ and $|M_i| \le 2pN$ for all $i$ (using Chernoff bounds).  We will now establish the conclusion of the lemma under the assumption that this is the case.

Using the second property, we can find a rainbow matching between $M_1$ and $M_2$ using at least $pN-2N^{1-1/500}$ colours from $S \setminus (S_F \cup C_R')$. We then remove these newly-used colours from consideration and use the second property to obtain a rainbow matching between $M_2$ and $M_3$ using at least $pN-2N^{1-1/500}$ colours, and so on.  We continue until there are fewer than $pN$ unused colours of $S \setminus (S_F \cup C_R')$ remaining; this happens after at most $t-1$ steps because otherwise we would have used up
$$(t-1)(pN-2N^{1-1/500})\ge tpN -pN-2tN^{1-1/500}=N(q-\tilde p-p-2tN^{-1/500})>(1-\mu/2)qN \ge |S|$$ 
colours, which is impossible.  (The last inequality uses the hypothesis on the size of $q$.)

Consider the union of the matchings constructed in the previous paragraph, and throw out all edges incident to $J \cup \{u,v\}$  The remainder is a rainbow (directed) path forest using all but at most
$$pN+|C_R'|+|J|+2\le pN+2\tilde pN+\mu q N/4+2\le \mu q N$$
colours of $S \setminus S_F$.  Since each matching left at most $2N^{1-1/500}$ uncovered vertices on each side of $(M_i, M_{i+1})$, the total number of degree-$1$ vertices in this path forest is at most
$$4pN+t \cdot 2N^{1-1/500}+2|J|+4\le \frac{4q}{t} \cdot N+\frac{2^{29}N}{q^2\mu^3\zeta^2\eps^2 \cdot N^{1/500}}+\frac{q^3\mu^3\zeta^2\eps^2}{2^{27}}\cdot N+4\le \frac{q^3\mu^3\zeta^2\eps^2}{2^{25}}\cdot N = \frac{\tilde p^3 \zeta^2\eps^2}{2^{19}}\cdot N.$$
Fix an ordering $P_1, \ldots, P_m$ of the paths in our path forest, where $m \leq \frac{\tilde p^3 \zeta^2\eps^2}{2^{19}}\cdot N$.
By the linking-up property guaranteed above, we can find vertex-disjoint paths in $R \setminus J$ using colours in $C'_R$ that connect $u$ to the initial vertex of $P_1$, connect the final vertex of $P_i$ to the initial vertex of $P_{i+1}$ for each $1 \leq i \leq m-1$, and connect the final vertex of $P_m$ to $v$.  Putting everything together produces the desired long rainbow path.
\end{proof}

\section{The absorption $(99\%\to100\%)$ lemmas}\label{sec:absorb}

In this section we prove several lemmas which will allow us to run the absorption argument.  We start with the simplest one, in part to illustrate an argument which, in a somewhat more complicated form, will appear in several later lemmas. 

\begin{lemma}\label{lem:absorbing-junk}
   Let $0<p \leq 1$, and let $G$ be a finite group. Suppose $J \subseteq E \subseteq G\setminus\{\id\}$ satisfy $$|E|p^2 \ge \max(40|J|,96 \log |G|).$$ Let $A$ be a $p$-random subset of $G$. Then with high probability, 
   we can find, for each vertex $u\in G$, a rainbow path in $\Cayley_G(E)$ that starts at $u$, has all other vertices in $A$, and contains all of the colours in $J$.
\end{lemma}
\begin{proof}
Set $N:=|G|.$ For each vertex $v \in G$ and colour $j \in J$, let $E_{v,j}$ be the event that there are at least $5|J|$ paths of the form $$v, vg, vgj$$ with $g \in E \setminus \{j\}$ and $vg,vgj \in A$, and these paths are vertex-disjoint except on $v$.

We will show that these events are very likely.  Fix some $v \in G, j \in J$.  There are at least $|E|-2$ candidate paths $v, vg, vgj$ in $\Cayley_G(E)$ (since we may have to exclude $g=j^{-1}$ to guarantee $vgj \neq v$), and each such path intersects at most two other paths (since $vg=vg'j$ implies that $g=g'j$).  Thus we can greedily find a collection at least $(|E|-2)/3$ disjoint such paths.  Each path in this collection survives in $A$ with probability $p^2$, and these events are independent.  Hence the number of surviving paths dominates $\Bin(|E|/4,p^2)$, and a Chernoff bound tells us that at least $|E|p^{2}/8 > 5|J|$ of them survive with probability at least $1-\exp(-|E|p^{2}/32)\ge 1-1/N^3$. 
 Thus $\mathbb{P}[E_{v,j}] \geq 1-1/N^3$.  By a union bound, we conclude that with probability at least $1-1/N$ all of the events $E_{v,j}$ simultaneously occur.

Suppose we are in such an outcome.  We can find our desired path by starting at $u$ and repeatedly adding a length-$2$ path containing an arbitrary hitherto-unused element of $J$.  Indeed, since we have at least $5|J|$ candidate extensions at each step, we can ensure that the colour $g$ is hitherto unused (there are at most $2|J|-2$ colours already used) and that the two new vertices do not intersect the part of the path (of length at most $2|J|-2$) that we have already built.
\end{proof}

In the remainder of this section, we shall work specifically over $\Fn$ since our absorbing structures for general groups have a very different form.

\subsection{Building an absorbing path}\label{subsec:absorbing-path}
In this section we describe our absorbing path and show how to find it robustly. By an \emph{ordered subset} of $\Fn$ we mean a subset $F \subseteq \Fn$ together with an ordering on its elements. We write $f_i$ for the $i$-th element of $F$, and we write $\langle F \rangle$ for the subspace generated by $F$.  

\begin{definition}[Gadget]\label{defn:gadget}
Let $S \subseteq \Fn$.  An ordered subset $F \subseteq S$ is a \emph{gadget} in $S$ if $|F| \leq 6$, the elements of $F$ sum to $0$, and no proper subset of $F$ is $0$-sum. 
A family $\mathcal{F}$ of gadgets in $S$ is \emph{flexible} if the following all hold:
\begin{enumerate}[label = {{{\textbf{F\arabic{enumi}}}}}]
    \item \label{itm:flex-1} The elements of $\mathcal{F}$ are pairwise disjoint.
    \item \label{itm:flex-2} The sets of partial sums $\{f_1,f_1+f_2,\ldots, f_1+\ldots+f_{|F|-1}\}$ for $F \in \mathcal{F}$ are all disjoint.
    \item \label{itm:flex-3} For any distinct $F_1,F_2 \in \mathcal{F}$, we have $|\langle F_1 \rangle \cap \langle F_2 \rangle |\le 2$.     
\end{enumerate}

\end{definition}

Equivalently, $F$ is a gadget if and only $|F| \leq 6$ and starting at any vertex $v$ and following the edges of the colours of $F$ (in order) produces a rainbow cycle.  Removing the last edge of such a rainbow cycle produces a rainbow path starting from $v$ associated to the gadget $F$.  If $\mathcal{F}$ is a flexible family of gadgets, then for each vertex $v$, the rainbow paths from $v$ associated to the gadgets in $\mathcal{F}$ are vertex-disjoint except for $v$.  (This fact uses only \ref{itm:flex-1} and \ref{itm:flex-2}. The role of \ref{itm:flex-3} will become clear later; at a high level, it ensures that different gadgets do not interact too much.)  The union of these paths is a rainbow tree which we will refer to as an \emph{out-spider} of $\mathcal{F}$.  An \emph{in-spider} of $\mathcal{F}$ is an out-spider of the gadget obtained from $\mathcal{F}$ by reversing the ordering of each gadget.

For example, if $v$ is a vertex and $F=\{f_1,f_2,f_3,f_4\}$ is a gadget in a flexible family $\mathcal{F}$, then the path $v,v+f_1,v+f_1+f_2,v+f_1+f_2+f_3$ forms a leg of the out-spider of $\mathcal{F}$ and $v,v+f_4,v+f_4+f_3,v+f_4+f_3+f_2$ forms a leg of an in-spider of $\mathcal{F}$. Notice that an out-spider and an in-spider with the same starting vertex $v$ have the same vertex set, since for any gadget $F$ we have $\{f_1,f_1+f_2,\ldots, f_1+\ldots+f_{|F|-1}\}=\{f_2+\ldots+f_{|F|},f_3+f_4+\ldots+f_{|F|},\ldots,f_{|F|}\}$ due to the $0$-sum assumption. See \Cref{fig:in-out-spider} for an illustration.

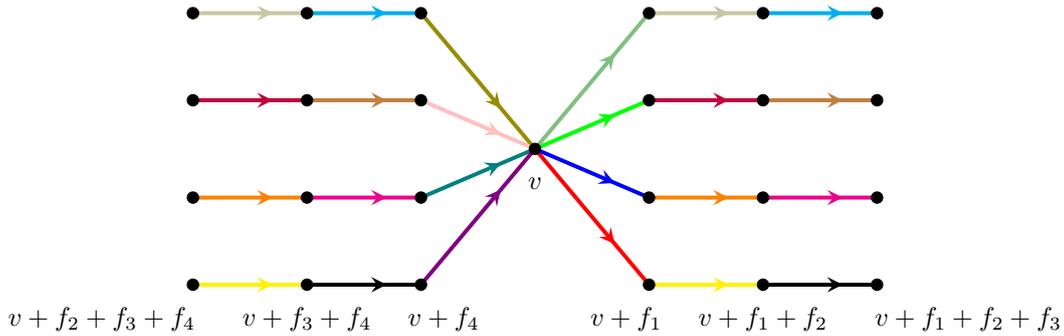
\begin{figure}[h]
\centering
\begin{tikzpicture}[scale=1.5]

\defPt{3.5}{0}{a1};
\defPt{3.5}{-0.43}{a2};
\defPt{3.5}{0.43}{a3};
\defPt{3.5}{1.2}{a4};
\defPt{4.5}{-1.2}{b1};
\defPt{4.5}{-0.43}{b2};
\defPt{4.5}{0.43}{b3};
\defPt{4.5}{1.2}{b4};
\defPt{5.5}{-1.2}{c1};
\defPt{5.5}{-0.43}{c2};
\defPt{5.5}{0.43}{c3};
\defPt{5.5}{1.2}{c4};
\defPt{6.5}{-1.2}{d1};
\defPt{6.5}{-0.43}{d2};
\defPt{6.5}{0.43}{d3};
\defPt{6.5}{1.2}{d4};
\defPt{7.5}{-0.6}{e1};
\defPt{7.5}{-0.4}{e2};
\defPt{7.5}{0}{e3};
\defPt{8.5}{0.16}{e4};
\defPt{7.5}{0.33}{e5};
\defPt{7.5}{0.86}{e6};
\defPt{7.5}{1}{e7};

\draw[diredge,red,line width=1.5pt] (a1) -- (b1);
\draw[diredge,blue,line width=1.5pt] (a1) -- (b2);
\draw[diredge,green,line width=1.5pt] (a1) -- (b3);
\draw[diredge,green!50!black!50!,line width=1.5pt] (a1) -- (b4);
\draw[diredge,yellow,line width=1.5pt] (b1) -- (c1);
\draw[diredge,orange,line width=1.5pt] (b2) -- (c2);
\draw[diredge,purple,line width=1.5pt] (b3) -- (c3);
\draw[diredge,yellow!50!black!50!,line width=1.5pt] (b4) -- (c4);
\draw[diredge,black,line width=1.5pt] (c1) -- (d1);
\draw[diredge,magenta,line width=1.5pt] (c2) -- (d2);
\draw[diredge,brown,line width=1.5pt] (c3) -- (d3);
\draw[diredge,cyan,line width=1.5pt] (c4) -- (d4);

\draw[diredge,yellow!50!black!50!,line width=1.5pt] ($2*(a1)-(d1)$) -- ($2*(a1)-(c1)$);
\draw[diredge,purple,line width=1.5pt] ($2*(a1)-(d2)$) -- ($2*(a1)-(c2)$);
\draw[diredge,orange,line width=1.5pt] ($2*(a1)-(d3)$) -- ($2*(a1)-(c3)$);
\draw[diredge,yellow,line width=1.5pt] ($2*(a1)-(d4)$) -- ($2*(a1)-(c4)$);

\draw[diredge,cyan,line width=1.5pt] ($2*(a1)-(c1)$) -- ($2*(a1)-(b1)$);
\draw[diredge,brown,line width=1.5pt] ($2*(a1)-(c2)$) -- ($2*(a1)-(b2)$);
\draw[diredge,magenta,line width=1.5pt] ($2*(a1)-(c3)$) -- ($2*(a1)-(b3)$);
\draw[diredge,black,line width=1.5pt] ($2*(a1)-(c4)$) -- ($2*(a1)-(b4)$);

\draw[diredge,olive,line width=1.5pt] ($2*(a1)-(b1)$) -- (a1);
\draw[diredge,pink,line width=1.5pt] ($2*(a1)-(b2)$) -- (a1);
\draw[diredge,teal,line width=1.5pt] ($2*(a1)-(b3)$) -- (a1);
\draw[diredge,violet,line width=1.5pt] ($2*(a1)-(b4)$) -- (a1);

\draw[] (a1) \smvx;
\foreach \i in {1,...,4}
{

\draw[] (b\i) \smvx;
\draw[] ($2*(a1)-(b\i)$) \smvx;
\draw[] (c\i) \smvx;
\draw[] ($2*(a1)-(c\i)$) \smvx;
\draw[] (d\i) \smvx;
\draw[] ($2*(a1)-(d\i)$) \smvx;
}

\node[] at ($(a1)+(0,-0.3)$) {$v$};
\node[] at ($(b1)+(-0.2,-0.3)$) {$v+f_1$};
\node[] at ($(c1)+(0,-0.3)$) {$v+f_1+f_2$};
\node[] at ($(d1)+(0.8,-0.3)$) {$v+f_1+f_2+f_3$};

\node[] at ($2*(a1)-(b4) +(0.2,-0.3)$) {$v+f_4$};
\node[] at ($2*(a1)-(c4)+(0,-0.3)$) {$v+f_3+f_4$};
\node[] at ($2*(a1)-(d4)+(-0.8,-0.3)$) {$v+f_2+f_3+f_4$};


\end{tikzpicture}
\caption{An out-spider and an in-spider of a flexible family $\mathcal{F}$. The figure is misleading in representing the out-spider and in-spider on different vertex sets. The bottom two legs correspond to the gadget $F=\{f_1,f_2,f_3,f_4\}$. 
}
\label{fig:in-out-spider}
\end{figure}

Our absorbing structure will allow us to choose, for each gadget $F \in \mathcal{F}$, either to leave all of the colours of $F$ in the absorbing structure or to free them all up for use embedding  other colours elsewhere. In an idealised scenario (which provides good intuition), each $F$ would consist of a single colour, and then $\bigcup \mathcal{F}$ would represent a set of flexible colours which we may absorb into our absorbing structure at the very end of the argument if they ended up being unneeded elsewhere.  Since of course there are no non-trivial $0$-sum single elements, we must package our flexible colours in short tuples (of size at most six), as encoded by our gadgets. 

We can find a large flexible family in any reasonably large subset of $\Fn$, essentially by the pigeonhole principle. 
\begin{lemma}\label{lem:good-family}
    Let $0< \eps\le 1$.  If $E \subseteq \Fn$ has size $|E| \ge \eps N^{1/2}$, where $N:=2^n$, then $E$ contains a flexible family with at least $\floor{\eps^2|E|/2^{49}}$ gadgets.
\end{lemma}

\begin{proof}
    We may assume that $|E| \ge 2^{49}/\eps^2$, as otherwise the statement is trivial.
    Let us take $\mathcal{F}$ to be a maximal flexible family of gadgets in $E$. 
    Towards a contradiction, let us assume
    that $|\mathcal{F}| < \eps^2|E|/2^{49}$.
    Define the set $B:= \bigcup_{F \in \mathcal{F}} \langle F \rangle$ of \emph{blocked} vertices. Note that each $F\in\mathcal{F}$ is a $0$-sum set of size at most $6$, so $|\langle F\rangle|\leq 32$ and hence $|B|\le 32|\mathcal{F}| \le \eps^2|E|/32$. 
     
    Now consider triples $\{e_1,e_2,e_3\} \subseteq E$ of linearly independent elements  such that $\langle e_1,e_2,e_3 \rangle \cap B = \{0\}$. 
    We have at least $(1-\eps^2/32)|E|-1\ge 2^{-1/3}|E|$ such choices for $e_1 \neq 0$ (ensuring $e_1 \notin B$), then $(1-\eps^2/16)|E|-2\ge 2^{-1/3}|E|$ choices for $e_2 \notin \langle e_1\rangle$ (ensuring  $e_2,e_1+e_2 \notin B$) and $(1-\eps^2/8)|E|-4\ge 2^{-1/3}|E|$ choices for $e_3 \notin \langle e_1,e_2\rangle$ (ensuring the remaining four subsums are not in $B$). 
    Since we counted each triple $6$ times, there are at least $|E|^3/12$ many such triples.  Fix an ordering of the elements of $\Fn$, and let $s_i$ denote the number of triples summing to the $i$-th element of $\Fn$.  Then $s_1+\ldots+s_N \ge |E|^3/12$, and (by convexity) there are at least $\binom{s_1}{2}+\ldots+\binom{s_N}{2} \ge |E|^6/(512N)$ $6$-tuples $(e_1,e_2,e_3,e_4,e_5,e_6)$ such that $e_1+e_2+e_3=e_4+e_5+e_6$ and $\langle e_1,e_2,e_3 \rangle, \langle e_4,e_5,e_6 \rangle$ are disjoint from $B\setminus \{0\}$; let us call such $6$-tuples \emph{good}.
    
    Since $\dim(\langle e_1,e_2,e_3,e_4,e_5,e_6 \rangle \leq 5$, there are at most $32^{6}$ good $6$-tuples with a given span.  Thus we can find a subcollection of at least $|E|^6/(2^{39}N)$ good $6$-tuples spanning pairwise distinct subspaces. We will be done if we can show that some such good $6$-tuple $F'=(e_1,e_2,e_3,e_4,e_5,e_6)$ satisfies $|\langle F'\rangle\cap\langle F\rangle|\leq 2$ for all $F\in \mathcal{F}$, since then we can add a suitable $0$-sum subset of $F'$ to $\mathcal{F}$, contradicting the maximality of $\mathcal{F}$.

    \par 
    There are at most $32^2 \cdot |\mathcal{F}|$ pairs of distinct nonzero elements $(a,b)$ such that $a,b \in \langle F\rangle$ for some $F \in \mathcal{F}$. Each such pair $(a,b)$ is contained in at most $1+|E|+\binom{|E|}{2}+\binom{|E|}{3} \leq |E|^3$ subspaces of the form $\langle F' \rangle$ as $F'$ ranges over our subcollection of good $6$-tuples (since any such subspace can be obtained as the span of $a,b$ and at most $3$ elements of $E$).    
    When we range over the pairs $(a,b)$, there are at most $$32^2 \cdot |\mathcal{F}| \cdot |E|^3 <\eps^2 |E|^4/2^{39}\le |E|^6/(2^{39}N)$$ such subspaces in total.  In particular, we can choose a good tuple $F'=(e_1,e_2,e_3,e_4,e_5,e_6)$ for which there are  no such pairs $(a,b)$; this means that $|\langle e_1,e_2,e_3,e_4,e_5,e_6 \rangle \cap \langle F \rangle|\le 2 $ for every $F \in \mathcal{F}$. Now let $F''$ be
a minimal $0$-sum subset of $\{ e_1,e_2,e_3,e_4,e_5,e_6 \}$ and fix an ordering of $F''$ which first traverses the elements from
$\{e_1, e_2, e_3\}$ and only afterwards traverses the elements from $\{e_4, e_5, e_6\}$; then $F''$ is a new gadget which can be added to $\mathcal{F}$, giving the desired contradiction.

Let us check more explicitly that $\mathcal{F}\cup\{F''\}$ is a flexible family. \ref{itm:flex-1} holds as we chose each $e_i\notin \bigcup_{F \in \mathcal{F}} \langle F \rangle$.  We chose $F'$ to satisfy $|\langle F' \rangle \cap \langle F \rangle|\le 2 $ for all $F \in \mathcal{F}$; a fortiori the same holds with $F'$ replaced by $F''$, so \ref{itm:flex-3} holds. 
It remains to verify \ref{itm:flex-2}.  Write $F''=\{e_{i_1}, \ldots, e_{i_t}\}$.  Each partial sum $e_{i_1}+\cdots+e_{i_r}=e_{i_{r+1}}+\cdots+e_{i_t}$ is in either $\langle e_1, e_2, e_3 \rangle$ or $\langle e_4, e_5, e_6 \rangle$ according to whether or not $i_r \leq 3$; either way, the sum is by construction not in $B$.
\end{proof}

To gain intuition on a first read-through, the reader may wish to think of the properties \ref{itm:flex-1}--\ref{itm:flex-3} in the definition of a flexible family as saying that $\langle F \rangle \cap\langle F' \rangle=\{0\}$ for distinct $F,F' \in \mathcal{F}$. This stronger property implies all of \ref{itm:flex-1}--\ref{itm:flex-3}. There is, however, one instance where we wish to find such a family but we will not be able to ensure this stronger zero-intersection property. 

The following easy proposition allows us to obtain a short rainbow path from a gadget and an arbitrary element not in the gadget. This will come in handy in several places. 

\begin{proposition}\label{prop:eat-an-element}
    Let $F$ be a gadget, and let $x \notin F$ be any nonzero element.  Then we can order the elements of $F$ in such a way that $x$ is not equal to any contiguous subsum of $F$. In particular, inserting $x$ into this ordering of $F$ in any position except for the first or the last produces a valid ordering of $F\cup\{x\}$.
\end{proposition}
\begin{proof}
    If $x \notin \langle F \rangle$, then any ordering of $F$ will do, so suppose that $x \in \langle F \rangle$. Since $F$ is a gadget, it has no nontrivial zero subsums.  Thus there is a nonempty subset $T \subsetneq F$, unique up to complementation, such that $$x=\sum_{f \in T} f=\sum_{f \in F \setminus T} f.$$ Our task is to show that the elements of $F$ can be ordered in such a way that neither the elements of $T$, nor the elements of $F\setminus T$ appear as a contiguous subsequence. Since $x \notin F$, we know that $|T|, |F \setminus T| \geq 2$.
    We can build our desired ordering by taking all but one of the elements of $T$, then one element of $F \setminus T$, then the last element of $T$, then the remaining elements of $F \setminus T$ (in any way).   
\end{proof}

The next step is incorporating a flexible family of gadgets into an absorbing path in $\Cay{S}$.
\begin{definition}[Absorbing path]
We say a rainbow path $P$ in $\Cay{S}$ is \emph{$\mathcal{F}$-absorbing} for a flexible family $\mathcal{F}$ of gadgets in $S$ if there exists an injective function $c: \mathcal{F} \to S \setminus \bigcup \mathcal{F}$ such that for each $F \in \mathcal{F}$ we can find a subpath of $P$ using precisely the colours in $F \cup c(F)$. We say the colours of $P$ not in $\bigcup \mathcal{F}$ are the \emph{fixed colours} of $P$.
\end{definition}

In an $\mathcal{F}$-absorbing path $P$, for each $F \in \mathcal{F}$ we can delete the subpath of $P$ consisting of the edges with colours $F \cup c(F)$. Doing so leaves two subpaths of $P$, which we can join using a single edge of colour $c(F)$ (since $F$ is zero-sum).  We will denote the resulting subpath by $P-F$; see \Cref{fig:absorbingpath} for an illustration.

The following lemma will let us find an absorbing path inside a random vertex subset while avoiding a small set of forbidden vertices.

\begin{lemma}\label{lem:absorbing-path}
   Let $p\in (0,1]$, let $\mathcal{F}$ be a flexible family of gadgets in $E \subseteq \Fn$, and let $U \subseteq \Fn$ be a subset of size $|U| \le |\mathcal{F}|$. Suppose $p^8|E| \ge 2^{12}|\mathcal{F}| \ge 2^{13}n$. Let $R$ be a $p$-random subset of $\Fn$. Then with high probability, 
   we can find, for each $u \in \Fn$, an $\mathcal{F}$-absorbing rainbow path in $\Cay{E}$ of length at most $8|\mathcal{F}|$ that starts at $u \in \Fn$ and has all other vertices in $R\setminus U$.
\end{lemma}
\begin{proof} Fix $N=2^n$. First we add a fixed, unique colour $c_F \in E \setminus \bigcup \mathcal F$ to each gadget $F \in \mathcal{F}$ and construct a rainbow path $P_F$ that starts at $0$ and uses the colours $\{c_F\} \cup F$ (which is possible by \Cref{prop:eat-an-element}).  Write $P_{F,y}$ for the translate of $P_F$ starting at the vertex $y \in G$. 
Let $$X:=E \setminus \bigcup_{F \in \mathcal{F}} \left( \{c_F\} \cup F\right)$$ be the set of unused colours from $E$, and notice that that $|X| \ge |E|-7|\mathcal{F}|\ge |E|/2$.
For each vertex $v \in \Fn$ and gadget $F \in \mathcal{F}$, we define $E_{v,F}$ to be the event that we can find a collection of at least $10|\mathcal{F}|$ elements $x \in X$ whose corresponding paths $P_{F,v+x}$ are all vertex-disjoint and contained in $R$.

We will show that these events are (very) likely.  Fix some $v \in \Fn$, $F \in \mathcal{F}$.  We will find many paths $P_{F,v+x}$ which are disjoint and do not contain $v$. There are $|X|$ paths in total.  Of these, at most $|P_F|+1$ contain $v$, since the position of $v$ in a translate of $P_F$ determines the translate.  Each path $P_{F,v+x}$ can intersect at most $(|P_F|+1)^2$ other such paths, since again the translate of the other path is determined by the relative positions of the intersection point in the two paths.
Thus there is a family of $|X|/100$ vertex-disjoint paths $P_{F,v+x}$ avoiding $v$. Each such path is contained in $R$ with probability $p^{|P_F|+1}$, and these events are independent. Hence the number of surviving paths dominates $\Bin(|X|/100,p^{|P_F|+1})$, 
and by a Chernoff bound at least $$|X|p^{|P_F|+1}/200 \ge |E|p^8/400 \ge 10|\mathcal{F}|$$ survive with probability at least $1-\exp(-|X|p^{|P_F|+1}/800)\ge 1-1/N^3$ (using $|E|p^8 \ge 2^{12}|\mathcal{F}| \ge 2^{13}n$). Thus $\mathbb{P}[E_{v,F}] \geq 1-1/N^3$, and by a union bound we conclude that with probability at least $1-1/N$ all of the events $E_{v,F}$ occur. 

Suppose we are in such an outcome. We find our $\mathcal{F}$-absorbing path by incorporating gadgets $F$ one at a time, as in the proof of \Cref{lem:absorbing-junk}.  We start our path $P$ at the vertex $u$ and iteratively add on paths of the form $P_{F,x+v}$, where $v$ is the current endpoint of $P$.  At each step, we identify a hitherto-unincorporated gadget $F$ and consider the $10|\mathcal{F}|$ paths $P_{F,x+v}$ identified in the previous paragraph.  Of these, at least $9|\mathcal{F}|$ correspond to values of $x$ that have not yet been used.  Since $|P|< 8|\mathcal{F}|$, there are more than $|\mathcal{F}|$ paths $P_{F,x+v}$ that remain disjoint from $P$.  Finally, since $|U| \leq |\mathcal{F}|$, we can choose a path $P_{F,x+v}$ that is also disjoint from $U$ (notice that each element of $U$ eliminates at most one choice of $x$ since the paths $P_{F,x+v}$ are vertex disjoint); we choose one such path and concatenate $P$ with it.
\end{proof}

\subsection{The absorbing lemma}\label{subsec:tails-lemma}
In this subsection we establish a lemma which will eventually allow us to ``absorb'' any small subset of colours using the flexibility provided by an absorbing path. We also need the ability to work within a random vertex subset and guarantee that we avoid a given small subset of forbidden vertices.

\begin{lemma}\label{lem:tails}
    Let $p\in(0,1]$, let $\mathcal{F}$ be a flexible family of at least $2^{12}p^{-7}n$ gadgets in $S \subseteq \Fn \setminus \{0\}$, let $U \subseteq \Fn$ be a set of size $|U| \le |\mathcal{F}|/128$. Let $T\subseteq \Fn$ be a $p$-random set.  Then with high probability, the following holds for \emph{every} $L \subseteq S$ of size $|L| \le |\mathcal{F}|p^7/2^{12}$ and every vertex $v\in \Fn$: There exist a subfamily of gadgets $\mathcal{F}' \subseteq \mathcal{F}$ and a rainbow path in $\Cay{L \cup \bigcup_{F\in \mathcal{F}'} F}$ that starts at $v$, is otherwise contained in $T \setminus U$, and uses all except possibly one colour from $L \cup \bigcup_{F\in \mathcal{F}'} F$.
\end{lemma}
\begin{proof}

     Consider a pair of distinct colours $a,b \in S$. Our first goal is to construct a subfamily $\mathcal{F}_{a,b} \subseteq \mathcal{F}$ consisting of at least $|\mathcal{F}|/64$ gadgets $F\in\mathcal{F}$ (possibly not inheriting the original orderings $\{f_1,f_2,\dots,f_{|F|}\}$) such that extending each leg of the $\mathcal{F}_{a,b}$-out-spider starting at $0$ by the edge of colour $a$ and then the edge of colour $b$ produces a family of vertex-disjoint paths (except for the shared initial vertex $0$).
     
For each $F \in \mathcal{F}$, consider the walk $P_F$ that starts at $0$ and then follows the edges of colours $f_1, \ldots, f_{|F|-1}, a, b$ (recall that $F=\{f_1,\dots,f_{|F|}\}$).  Note that $P_F$ is a bona fide path as long as $a,a+b \notin \langle F \rangle$.  We claim that each path $P_F$ can intersect at most $11$ other paths $P_{F'}$ at vertices other than $0$.  Indeed, $P_F$ can intersect $P_{F'}$ only if $\{f'_1, f'_1+f'_2, \ldots, f'_1+\cdots+f'_{|F'|-1}\}$ intersects the set
$$\{f_1+\cdots+f_i, f_1+\cdots+f_i+a, f_1+\cdots+f_i+a+b: 1 \leq i \leq |F|-1\}\cup \{f_1+\cdots+f_{|F|-1}+b\}.$$
Property \ref{itm:flex-2} in the definition of flexibility ensures that $f_1+\cdots+f_i$ can never appear in $\{f'_1, f'_1+f'_2, \ldots, f'_1+\cdots+f'_{|F'|-1}\}$.  This leaves us with at most $2(|F|-1)+1 \leq 2(5)+1=11$ possible collisions.

It follows that if we can find a collection of $|\mathcal{F}|/3$  gadgets $F\in\mathcal{F}$ for which $P_F$ is a path (as opposed to just a walk), then we can find the desired subfamily $\mathcal{F}_{a,b}$ consisting of at least $|\mathcal{F}|/36$ gadgets $F$ whose corresponding paths $P_F$ are vertex-disjoint (except for $0$).

Suppose instead that for some $x \in \{a,a+b\}$ there are at least $|\mathcal{F}|/3$ gadgets $F\in\mathcal{F}$ with $x \in \langle F \rangle$. As the sets $F\in\mathcal{F}$ are disjoint by property \ref{itm:flex-1} of flexibility, there is at most one such $F$ which contains $x$; let us remove it from consideration (if it exists). For each remaining $F$ we have $x \in\langle F \rangle \setminus F$; \Cref{prop:eat-an-element} provides an ordering of the elements of $F$ such that $x$ is not equal to any contiguous subsum of $F$. 
     
     The walk $P_F$ with respect to this ordering of $F$ is a bona fide path. 
     The spans of any two such $F$'s intersect precisely in $\langle x \rangle$ by \ref{itm:flex-3}, so there are no collisions among the sets $\{f_1, f_1+f_2, \ldots, f_1+\cdots+f_{|F|-1}\}$. 
 We can thus repeat the above argument from the second paragraph of the proof in order to find the desired family $\mathcal{F}_{a,b}$. 
     
     Let $\ell:=|\mathcal{F}|p^7/2^{12}$.
     For each vertex $u \in \Fn \setminus U$ and two colours $a,b \in S$, let $E_{u,a,b}$ be the event that we can find a collection of at least $10\ell$ gadgets $F \in \mathcal{F}_{a,b}$ for which the translate of $P_F$ starting at $u$ is contained in $T$ (except possibly $u$) and does not intersect $U$. By the above considerations, there are at least $|\mathcal{F}|/128$ such paths which avoid $U$.
     The number of surviving such paths in $R$ dominates $\Bin(|\mathcal{F}|/128,p^7)$. By Chernoff's bound, at least $|\mathcal{F}|p^7/2^8>10\ell$ of these paths survive (i.e., $E_{u,a,b}$ occurs) with probability at least $1- \exp(-|\mathcal{F}|p^7/2^{10})\ge 1-1/N^4$. A union bound over $u,a,b$ ensures that with probability at least $1-1/N$ all of the events $E_{u,a,b}$ occur. 

     Suppose we are in such an outcome.  We will construct a sequence of subsets $L=L_0, L_1, \ldots, L_{m}$ and a sequence of directed rainbow paths $P_0 \subset P_1 \subset \cdots \subset P_{m}$ starting at $v$ such that for each $0 \leq i \leq m\le |L|-1$, we have $|L_i|\le |L_{i-1}|-1$, and the path $P_i$ is contained in $T\setminus U$, has size $|P_i|\le |P_{i-1}|+7$ 
     , and contains $L \setminus L_i$.  The path $P_{m}$ will satisfy the conclusion of the lemma.  Suppose we have already done this for some $i<|L|-1$.  Then $|L_i| \geq 2$.  Pick some distinct $a,b \in L_i$.  By the above considerations, we can find $10\ell-2 \ge 10|L|-2$ vertex-disjoint rainbow paths, each of which uses edges with colours $a,b$ and all but one of the elements of some $F \in \mathcal{F}$, and starts at the endpoint of $P_i$ and has all other vertices in $T\setminus U$.  One of these paths uses a new $F$, does not use any of the already used colours and is vertex-disjoint from $P_i$ since $2i+|P_i|\le 9i+1 \leq 9|L|$.  Now append this path to $P_i$ to obtain $P_{i+1}$.  To obtain $L_{i+1}$ from $L_i$, remove $a,b, F \cap L_i$ and add the unused element of $F$. 
\end{proof}

\section{Proof of the dense case for \texorpdfstring{$\Fn$}{F2n}}\label{sec:f2n-dense}
In this section we prove Graham's conjecture over $\Fn$ in the dense case, namely, the case where the size of the set $S$ is linear in $N:=2^n$. The results of \Cref{sec:general-dense} show that any subset $S\subset G\setminus\{0\}$ of size $|S|\geq |G|^{1-c}$, in any finite group $G$, admits a valid ordering, so those results subsume the results in this section. We include a short proof of the weaker result here to demonstrate the implementation of the tools from the previous two sections, which we will also need for the sparse case of $\Fn$.  We also note that the simpler results in this section already suffice for proving \Cref{thm:mainthm} using only the basic absorption argument (similar in spirit to one used in \cite{erdHos1991vertex}), rather than the distributive absorption tools that we will need for the general dense case in \Cref{sec:general-dense}. 

The following theorem handles the extremely dense case. It is convenient to isolate this regime since in the dense-but-not-extremely-dense case our absorption arguments will make essential use of the resulting extra vertex space.
The result that we need is contained in \cite{muyesser2022random}; see \Cref{sec:appendix} for more details. 

\begin{theorem}[\cite{muyesser2022random}]\label{thm:ultra-dense}
 Let $\gamma>0$.  Then for all sufficiently large $N$ the following holds: For every group $G$ of order $N$, every subset $S\subseteq G\setminus\{\mathrm{id}\}$ with $|S|\geq N - N^{1-\gamma}$ has a valid ordering.
\end{theorem}

For the rest of this section, assume that $S\subset\Fn\setminus\{0\}$ has  linear size in $N=2^n$. The case $|S| \ge \frac34 N$, which requires a separate treatment due to tighter space constraints, will serve as a simple illustration of our strategy.  The main idea is that we first set aside an absorbing path, then find a rainbow path using nearly all of the remaining colours of $S$, and finally use some of the gadgets from the absorbing path to integrate the remaining few colours of $S$.  To prevent unwanted collisions among the rainbow paths produced in these three steps, we carry out each step in its own random vertex subset.

\begin{theorem}\label{thm:very-dense}
Let $n$ be sufficiently large, and set $N:=2^n$.  If $S \subseteq \Fn\setminus\{0\}$ is a subset of size $ |S| \ge \frac34 N$, then $\Cay{S}$ has a rainbow path of length $|S|-1.$ 
\end{theorem}
\begin{proof}
Set $\gamma:=2^{-14}$. If $|S| \ge N-N^{1-\gamma/32}$, then we are done by \Cref{thm:ultra-dense}.  It remains to consider the case $|S| \le N-N^{1-\gamma/32}$. Set $p:=N^{-\gamma/16}/2$.

Let $E$ be a $\frac14$-random subset of $S$.
Partition $\Fn$ into three sets $R \sqcup M \sqcup T$ by independently assigning each vertex to $R,M,T$ with probabilities $p,1-2p,p$, respectively.

We apply  \Cref{lem:asymptoticinrandom-dense} with $S=S$, $J=\emptyset$, $M=M$, 
 $S'=E$ and parameters 
$$\eps=3/4, \quad q=1-2p, \quad q'=1/4, \quad \zeta=1/8, \quad \mu=N^{-\gamma}.$$
The graph $\Cay{S}$ has no $\frac14$-sparse cuts since $|S| \ge \frac34 N$, and the other hypotheses of the lemma are easy to check ($|J|=0$, $1-2p=q\ge (1+\mu)(1-N^{-\gamma/32})$, and $q' \le 1-\mu q/4$).  Thus with high probability we have:

\begin{enumerate}[label = {{{\textbf{P\arabic{enumi}}}}}]
    \item\label{S1} For any $S_F \subseteq E$ and any two vertices $u,v \in \Fn$, we can find a rainbow path from $u$ to $v$ in $\Cay{S \setminus S_F}$, using all but at most $\mu q$ colours from $S \setminus S_F$, such that all of the internal vertices of the path lie in $M$.
\end{enumerate}

By a Chernoff bound we have $|E| \ge N/8$ with high probability. In any such outcome, \Cref{lem:good-family} (with $\eps=1$) lets us find a flexible family $\mathcal{F}$ of gadgets in $E$ of size $Np^8/2^{15}$ (since this is at most $|E|/2^{50}$); fix such a flexible family for each outcome. 
\Cref{lem:absorbing-path} with $\mathcal{F}=\mathcal{F}, E=E, U=\emptyset, R=R$ (note that $|\mathcal{F}|\le |E|p^8/{2^{12}}$) 
guarantees that with high probability we have:

\begin{enumerate}[label = {{{\textbf{P\arabic{enumi}}}}}] \setcounter{enumi}{1}
    \item\label{S2} There is an $\mathcal{F}$-absorbing rainbow path in $\Cay{E}$ starting from any given vertex and otherwise contained in $R$.
\end{enumerate}
\Cref{lem:tails} with $\mathcal{F}=\mathcal{F}, S=S, U=\emptyset, T=T$ (note that $|\mathcal{F}|\ge 2^{12}p^{-7}\log N$) shows that with high probability we have: 
\begin{enumerate}[label = {{{\textbf{P\arabic{enumi}}}}}] \setcounter{enumi}{2}
    \item\label{S3} For any $L \subseteq S$ of size $|L| \le  |\mathcal{F}|p^7/2^{12}$ and any vertex $v\in \Fn$, there is some $\mathcal{F}'\subseteq \mathcal{F}$ such that $\Cay{L \cup \bigcup_{F \in \mathcal{F}'}F}$ has a rainbow path that starts at $v$, is otherwise contained in $T$, and uses all except possibly one the colours from $L \cup \bigcup_{F \in \mathcal{F}'}F$.
\end{enumerate}
From now on consider an outcome for $E,M,R,T$ where conclusions \ref{S1}-\ref{S3} hold. 

Fix some distinct vertices $u,v \in M$. \ref{S2} provides an $\mathcal{F}$-absorbing rainbow path $P_A$ starting at $u$, otherwise contained in $R$, and using only colours from $E$; write $S_F$ 
for the the set of colours from $E$ appearing in $P_A$. Now \ref{S1} 
allows us to find a rainbow path $P_M$ from $u$ to $v$ which is contained in $M$ and saturates all but some set $L$ of up to $\mu qN$ colours from $S\setminus S_F$. Note that $|L| \leq \mu q N \le N^{1-\gamma} \le Np^{15}/2^{27}=|\mathcal{F}|p^7/2^{12}$.  Now $P_A \cup P_M$ is a rainbow path using precisely the colours in $S \setminus L$.
Next, as $|L|\le |\mathcal{F}|p^7/2^{12}$, by \ref{S3} 
we can find a subfamily of gadgets $\mathcal{F}' \subseteq \mathcal{F}$ and a rainbow path $P_T$ starting at $v$ and otherwise contained in $T$ which uses all except possibly one colour of $L \cup \bigcup_{F \in \mathcal{F}'} F$. 
Now we use the $\mathcal{F}$-absorbing properties of $P_A$ to remove $\bigcup_{F \in \mathcal{F}'} F$ and pass to a shorter path $P_A'$ using only a subset of the vertices of $P_A$ (see the illustrations in \Cref{99problemsbut1percentaintone}). Finally, $P_A' \cup P_M \cup P_T$ is a rainbow path using all but one colour from $S$, as desired.
\end{proof}

We now turn to the main argument for the dense regime.  In the very-dense setting of \Cref{thm:very-dense}, the Cayley graph $\Cayley_{\Fn}(S)$ was automatically a robust expander.  In the merely-dense regime, we have to use our regularity lemma to locate a robustly expanding part of $\Cayley_{\Fn}(S)$; this requires setting aside a few colours of $S$ that lie outside of the subspace $H$ from \Cref{cor:robexpander}, and re-integrating these colours causes some additional technical complications.  Also, to avoid the case where $S \cap H$ is too dense in $H$ for \Cref{lem:asymptoticinrandom-dense} to apply, we artificially remove a few of these colours and re-integrate them separately, as with the colours in $S \setminus H$.

\begin{theorem}\label{thm:dense}
    Let $\eps \in (0,1/16)$, and let $n$ be sufficiently large in terms of $\varepsilon$.  Set $N:=2^n$. Then for any $S \subseteq \Fn$ of size $|S| \ge \eps N$, the graph $\Cay{S}$ has a rainbow path of length $|S|-1.$ 
\end{theorem}

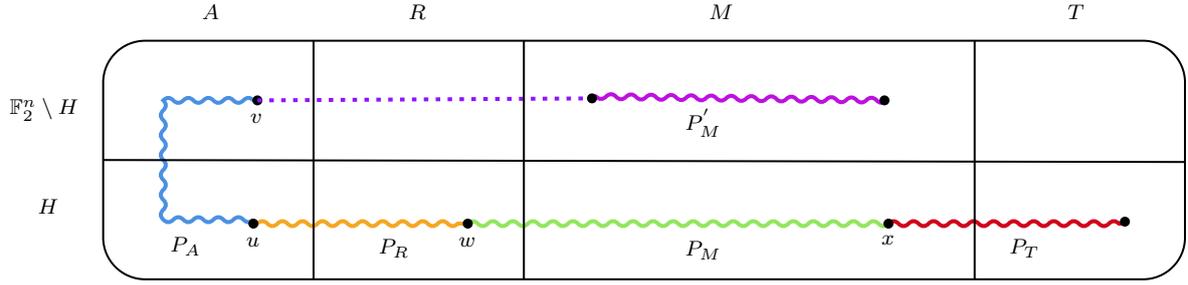
\begin{figure}[h]
\centering

\tikzset{every picture/.style={line width=0.75pt}} 

\begin{tikzpicture}[x=0.75pt,y=0.75pt,yscale=-1,xscale=1]

\draw [color={rgb, 255:red, 74; green, 144; blue, 226 }  ,draw opacity=1 ][line width=1.5]    (157,120) .. controls (155.33,121.67) and (153.67,121.67) .. (152,120) .. controls (150.33,118.33) and (148.67,118.33) .. (147,120) .. controls (145.33,121.67) and (143.67,121.67) .. (142,120) .. controls (140.33,118.33) and (138.67,118.33) .. (137,120) .. controls (135.33,121.67) and (133.67,121.67) .. (132,120) .. controls (130.33,118.33) and (128.67,118.33) .. (127,120) .. controls (125.33,121.67) and (123.67,121.67) .. (122,120) .. controls (120.33,118.33) and (118.67,118.33) .. (117,120) .. controls (115.33,121.67) and (113.67,121.67) .. (112,120) -- (112,120) .. controls (110.33,118.33) and (110.33,116.67) .. (112,115) .. controls (113.67,113.33) and (113.67,111.67) .. (112,110) .. controls (110.33,108.33) and (110.33,106.67) .. (112,105) .. controls (113.67,103.33) and (113.67,101.67) .. (112,100) .. controls (110.33,98.33) and (110.33,96.67) .. (112,95) .. controls (113.67,93.33) and (113.67,91.67) .. (112,90) .. controls (110.33,88.33) and (110.33,86.67) .. (112,85) .. controls (113.67,83.33) and (113.67,81.67) .. (112,80) .. controls (110.33,78.33) and (110.33,76.67) .. (112,75) .. controls (113.67,73.33) and (113.67,71.67) .. (112,70) .. controls (110.33,68.33) and (110.33,66.67) .. (112,65) .. controls (113.67,63.33) and (113.67,61.67) .. (112,60) -- (112,60) .. controls (113.67,58.33) and (115.33,58.33) .. (117,60) .. controls (118.67,61.67) and (120.33,61.67) .. (122,60) .. controls (123.67,58.33) and (125.33,58.33) .. (127,60) .. controls (128.67,61.67) and (130.33,61.67) .. (132,60) .. controls (133.67,58.33) and (135.33,58.33) .. (137,60) .. controls (138.67,61.67) and (140.33,61.67) .. (142,60) .. controls (143.67,58.33) and (145.33,58.33) .. (147,60) .. controls (148.67,61.67) and (150.33,61.67) .. (152,60) .. controls (153.67,58.33) and (155.33,58.33) .. (157,60) -- (159,60) -- (159,60) ;
\draw [color={rgb, 255:red, 138; green, 227; blue, 80 }  ,draw opacity=0.92 ][line width=1.5]    (264,122) .. controls (265.67,120.33) and (267.33,120.33) .. (269,122) .. controls (270.67,123.67) and (272.33,123.67) .. (274,122) .. controls (275.67,120.33) and (277.33,120.33) .. (279,122) .. controls (280.67,123.67) and (282.33,123.67) .. (284,122) .. controls (285.67,120.33) and (287.33,120.33) .. (289,122) .. controls (290.67,123.67) and (292.33,123.67) .. (294,122) .. controls (295.67,120.33) and (297.33,120.33) .. (299,122) .. controls (300.67,123.67) and (302.33,123.67) .. (304,122) .. controls (305.67,120.33) and (307.33,120.33) .. (309,122) .. controls (310.67,123.67) and (312.33,123.67) .. (314,122) .. controls (315.67,120.33) and (317.33,120.33) .. (319,122) .. controls (320.67,123.67) and (322.33,123.67) .. (324,122) .. controls (325.67,120.33) and (327.33,120.33) .. (329,122) .. controls (330.67,123.67) and (332.33,123.67) .. (334,122) .. controls (335.67,120.33) and (337.33,120.33) .. (339,122) .. controls (340.67,123.67) and (342.33,123.67) .. (344,122) .. controls (345.67,120.33) and (347.33,120.33) .. (349,122) .. controls (350.67,123.67) and (352.33,123.67) .. (354,122) .. controls (355.67,120.33) and (357.33,120.33) .. (359,122) .. controls (360.67,123.67) and (362.33,123.67) .. (364,122) .. controls (365.67,120.33) and (367.33,120.33) .. (369,122) .. controls (370.67,123.67) and (372.33,123.67) .. (374,122) .. controls (375.67,120.33) and (377.33,120.33) .. (379,122) .. controls (380.67,123.67) and (382.33,123.67) .. (384,122) .. controls (385.67,120.33) and (387.33,120.33) .. (389,122) .. controls (390.67,123.67) and (392.33,123.67) .. (394,122) .. controls (395.67,120.33) and (397.33,120.33) .. (399,122) .. controls (400.67,123.67) and (402.33,123.67) .. (404,122) .. controls (405.67,120.33) and (407.33,120.33) .. (409,122) .. controls (410.67,123.67) and (412.33,123.67) .. (414,122) .. controls (415.67,120.33) and (417.33,120.33) .. (419,122) .. controls (420.67,123.67) and (422.33,123.67) .. (424,122) .. controls (425.67,120.33) and (427.33,120.33) .. (429,122) .. controls (430.67,123.67) and (432.33,123.67) .. (434,122) .. controls (435.67,120.33) and (437.33,120.33) .. (439,122) .. controls (440.67,123.67) and (442.33,123.67) .. (444,122) .. controls (445.67,120.33) and (447.33,120.33) .. (449,122) .. controls (450.67,123.67) and (452.33,123.67) .. (454,122) .. controls (455.67,120.33) and (457.33,120.33) .. (459,122) .. controls (460.67,123.67) and (462.33,123.67) .. (464,122) .. controls (465.67,120.33) and (467.33,120.33) .. (469,122) .. controls (470.67,123.67) and (472.33,123.67) .. (474,122) -- (474,122) ;
\draw  [fill={rgb, 255:red, 0; green, 0; blue, 0 }  ,fill opacity=1 ] (157,60) .. controls (157,61.1) and (157.9,62) .. (159,62) .. controls (160.1,62) and (161,61.1) .. (161,60) .. controls (161,58.9) and (160.1,58) .. (159,58) .. controls (157.9,58) and (157,58.9) .. (157,60) -- cycle ;
\draw [color={rgb, 255:red, 208; green, 2; blue, 27 }  ,draw opacity=1 ][line width=1.5]    (474,122) .. controls (475.67,120.33) and (477.33,120.33) .. (479,122) .. controls (480.67,123.67) and (482.33,123.67) .. (484,122) .. controls (485.67,120.33) and (487.33,120.33) .. (489,122) .. controls (490.67,123.67) and (492.33,123.67) .. (494,122) .. controls (495.67,120.33) and (497.33,120.33) .. (499,122) .. controls (500.67,123.67) and (502.33,123.67) .. (504,122) .. controls (505.67,120.33) and (507.33,120.33) .. (509,122) .. controls (510.67,123.67) and (512.33,123.67) .. (514,122) .. controls (515.67,120.33) and (517.33,120.33) .. (519,122) .. controls (520.67,123.67) and (522.33,123.67) .. (524,122) .. controls (525.67,120.33) and (527.33,120.33) .. (529,122) .. controls (530.67,123.67) and (532.33,123.67) .. (534,122) .. controls (535.67,120.33) and (537.33,120.33) .. (539,122) .. controls (540.67,123.67) and (542.33,123.67) .. (544,122) .. controls (545.67,120.33) and (547.33,120.33) .. (549,122) .. controls (550.67,123.67) and (552.33,123.67) .. (554,122) .. controls (555.67,120.33) and (557.33,120.33) .. (559,122) .. controls (560.67,123.67) and (562.33,123.67) .. (564,122) .. controls (565.67,120.33) and (567.33,120.33) .. (569,122) .. controls (570.67,123.67) and (572.33,123.67) .. (574,122) .. controls (575.67,120.33) and (577.33,120.33) .. (579,122) .. controls (580.67,123.67) and (582.33,123.67) .. (584,122) .. controls (585.67,120.33) and (587.33,120.33) .. (589,122) -- (592,122) -- (592,122) ;
\draw [color={rgb, 255:red, 144; green, 19; blue, 254 }  ,draw opacity=1 ][line width=1.5]  [dash pattern={on 1.69pt off 2.76pt}]  (159,60) -- (326,59) ;
\draw [color={rgb, 255:red, 245; green, 166; blue, 35 }  ,draw opacity=1 ][line width=1.5]    (159,122) .. controls (160.67,120.33) and (162.33,120.33) .. (164,122) .. controls (165.67,123.67) and (167.33,123.67) .. (169,122) .. controls (170.67,120.33) and (172.33,120.33) .. (174,122) .. controls (175.67,123.67) and (177.33,123.67) .. (179,122) .. controls (180.67,120.33) and (182.33,120.33) .. (184,122) .. controls (185.67,123.67) and (187.33,123.67) .. (189,122) .. controls (190.67,120.33) and (192.33,120.33) .. (194,122) .. controls (195.67,123.67) and (197.33,123.67) .. (199,122) .. controls (200.67,120.33) and (202.33,120.33) .. (204,122) .. controls (205.67,123.67) and (207.33,123.67) .. (209,122) .. controls (210.67,120.33) and (212.33,120.33) .. (214,122) .. controls (215.67,123.67) and (217.33,123.67) .. (219,122) .. controls (220.67,120.33) and (222.33,120.33) .. (224,122) .. controls (225.67,123.67) and (227.33,123.67) .. (229,122) .. controls (230.67,120.33) and (232.33,120.33) .. (234,122) .. controls (235.67,123.67) and (237.33,123.67) .. (239,122) .. controls (240.67,120.33) and (242.33,120.33) .. (244,122) .. controls (245.67,123.67) and (247.33,123.67) .. (249,122) .. controls (250.67,120.33) and (252.33,120.33) .. (254,122) .. controls (255.67,123.67) and (257.33,123.67) .. (259,122) -- (262,122) -- (262,122) ;
\draw   (82,51) .. controls (82,39.4) and (91.4,30) .. (103,30) -- (601,30) .. controls (612.6,30) and (622,39.4) .. (622,51) -- (622,129) .. controls (622,140.6) and (612.6,150) .. (601,150) -- (103,150) .. controls (91.4,150) and (82,140.6) .. (82,129) -- cycle ;
\draw    (82,90) -- (622,91) ;
\draw    (187,30) -- (187,150) ;
\draw    (292,30) -- (292,150) ;
\draw    (517,30) -- (517,150) ;
\draw  [fill={rgb, 255:red, 0; green, 0; blue, 0 }  ,fill opacity=1 ] (155,122) .. controls (155,123.1) and (155.9,124) .. (157,124) .. controls (158.1,124) and (159,123.1) .. (159,122) .. controls (159,120.9) and (158.1,120) .. (157,120) .. controls (155.9,120) and (155,120.9) .. (155,122) -- cycle ;
\draw  [fill={rgb, 255:red, 0; green, 0; blue, 0 }  ,fill opacity=1 ] (262,122) .. controls (262,123.1) and (262.9,124) .. (264,124) .. controls (265.1,124) and (266,123.1) .. (266,122) .. controls (266,120.9) and (265.1,120) .. (264,120) .. controls (262.9,120) and (262,120.9) .. (262,122) -- cycle ;
\draw  [fill={rgb, 255:red, 0; green, 0; blue, 0 }  ,fill opacity=1 ] (472,122) .. controls (472,123.1) and (472.9,124) .. (474,124) .. controls (475.1,124) and (476,123.1) .. (476,122) .. controls (476,120.9) and (475.1,120) .. (474,120) .. controls (472.9,120) and (472,120.9) .. (472,122) -- cycle ;
\draw [color={rgb, 255:red, 189; green, 16; blue, 224 }  ,draw opacity=1 ][line width=1.5]    (473,60) .. controls (471.31,61.64) and (469.65,61.62) .. (468,59.93) .. controls (466.35,58.24) and (464.69,58.22) .. (463,59.86) .. controls (461.31,61.5) and (459.65,61.48) .. (458,59.79) .. controls (456.35,58.1) and (454.69,58.08) .. (453,59.73) .. controls (451.31,61.37) and (449.65,61.35) .. (448,59.66) .. controls (446.35,57.97) and (444.69,57.95) .. (443,59.59) .. controls (441.31,61.23) and (439.65,61.21) .. (438,59.52) .. controls (436.35,57.83) and (434.69,57.81) .. (433,59.45) .. controls (431.31,61.09) and (429.65,61.07) .. (428,59.38) .. controls (426.35,57.69) and (424.69,57.67) .. (423,59.32) .. controls (421.31,60.96) and (419.65,60.94) .. (418.01,59.25) .. controls (416.36,57.56) and (414.7,57.54) .. (413.01,59.18) .. controls (411.32,60.82) and (409.66,60.8) .. (408.01,59.11) .. controls (406.36,57.42) and (404.7,57.4) .. (403.01,59.04) .. controls (401.32,60.68) and (399.66,60.66) .. (398.01,58.97) .. controls (396.36,57.28) and (394.7,57.26) .. (393.01,58.9) .. controls (391.32,60.55) and (389.66,60.53) .. (388.01,58.84) .. controls (386.36,57.15) and (384.7,57.13) .. (383.01,58.77) .. controls (381.32,60.41) and (379.66,60.39) .. (378.01,58.7) .. controls (376.36,57.01) and (374.7,56.99) .. (373.01,58.63) .. controls (371.32,60.27) and (369.66,60.25) .. (368.01,58.56) .. controls (366.36,56.87) and (364.7,56.85) .. (363.01,58.49) .. controls (361.32,60.13) and (359.66,60.11) .. (358.01,58.42) .. controls (356.36,56.73) and (354.7,56.71) .. (353.01,58.36) .. controls (351.32,60) and (349.66,59.98) .. (348.01,58.29) .. controls (346.36,56.6) and (344.7,56.58) .. (343.01,58.22) .. controls (341.32,59.86) and (339.66,59.84) .. (338.01,58.15) .. controls (336.36,56.46) and (334.7,56.44) .. (333.01,58.08) .. controls (331.32,59.72) and (329.66,59.7) .. (328.01,58.01) -- (327,58) -- (327,58) ;
\draw  [fill={rgb, 255:red, 0; green, 0; blue, 0 }  ,fill opacity=1 ] (470,60) .. controls (470,61.1) and (470.9,62) .. (472,62) .. controls (473.1,62) and (474,61.1) .. (474,60) .. controls (474,58.9) and (473.1,58) .. (472,58) .. controls (470.9,58) and (470,58.9) .. (470,60) -- cycle ;
\draw  [fill={rgb, 255:red, 0; green, 0; blue, 0 }  ,fill opacity=1 ] (590,121) .. controls (590,122.1) and (590.9,123) .. (592,123) .. controls (593.1,123) and (594,122.1) .. (594,121) .. controls (594,119.9) and (593.1,119) .. (592,119) .. controls (590.9,119) and (590,119.9) .. (590,121) -- cycle ;
\draw  [fill={rgb, 255:red, 0; green, 0; blue, 0 }  ,fill opacity=1 ] (324,59) .. controls (324,60.1) and (324.9,61) .. (326,61) .. controls (327.1,61) and (328,60.1) .. (328,59) .. controls (328,57.9) and (327.1,57) .. (326,57) .. controls (324.9,57) and (324,57.9) .. (324,59) -- cycle ;

\draw (371,129.4) node [anchor=north west][inner sep=0.75pt]  [font=\footnotesize]  {$P_{M}$};
\draw (383,10.4) node [anchor=north west][inner sep=0.75pt]  [font=\footnotesize]  {$M$};
\draw (48,108.4) node [anchor=north west][inner sep=0.75pt]  [font=\footnotesize]  {$H$};
\draw (34,57.4) node [anchor=north west][inner sep=0.75pt]  [font=\footnotesize]  {$\mathbb{F}_{2}^{n} \setminus H$};
\draw (130,10.4) node [anchor=north west][inner sep=0.75pt]  [font=\footnotesize]  {$A$};
\draw (233,10.4) node [anchor=north west][inner sep=0.75pt]  [font=\footnotesize]  {$R$};
\draw (562,10.4) node [anchor=north west][inner sep=0.75pt]  [font=\footnotesize]  {$T$};
\draw (371,62.4) node [anchor=north west][inner sep=0.75pt]  [font=\footnotesize]  {$P_{M}^{'}$};
\draw (218,128.4) node [anchor=north west][inner sep=0.75pt]  [font=\footnotesize]  {$P_{R}$};
\draw (114,127.4) node [anchor=north west][inner sep=0.75pt]  [font=\footnotesize]  {$P_{A}$};
\draw (533,128.4) node [anchor=north west][inner sep=0.75pt]  [font=\footnotesize]  {$P_{T}$};
\draw (154,65.4) node [anchor=north west][inner sep=0.75pt]  [font=\footnotesize]  {$v$};
\draw (152,127.4) node [anchor=north west][inner sep=0.75pt]  [font=\footnotesize]  {$u$};
\draw (258,127.4) node [anchor=north west][inner sep=0.75pt]  [font=\footnotesize]  {$w$};
\draw (469,126.4) node [anchor=north west][inner sep=0.75pt]  [font=\footnotesize]  {$x$};

\end{tikzpicture}

\captionsetup{width=0.98\linewidth}
\caption{Illustration of the rainbow path constructed in the proof of \Cref{thm:dense}. The dashed line indicates that the path $P_M'$ does not intersect $R$. Colours indicate the different segments of the path, which of course are all rainbow with disjoint colour sets (except for $P_T$ and $P_R$, whose colours are disjoint only after we activate the appropriate gadgets to replace $P_R$ by $P_R'$).  
The picture depicts the case $\sum J \notin H$; the other case looks only marginally different ($P_T$ might ``jump'' once from $T \cap H$ to $T \setminus H$ at the point where it uses an edge of the colour that we removed from $P_A$). 
$P_A$ uses the colours $J \cup S_F'$; $P_M'$ uses the colours $S_1 \cup S_F''$, $P_R$ uses the colours $S_F$; $P_M$ uses the colours $S_0 \setminus (S_F \cup S_F' \cup S_F'' \cup L)$; and $P_T$ uses all but at most one of the colours of $L \cup \bigcup_{F\in \mathcal{F}'} F$.
}

\label{fig:dense_f2n_proof}
\end{figure}

\begin{proof}
We apply our regularity result \Cref{cor:robexpander} (with $\eps$ replaced by $2\eps^{11}$ and $\sigma=\eps$) to find a subspace $H$ such that $|S \cap H|\ge (1-2\eps^{11})|S|$ and 
$\Cayley_H(S \cap H)$ has no $\eps^{12}$-sparse cuts. 
Identify $H\cong \mathbb{F}_2^m$ and set $J:=S \setminus H$; notice that
\begin{equation}\label{eq:J-bound}
     |J| \le 2\eps^{11} |S|.
\end{equation}

We now partition $S$ into a large set set $S_0$ and a small set $S_1$; how we do so depends on the proportion of $H$ occupied by $S$.
If $|S \cap H| \le (1-\eps^3)|H|$, then we set $S_0:=S \cap H$ and $S_1:=\emptyset$.  If instead $|S \cap H| > (1-\eps^3)|H|$, then we let $S_0$ be an arbitrary subset of $S \cap H$ of size $(1-\varepsilon^3)|H|$ and set $S_1:=(S \cap H) \setminus S_0$.  Notice that either way $|S_1|\le \eps^3|H|$ and the graph $\Cayley_H(S_0)$ has no $\eps^{12}$-sparse cuts\footnote{This is clear when $S_1=\emptyset$, and otherwise the large size of $S_0$ gives the stronger property that $\Cayley_H(S_0)$ has no $\frac14$-sparse cuts.}.  Notice also that when $S_1 \neq \emptyset$, we may assume that $J \neq \emptyset$ since the $J=\emptyset$ case reduces to the situation already handled by \Cref{thm:very-dense}.

So far, we have a partition $S=S_0 \cup S_1 \cup J$ such that
\begin{equation}\label{eq:Si_sizes}
    |S_0| \le (1-\eps^{3}) |H|, \quad |S_1|\le \eps^3|H|,\quad \text{ and } \quad\Cayley_H(S_0)\text{ has no }\eps^{12}\text{-sparse cuts.}
\end{equation}

Set $p:=\eps^{4}$.
Let $S',E_1,E_2$ be disjoint $\frac14$-random subsets of $S_0$, and let $A \sqcup R \sqcup M \sqcup T$ be a random partition of $\Fn$ where each vertex is (independently) assigned to $A,R,M,T$ with probabilities $p,p,1-3p,p$, respectively.  We will activate the lemmas from the previous two sections to show that five desirable properties hold with high probability, and then we will show how to use these properties to find the desired rainbow path in $\Cayley_{\Fn}(S)$.

We now apply \Cref{lem:asymptoticinrandom-dense} with $S=S_0, \:M=M, \: S'=S' \cup E_1 \cup E_2, \: J=\emptyset, \: \mathbb{F}_2^m \cong H$ and parameters 
\begin{equation}\label{eq:parameters}
    \quad \zeta=\eps^{12},\quad q=1-3p, \quad \mu = \eps p^{15}/2^{28}, \quad q'=3/4,
\end{equation} 
and we replace $\eps$ in \Cref{lem:asymptoticinrandom-dense} by $\eps/2$. The hypotheses of the lemma are satisfied since 
$$|S_0|= |S|-|S_1|-|J| \overset{\eqref{eq:Si_sizes}}{\ge} (\eps-\eps^3-2\eps^{11})|H|\ge \frac{\eps}{2} |H|$$ 

and $q\ge  (1+\mu)|S_0|/|H|$ (which holds since $|S_0| \overset{\eqref{eq:Si_sizes}}{\le} (1-\eps^{3})|H|$, while $q=1-3p=1-3\eps^4$) and $q' \le 1-\mu q/4.$  
The conclusion of the lemma tells us that with high probability we have:
\begin{enumerate}[label = {{{\textbf{Q\arabic{enumi}}}}}] \setcounter{enumi}{0}
    \item\label{Q1} For any $S_F \subseteq S' \cup E_1 \cup E_2$ and any distinct vertices $x,w \in H$, we can find a rainbow path in $\Cayley_H(S_0 \setminus S_F)$ from $x$ to $w$, with all other vertices in $M$, such that the path uses all but at most $\mu q$ of the colours of $S_0 \setminus S_F$.
\end{enumerate}

The next two properties depend on the random set $S'$. If $|S'| \le |S_0|/8$, which by a Chernoff bound occurs with probability $o(1)$, then we declare both properties to fail. 
So we will work only with outcomes where $|S'| \ge |S_0|/8\ge |S|/16$. In this case, \Cref{lem:good-family} (with $\eps=1$) produces a flexible family $\mathcal{F}$ of gadgets in $S'$ of size 
\begin{equation}\label{eq:size_of_F}
    |\mathcal{F}|=p^8|S|/2^{16},
\end{equation} 
since this is smaller than $|S|/2^{54}\le |S'|/2^{50}$. 

\Cref{lem:absorbing-path} with $\mathcal{F}, E=S', U=\emptyset$, and $R$ (which is allowed since $|\mathcal{F}| \overset{\eqref{eq:size_of_F}}{=} p^8|S|/2^{16}\le|S'|p^8/2^{12}$) tells us that with high probability we have: 
\begin{enumerate}[label = {{{\textbf{Q\arabic{enumi}}}}}] \setcounter{enumi}{1}
    \item\label{Q2}  For each vertex $u \in H$, there is an $\mathcal{F}$-absorbing rainbow path in $\Cayley_H(S')$ starting from $u$ and otherwise contained in $R$.
\end{enumerate}

\Cref{lem:tails} with $\mathcal{F}, U=\emptyset, S$ and $T$ (which is allowed since $|\mathcal{F}| \ge 2^{12}p^{-7} \log N$) tells us that  with high probability we have:
\begin{enumerate}[label = {{{\textbf{Q\arabic{enumi}}}}}] \setcounter{enumi}{2}
    \item\label{Q3}  For any $L \subseteq S$ of size $|L| \le  \mu q N+1$ and any vertex $v \in \Fn$, there is some $\mathcal{F}'\subseteq \mathcal{F}$ such that $\Cay{L \cup \bigcup_{F \in \mathcal{F}'}F}$ has a rainbow path that starts at $v$, is otherwise contained in $T$, and uses all except possibly one of the colours from $L \cup \bigcup_{F \in \mathcal{F}'}F$.
\end{enumerate} 
To check that we may indeed take $|L|$ up to $\mu qN+1$, note that $\mu q N +1\overset{\eqref{eq:parameters}}{\le}\eps p^{15}N/2^{28} \le p^{15}|S|/2^{28}\overset{\eqref{eq:size_of_F}}{=} |\mathcal{F}|p^7/2^{12}.$

As above, a Chernoff bound tells us that with high probability $|E_1|\ge \frac18|S_0| \ge\frac{1}{16}|S| \overset{\eqref{eq:J-bound}}{\ge}  p^{-2}\cdot \max\{40|J|,96 \log N\}$. In this case we apply 
\Cref{lem:absorbing-junk} with $G=\Fn$, $J, E=J \cup E_1$, and our $p$-random subset $A$ to conclude that with high probability we have:
\begin{enumerate}[label = {{{\textbf{Q\arabic{enumi}}}}}] \setcounter{enumi}{3}
    \item\label{Q4}  For each vertex $u \in \Fn$, there is a rainbow path in $\Cay{E_1 \cup J}$, using all of the colours from $J$, that starts at $u$ and is otherwise contained in $A$.
\end{enumerate}

Again by Chernoff's bound, we have that with high probability $|E_2|(1-3p)^2\ge \frac{1}{16}|S| \ge \max\{40|S_1|,96 \log N\}$.  In this case, another application of \Cref{lem:absorbing-junk}, this time with $G=H$, $J=S_1, E=S_1 \cup E_2$, and our $(1-3p)$-random subset $M$, tells us that with high probability we have:
\begin{enumerate}[label = {{{\textbf{Q\arabic{enumi}}}}}] \setcounter{enumi}{4}
    \item\label{Q5}  For each vertex $u \in \Fn$, there is a rainbow path in $\Cay{S_1 \cup E_2}$, using all of the colours from $S_1$, that starts at $u$ and is otherwise contained in $M$.
\end{enumerate}

Consider now an outcome for which the properties \ref{Q1}-\ref{Q5} all hold.  Fix a vertex $u \in H$.

First, we deal with the junk set $J$.  Using \ref{Q4}, we find a rainbow path $P_A$, starting at $u$ and otherwise contained in $A$, such that $P_A$ uses all of the colours from $J$ and some subset $S_F'$ of the colours of $E_1$.  Among all such paths, choose one of minimal length.  Let $v$ denote the last vertex in $P_A$. 

Next, we use \ref{Q5} to handle the set $S_1$.  If $S_1=\emptyset$, then let $P'_M$ be the trivial $1$-vertex path at $v$ and set $S_F'':=\emptyset$.  Now suppose that $S_1 \neq \emptyset$, and recall that in this case we have $J \neq \emptyset$.  If $\sum J \notin H$, then we find a rainbow path $P_M'$, starting at $v$ and otherwise contained in $M$, such that $P'_M$ uses all of the colours from $S_1$ and some subset $S_F''$ of the colours of $E_2$.  Notice that $P'_M$ is entirely contained in a single proper $H$-coset since its starting vertex $v$ is not in $H$ and all of its edges have colours in $H$. 
It remains to consider the case where $\sum J \in H$.
The minimality of $P_A$ guarantees that the last edge of $P_A$ uses some colour $j^* \in J$.  Now truncate $P_A$ by removing $v$, and let $v'$ denote its new final vertex.  Now use \ref{Q5} as above but with $v$ replaced by $v'$ and with $J$ replaced by $J \setminus \{j^*\}$; again note that the resulting $P'_M$ is contained in a proper $H$-coset.  We remark that we will later build another path $P_M$ contained in $M \cap H$, and it will automatically be disjoint from $P'_M$ even though both paths live in the same random subset $M$; this ability to jump between $H$-cosets is the key benefit of working in the setting $J \neq \emptyset$.

We now use \ref{Q2} to find an $\mathcal{F}$-absorbing rainbow path $P_R$, ending\footnote{Strictly speaking, we take the reverse of the path produced by \ref{Q2}.} at $u$ and otherwise contained in $R$, whose colour set is some $S_F \subseteq S'$. Let $w$ denote the first vertex of $P_R$. We have $w\in H$ since $u \in H$ and $S' \subseteq H$.

Now \ref{Q1} gives us a rainbow path $P_M$ in $M$ from $x$ to $w$ which uses all of the colours of $S_0 \setminus (S_F \cup S_F' \cup S_F'')$ except for some set $L$ of size $|L| \leq \mu q N$ (note that \ref{Q1} applies since $S_F \subseteq S'$ and $S_F' \cup S_F'' \subseteq E_1 \cup E_2$). The path $P_M$ is fully contained in $H$ since $w \in H$ and $S_0 \subseteq H$; crucially, $P_M$ is disjoint from $P'_M$. If we removed the last edge of $P_A$ in our application of \ref{Q5}, then we add $j^*$ to $L$.  So far we have found a rainbow path $P_M' \cup P_A \cup P_R \cup P_M$ which avoids the vertex set $T$ and uses precisely the colours in $S \setminus L$.

Finally, by \ref{Q3} we can find a subfamily of gadgets $\mathcal{F}' \subseteq \mathcal{F}$ and a rainbow path $P_T$, ending at $x$ and otherwise contained in $T$, which uses all except possibly one colour of $L \cup \bigcup_{F \in \mathcal{F}'} F$.  Since $P_R$ is $\mathcal{F}$-absorbing, we may remove the gadgets in $\mathcal{F}'$ from it to obtain a shorter path $P_R'$. Now $P'_M\cup P_A \cup P_R' \cup P_M \cup P_T$ is a rainbow path using all but at most one colour from $S$, as desired.
\end{proof}

\section{The sparse case}\label{sec:sparse-F2n}

In this section we treat the sparse case of \Cref{thm:mainthm}. This takes the following shape.

\begin{theorem}\label{thm:sparsecase} 
There is a constant $\nu > 0$ such that every subset $S\subseteq \Fn\setminus\{0\}$ of size $|S|\leq \nu \cdot 2^n$ satisfying $S+S=\Fn$ has a valid ordering.
\end{theorem}

This theorem shows that sparse subsets $S$ of $\Fn$ have valid orderings as long as $S+S=\Fn$. The following simple lemma shows that this additional assumption is in fact not restrictive.

\begin{lemma}\label{lem:full-expansion}
    Let $S\subseteq \F_2^n$. If $S+S\neq \F_2^n$, then there exists a non-trivial quotient group $H$ of $\F_2^n$ such that the projection map $\pi\colon \F_2^n \to H$ is injective on $S$. In particular, $\pi(S)$ having a valid ordering in $H$ implies that $S$ has a valid ordering in $\F_2^n$.
\end{lemma}
\begin{proof}
    Let $v\in \F_2^n\setminus (S+S)$, and set $H:= \F_2^n/\langle v\rangle$. Now $\pi$ is injective on $S$ because otherwise we would have distinct $s_1,s_2 \in S$ with $\pi(s_1)=\pi(s_2)$, and then we would have $s_1+s_2=v$, which is impossible.  The second part of the lemma is obvious.
\end{proof}

Before turning to the proof of \Cref{thm:sparsecase}, let us confirm that \Cref{thm:sparsecase} and \Cref{thm:dense} indeed combine to establish \Cref{thm:mainthm}.  \Cref{lem:full-expansion} shows that it suffices to consider sets $S \subseteq \Fn$ with $S+S=\Fn$. Now \Cref{thm:sparsecase} handles the regime $|S| \leq \nu N$ (with $\nu$ as given by \Cref{thm:sparsecase}), and \Cref{thm:dense} handles the regime $|S| \geq \nu N$.

Our proof of \Cref{thm:sparsecase} splits into a ``structured'' (non-expanding) case and ``random-like'' (expanding) case. Recall that we say a subset $E \subseteq \Fn$ is \emph{$(\gamma,K)$-everywhere-expanding}
if every subset $E' \subseteq E$ of size $\gamma |E|$ satisfies $|E'+E'|\ge K|E'|$.

\subsection{The structured case}\label{subsec:structured}
We start with the non-expanding case since it will essentially reduce to (several interdependent instances of) the dense case and the argument is similar to what we saw in the previous section.

\subsubsection{Preliminaries}
We always work with a set $S\subset \Fn$ with $|S|\leq \nu \cdot 2^n$ and $S+S=\Fn$.  These assumptions already guarantee that $S$ has at least a bit of expansion. The following lemma lets us set aside a small, well-expanding reservoir of colours for later use. As usual, we omit floor and ceiling functions throughout.

\begin{lemma}\label{lem:colourreservoir}
    Let $S\subseteq \Fn$, and let $2/|S| \le \gamma\le 1$. Then there is a subset $X\subseteq S$ of size $|X|=\gamma|S|$  such that $|X+X|\geq \frac{\gamma^2}{2} |S+S|$.
\end{lemma}
\begin{proof} Take a uniformly random subset $X$ of the specified size. Each element of $S+S$ survives in $X+X$ with probability at least $\gamma \cdot \frac{\gamma|S|-1}{|S|-1} \ge \gamma^2/2$, so $X+X$ has size at least $\frac{\gamma^2}2 |S+S|$ in expectation.
\end{proof}

We will often use Ruzsa's triangle inequality to translate large doubling of $T+T$ into good expansion of $V+T$ for any other (reasonably large) subset $V$.

\begin{lemma}[Ruzsa triangle inequality]\label{lem:triangle-ineq}
    For any subsets $V,T$ of an abelian group, we have $|V+T|^2\geq |V|\cdot |T+T|$.
\end{lemma}

We also need a version of the Freiman--Ruzsa Theorem in $\Fn$.  An asymptotic formulation of the relevant result was first proven by Green and Tao \cite{green-tao}, and we will use the following version due to \cite{finite-field-freiman}. 

\begin{theorem}\label{lem:freiman--Ruzsa}  Let $K \geq 0$. If $T\subseteq \Fn$ satisfies $|T+T| \le K |T|$, then there is a subspace $H$ of $\Fn$ such that $T \subseteq H$ and $|H|\leq 2^{2K} |T|.$ 
\end{theorem}

If a set $S$ lacks everywhere-expanding subsets, then we can (almost) partition it into a small number of subsets with small doubling, and

each such subsets is dense in a (smaller) subspace. We will analyse most of the subsets within their respective subspaces. The following lemma will allow us to link up the resulting pieces. Here and in the subsequent theorem, one should think of $\nu, \gamma, K$ as constants where $K>0$ is sufficiently large in terms of $\gamma$ and $\nu>0$ is sufficiently small in terms of $K$. We work with concrete dependences among these constants to make the calculations easier to verify.

\begin{lemma} \label{lem:cosets}

Let $0<\nu, \gamma \le 1 \le K$ satisfy $2\nu^{1/48} \le 2^{1-K}  \le \gamma$.  Suppose $s,n \in \mathbb{N}$ are such that $8\gamma^{-2}\le s\leq \nu N$, where $N:=2^n$. Let $t\leq 2/\gamma$, and let $\{H_i\}_{i\in[t]}$ be a sequence of (not necessarily distinct) subspaces of $\Fn$, each of size between $\gamma s$ and $2^{2K}s$. If $X\subseteq \Fn \setminus \{0\}$ satisfies $|X|\leq \gamma s$ and $|X+X|\geq (\gamma^2/5)N$, then there exist $w_1,\ldots,w_t \in \Fn$ such that the following holds with $W_i:=w_i+H_i$:
\begin{enumerate}
    \item For each $i\in [t]$, we have $|W_i\cap \bigcup_{\ell\in[t]\setminus\{i\}}W_{\ell}|\leq \nu^{1/4}s.$ 
    \item There is a sequence of distinct elements $x_1,y_1,\ldots, x_{t-1},y_{t-1}\in \Fn$ such that $x_i\in W_i, y_i \in W_{i+1}$, and $x_i+y_i \in X$ (i.e., $(x_i,y_i)$ is an edge in $\Cayley_{\Fn}(X)$) for each $i \in [t-1]$, and the $x_i+y_i$'s are distinct.
\end{enumerate}
\end{lemma}

\begin{proof}
We find suitable elements $w_{i+1},x_i,y_i$ one value of $i$ at a time.
    Start with $w_1:=0$, so that $W_1=H_1$. Suppose we have already found $w_1,\ldots, w_{m},x_1,y_1,\ldots,x_{m-1},y_{m-1}$ 
    such that $$\left|W_i\cap \bigcup_{\ell\in[m]\setminus\{i\}}W_{\ell}\right|\leq (m+1-i)\nu^{1/4}\gamma s/2$$ for each $i \in [m]$ and the conditions in part (2) of the lemma statement are so far satisfied.

As long as $m<t$, we will find $w_{m+1},x_m,y_m$ preserving these conditions (with $m$ replaced by $m+1$).
    
    Set $X_{\mathrm{free}}:=X\setminus \{x_1+y_1,\ldots, x_{m-1}+y_{m-1}\}$.  Then 
    $$|X_{\textrm{free}}+X_{\textrm{free}}|\geq |X+X|-t|X| \geq (\gamma^2/5)N - t\gamma s\geq(\gamma^2/5)N-2\nu N\geq (\gamma^2/6)N $$
    since $\nu \le \gamma^2/60$ (with room to spare). 
    Likewise, set $W_{\mathrm{free}}:=W_m \setminus \{x_1,y_1,\ldots, x_{m-1},y_{m-1}\}$, so that $$|W_{\mathrm{free}}| \ge |W_m|-2t \ge \gamma s-4/\gamma\ge \gamma s/2.$$ 
    Now the Ruzsa triangle inequality (\Cref{lem:triangle-ineq}) gives
    $$|W_{\mathrm{free}}+X_{\mathrm{free}}|\geq \sqrt{|W_{\mathrm{free}}|} \sqrt{|X_{\mathrm{free}}+X_{\mathrm{free}}|}\geq \sqrt{\gamma s/2} \sqrt{(\gamma^2/6)N}\geq \sqrt{\frac{\gamma^3}{12\nu }}  s \geq \nu^{-2/5} s,$$
    where we used $N\ge s/\nu$ and $\nu \le \gamma^{15}/12^5$ (say).

Thus there are at least $\frac{\nu^{-2/5}s}{2^{2K} s}\geq \nu^{-1/3}$
    cosets of $H_{m+1}$ which intersect $W_{\mathrm{free}}+X_{\mathrm{free}}$.  At most $2m-2$ of these cosets intersect $\{x_1,y_1,\ldots, x_{m-1},y_{m-1}\}$, and at most $\frac{2^{3K+1}s}{\nu^{1/4}\gamma s}\le \nu^{-1/3}-2m$ of them contain at least $\nu^{1/4}\gamma s/2$ elements of $W_1 \cup \cdots \cup W_m$ (the last inequality uses $m\le t \le 2/\gamma$ and $\nu \le 2^{-36K-24}\gamma^{12}$).
    Hence there is a coset $W_{m+1}=w_{m+1}+H_{m+1}$ which intersects $W_{\mathrm{free}}+X_{\mathrm{free}}$ in some element $y_m \notin \{x_1,y_1,\ldots, x_{m-1},y_{m-1}\}$ and contains at most $\nu^{1/4}\gamma s/2$ elements of $W_1 \cup \cdots \cup W_m$. In particular, there is some $x_m \in W_{\mathrm{free}}$ such that $x_m+y_m \in X_{\mathrm{free}}$ (notice that $x_m \neq y_m$ since $0 \notin X_{\mathrm{free}}$); this choice of $x_m,y_m$ works for our induction.

    Once we reach $m=t$, we have $|W_i\cap \bigcup_{\ell\in[t]\setminus\{i\}}W_{\ell}|\leq t\nu^{1/4}\gamma s/2 \le \nu^{1/4}s$ for every $i \in [t]$, as desired. 
\end{proof}

\subsubsection{The main argument}
We are now ready to handle the fully-structured case.  The main idea is that if $S$ is ``fully-structured'', then we can decompose it into sets $S_1, \ldots, S_t$ of small doubling, each of which is fairly dense in some subspace.  We then obtain cosets $W_1, \ldots, W_t$ of these subspaces which are nearly disjoint, and in each $W_i$ we build a rainbow path that uses most of the colours of $S_i$; at the end we join up these short rainbow paths and absorb the ``junk set'' of hitherto-unused colours.  See Figure~\ref{fig:non-expanding}.

In the following theorem, we write $1/s,\nu \ll 1/K \ll \gamma \ll \alpha \ll 1$ to mean that $\alpha\in(0,1)$ is a sufficiently small constant; $\gamma$ is sufficiently small in terms of $\alpha$; $K$ is sufficiently large in terms of $\gamma$; and $\nu, 1/s$ are sufficiently small in terms of $K$.

\begin{theorem}\label{thm:non-expandingcase}
Suppose $1/s,\nu \ll 1/K \ll \gamma \ll \alpha \ll 1$.  Let $S \subseteq \Fn \setminus \{0\}$ be a set of size $s:=|S|\le \nu N$ (where $N:=2^n$ as usual), and suppose that $S+S=\Fn$. If $S$ has no  $(\gamma/\alpha,K/\gamma)$-everywhere-expanding subset $E$ of size $|E|=\alpha s$, then $\Cay{S}$ has a rainbow path of length $|S|-1$. 
\end{theorem}

\begin{figure}[h]
    \centering
    \includegraphics[width=1\linewidth]{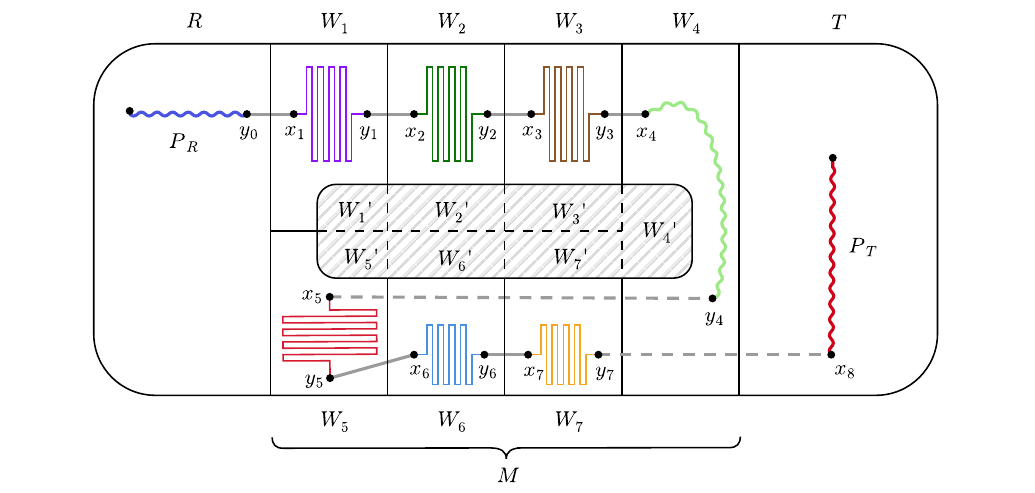}
    \caption{Illustration of the argument in \Cref{thm:non-expandingcase} with $t=4,m=3$. $W_1,W_5$; $W_2,W_6$; $W_3,W_7$ are pairs of cosets of the same subspace, and the $W_i$'s are chosen to have very small intersections $W_i'$ with the other $W_j$'s. We also have fixed ``connection points'' $y_0,x_1,y_1,\ldots, x_{7},y_7,x_8$. $P_R$ is an absorbing path built inside a random subset $R$. Each $P_{M,i}$ is a rainbow path from $x_i$ to $y_i$ built in the intersection of $W_i$ with a random set $M$, while avoiding $W_i'$ and all of the other connection points. The path $P_T$ uses up all of the unused colours, with the help of some gadgets activated in $P_R$.}
    \label{fig:non-expanding}
\end{figure}

\begin{proof}
By assumption, every subset of $S$ of size at least $\alpha s$ contains a subset of size $\gamma s$ with doubling constant at most $K$. We first extract $X\subseteq S$ of size $\gamma s$ having $|X+X| \ge \gamma^2 |S+S|/2=\gamma^2N/2$ by \Cref{lem:colourreservoir}. We can iteratively remove disjoint subsets $S_1,\ldots, S_t \subseteq S \setminus X$, each of size $\gamma s$ (so $t \le 1/\gamma$), such that each $|S_i+S_i| \le K |S_i|$ and $S_1 \cup \ldots \cup S_t$ covers all of $S \setminus X$ except for a set $J_0$ of size at most $\alpha s$. By \Cref{lem:freiman--Ruzsa}, this implies that there exist subspaces $H_i \supseteq S_i$ such that $|H_i| \le 2^{2K} |S_i|$. Next, we can apply the regularity-type result \Cref{cor:robexpander} to each $S_i$ (which has density at least $2^{-2K}$ inside of its $H_i$) to find a subset $S_i' \subseteq S_i$ of size at least $(1-\alpha)|S_i|$ and a subspace $H_i'$ of $H_i$ containing $S'_i$ such that $\Cayley_{H_i'}(S_i')$ has no $\alpha 2^{-2K-1}$-sparse cuts. We add the leftover elements $\bigcup S_i \setminus S_i'$ (there are at most $\sum_i \alpha |S_i| \le \alpha s$ such elements in total) to the set $J_0$ to obtain the set $J_1$, which has size $|J_1|\le 2\alpha s$.

Let us re-index so that the quantities $|S_i'|/|H_i'|$ are non-increasing with $i$.  Let $m$ be the largest index for which $|S_m'|/|H_m'|\ge 3/4$.
Next, we invoke \Cref{lem:cosets} with the sequence of subspaces $H_1',\ldots,  H_t',H_1',\ldots,H_m'$ (we list all of the subspaces in order and then list the first $m$ again) and the set $X$. This gives us cosets $W_1,\ldots, W_{t+m}$ and distinct elements $x_1,y_1,\ldots, x_{t+m-1},y_{t+m-1}\in \Fn$ such that each $W_i$ is a coset of $H_{i \bmod t}'$, each $W_i$ intersects the union of other $W_j$'s in at most $\nu^{1/4} s$ elements, each $x_i \in W_i, y_i \in W_{i+1},x_i+y_i \in X$, and the quantities $x_i+y_i$ are distinct. Our rainbow path will end up including the $t+m-1$ edges $(x_i,y_i)$ (with colours $x_i+y_i$).  As such, we mark the set of vertices $U:=\{x_1,y_1,\ldots,x_{t+m-1},y_{t+m-1}\}$ as ``already used''. We also add the unused colours from $X$, namely, $X \setminus \{x_1+y_1,\ldots,x_{t+m-1}+y_{t+m-1}\}$, to $J_1$ to obtain $J_2$; thus $J_2$ represents a ``junk set'' of colours that we will need to absorb into our rainbow path later. Write $W_i':=W_i \cap \bigcup_{j \neq i} W_j$ for each $i$, and note that $|W_i'|\le \nu^{1/4} s$.

Next, for each $i \in [t]$, let $E_i$ be a $\frac{1}{4}$-random subset of $S_i'$. 
Set $E:= \bigcup E_i.$ 

For each $i \le m$, take random subsets $S_{i,1}, S_{i,2}$ of $S'_i$ by assigning each element of $S'_i$ independently to $S_{i,1}$ with probability $\frac14$, to $S_{i,2}$ with probability $\frac14$, and to both $S_{i,1},S_{1,2}$ with probability $\frac12$.  So $S_{i,1},S_{i,2}$ are $\frac34$-random subsets of $S'_i$, and $S_{i,1} \cup S_{i,2}=S'_i$.  Although $S_{i,1},S_{i,2}$ are not independent, it is true that after we reveal $S_{i,1}$ (respectively, $S_{i,2}$), the intersection $S_{i,1} \cap S_{i,2}$ is a $\frac{2}{3}$-random subset of $S_{i,1}$ (respectively, $S_{i,2}$).

Set $p:=1/16$.  Let $R \sqcup M \sqcup T$ be random partition of $\Fn$, where each vertex is assigned to $R,M,T$ with probabilities $p,1-2p,p$, respectively.

For each $i$, we want to apply \Cref{lem:asymptoticinrandom-dense} in $W_i \cong \mathbb{F}_2^{n_i}$ with the sets
$$S=\begin{cases}
    S_{i,1} &\text{ if $i\le m$,}\\
    S_i' &\text{ if $m<i\le t$,}\\
    S_{i,2} &\text{ if $i>t$,}
\end{cases} \quad J=W_i' \cup U, \quad M=M \cap W_i, \quad S'=\begin{cases}
    E_i \cup (S_{i,1} \cap S_{i,2}) &\text{ if $i\le m$,}\\
    E_i &\text{ if $m<i\le t$,}\\
    E_i \cup (S_{i,1} \cap S_{i,2}) &\text{ if $i>t$.}
\end{cases}$$
Here we choose the parameters\footnote{It is possible to make ``tighter'' choices of parameters in some of the cases, but the choices listed here work for all cases and are tight enough for later applications.}
 $$q=1-2p, \quad \eps=2^{-2K-1}, \quad \zeta=\alpha 2^{-2K-1}, \quad \mu= \gamma \alpha 2^{-2K}, \quad q'=\begin{cases}
     1/4 &\text{ if $m<i \leq t$,}\\
     3/4, &\text{ otherwise,}
 \end{cases} \quad j=2\nu^{1/4}.$$
Let us say a brief word about the order in which the random sets are revealed.  For the case $i \leq m$, we first reveal $S_{i,1}$; at this point, the set $S_{i,1} \cap S_{i,2}$ is a $\frac23$-random subset of $S_{i,1}$, and we wish to apply \Cref{lem:asymptoticinrandom-dense} with this \emph{deterministic} choice of $S_{i,1}$ (as long as this set has the appropriate size, which is the case with high probability) and the random subset $S_{i,1} \cap S_{i,2}$.  The case $i>t$ goes the same way.  Let us check that the remaining hypotheses of the lemma hold with high probability (since there are only $2/\gamma=O(1)$ indices $i$, it suffices to check that the hypotheses hold with high probability for each index individually):
 \begin{itemize}
     \item We have $|S_{i,1}|,|S_{i,2}|,|S_i'| \ge 2^{-2K-1}|W_i|$, since $|S_i'|\ge (1-\alpha)|S_i|\ge (1-\alpha) 2^{-2K}|W_i|\ge \frac34 \cdot 2^{-2K}|W_i|$ (deterministically) and $|S_{i,1}|,|S_{i,2}|\ge \frac23 |S_i'|$ (with high probability by Chernoff's bound).
    \item The graphs $\Cayley_{W_i}(S_{i,1}), \Cayley_{W_i}(S_{i,2}), \Cayley_{W_i}(S_{i}')$ have no $\zeta$-cuts. This holds for $m<i \le t$ since Lemma~\ref{cor:robexpander} guarantees that $\Cayley_{W_i}(S_i')$ has no $\alpha2^{-2K-1}$ cuts. For $i \le m$ and $i >t$, a Chernoff bound guarantees that with high probability $|S_{i,1}|,|S_{i,2}| \ge \frac{17}{32} |W_i|$, since $S_{i,1},S_{i,2}$ are $\frac34$-random subsets of $S'_i$, which itself has size at least $\frac{3}{4}|H'_i|$; thus $\Cayley_{W_i}(S\setminus S')$ lacks $\frac1{32}$-sparse cuts just by density considerations. 
    \item We have $q|W_i|/(1+\mu)\ge |S_i'|$ if $m<i \le t$ (since $|S_i'|/|W_i| \le 3/4$ for such $i$), and $q|W_i|/(1+\mu) \ge |S_{i,1}|, |S_{i,2}|$ if $i \le m$ or $i >t$ (since  $|S_{i,j}|/|W_i| \le 5/6$ with high probability, again by Chernoff's bound).
    \item We have $q' \le 1-\mu q/4$ with room to spare.
 \end{itemize}
 
Now \Cref{lem:asymptoticinrandom-dense} tells us that with high probability, for all $i$ we have:
\begin{enumerate}[label = {{{\textbf{Z\arabic{enumi}}}}}] \setcounter{enumi}{0}
    \item\label{W1} For every subset $S_F$ of $E_i$ (if $m<i \le t$) or of $E_i \cup (S_{i,1} \cap S_{i,2})$ (if $i\le m$, or $i>t$), we can find a rainbow path in $\Cayley_{\Fn}({S'_i})$, from $x_i$ to $y_i$, that has all of its internal vertices in  $M \setminus (W_i' \cup U)$ and uses all but at most $\mu q |W_i|$ of the colours of $S_i' \setminus S_F$ (if $m<i \le t$), of $S_{i,1} \setminus S_F$ (if $i\le m$), or of $S_{i,2} \setminus S_F$ (if $i > t$). 
\end{enumerate}

We now reveal the sets $E_1,\ldots, E_t$. Chernoff's bound implies that with high probability $|E_i| \ge \frac18|S_i'|$ for each $i$; suppose we are in such an outcome. Then $|E| \ge s/16 \ge N^{1/2}/16$ (since $s=|S|\ge N^{1/2}$ due to $S+S=\Fn$). 
\Cref{lem:good-family} (with $\eps=1/16$) provides a flexible family $\mathcal{F}$ of $s/2^{62}$ gadgets in $E$; note that the assumption of this lemma is satisfied since $s/2^{62}\le |E|/2^{58}$.

We now apply \Cref{lem:absorbing-path} with the random set $R$, the set $U$, the set $E$, and this flexible family $\mathcal{F}$. The hypotheses of the lemma are satisfied since $|\mathcal{F}| \ge 2/\gamma \ge |U|$ and $2^{13}\log N \le 2^{12}|\mathcal{F}|\le |E|p^8$. The lemma guarantees that with high probability we have:
\begin{enumerate}[label = {{{\textbf{Z\arabic{enumi}}}}}] \setcounter{enumi}{1}
    \item\label{W2} For any vertex $y_0$, there is a rainbow $\mathcal{F}$-absorbing path in $\Cay{E}$ that ends at $y_0$ and is otherwise contained in $R \setminus U$.
\end{enumerate}

Finally, apply \Cref{lem:tails} with the random set $T$, the set $U$ of already-used vertices, and our flexible family $\mathcal{F}$. To verify the hypotheses of the lemma, note that $|\mathcal{F}| \ge \max\{2^{12}p^{-7} \log N, 128|U|\}$ (by a huge margin). The lemma says that with high probability we have:
\begin{enumerate}[label = {{{\textbf{Z\arabic{enumi}}}}}] \setcounter{enumi}{2}
    \item\label{W3} For any set $L \subseteq S$ of at most $(t+m)\mu q 2^{2K}s +2\alpha s+\gamma s \le |\mathcal{F}|p^7/2^{12}$ colours and any vertex $x_{t+m}$, there are a sub-family $\mathcal{F}' \subseteq \mathcal{F}$ and a rainbow path in $\Cay{L \cup \bigcup_{F \in \mathcal{F'}} F}$ that starts at $x_{t+m}$, is otherwise contained in $T \setminus U$, and uses all except possibly one colour from $L \cup \bigcup_{F \in \mathcal{F'}} F$.   
\end{enumerate}

Consider an outcome where \ref{W1}-\ref{W3} hold.  We will deduce the existence of the desired rainbow path.
Using \ref{W2}, we find an $\mathcal{F}$-absorbing rainbow path $P_R$ that ends at 
some vertex $y_0$, is otherwise contained in $R \setminus U$, and uses the colours of some subset $S_F \subseteq E$. Next using \ref{W1} we find rainbow paths $P_{M,1},\ldots, P_{M,t+m}$ with internal vertices in $M \setminus U$, where each $P_i$ goes from $y_{i-1}$ to $x_i$ (with $x_{t+m}$ chosen arbitrarily) and uses all but some set $L_i$ of at most $\mu q |W_i|$ colours from $S_{i,1} \setminus S_F$ (for $i \le m$), from $S_{i}' \setminus S_F$ (for $m< i \le t$), and from $S_{i,2} \setminus (S_F \cup c(P_{i-t}))$ (for $i>t$). Let $L:= J_2 \cup \bigcup_{i=1}^{t+m} L_i$.

We can use the edges $(x_i,y_i)$ to join the paths $P_R,P_{M,1},\ldots, P_{M,t+m}$.  The result is a rainbow path that avoids $T \setminus U$ and uses precisely the colours in $S \setminus L$. Note also that $$|L|\le |J_2|+(t+m)\mu q \cdot \max_i|W_i|\le |J_1|+|X|+(t+m)\mu q2^{2K}s\le 2\alpha s+\gamma s+(t+m)\mu q 2^{2K}s,$$ so the hypothesis of \ref{W3} is satisfied.
Now \ref{W3} provides a subfamily $\mathcal{F}' \subseteq \mathcal{F}$ and a rainbow path $P_T$ that starts at $x_{t+m}$, is otherwise contained in $T\setminus U$, and uses all except possibly one colour from $L \cup \bigcup_{F \in \mathcal{F}'} F$. 
We use the $\mathcal{F}$-absorbing properties of $P_R$ to remove $\bigcup_{F \in \mathcal{F}'} F$ and pass to a shorter path $P_R'$ using only a subset of vertices of $P_R$. Now the concatenation of $P_R',P_{M,1},\ldots, P_{M,t+m}, P_T$ gives the desired a rainbow path using all but at most one colour from $S$. 
\end{proof}

\subsection{Expanding case}\label{subsec:expanding}
In this section we will show how to find a rainbow path of length $|S|-1$ in $\Cay{S}$ in the case where $S$ has a subset $E$ with suitable everywhere-expanding properties. In order to incorporate the final few colours at the end of our argument, we will need to replace absorbing paths with slightly larger structures.  Recall the definition of an in-spider from the discussion following \Cref{defn:gadget}.

\begin{definition}[Absorbing fork]
Let $\mathcal{F}$ be a flexible family of gadgets in $S \subseteq \Fn$.  A a corresponding \emph{absorbing fork} $(P,Q)$ consists of a (directed) $\mathcal{F}$-absorbing path $P$ and an in-spider $Q$ of $\mathcal{F}$ such that $Q$ is based at the initial endpoint of $P$ and otherwise $P,Q$ are disjoint. We refer to the final endpoint of $P$ as the \emph{final} vertex of the entire absorbing fork. 
\end{definition}

We now show how to robustly embed absorbing forks in $\Cay{S}$ as long as $S$ is large enough. 

\begin{lemma}\label{lem:absorbing-mop}
    Let $N=2^n$, and let $E \subseteq S \subseteq \Fn$ be such that $|S| \ge 2^{11} |E|$. Let $\mathcal{F}$ be a flexible family of subsets of $E$.  Then $\Cayley_{\Fn}(S)$ contains an $\mathcal{F}$-absorbing fork $(P,Q)$ with $|P| \le 8|\mathcal{F}|+1$.
\end{lemma}
\begin{proof}
Choose an element $c(F) \in S \setminus \bigcup \mathcal{F}$ for each $F \in \mathcal {F}$; since $|S|-|\bigcup \mathcal{F}|\ge |\mathcal{F}|$, we can ensure that these elements $c(F)$ are all distinct. \Cref{prop:eat-an-element} tells us that for each $F \in \mathcal{F}$, there is a rainbow path $P_F$ using precisely the colours in $F \cup c(F)$. 
It remains to concatenate these paths and join them to an in-spider of $\mathcal{F}$. 

To start, let $Q$ be the in-spider of $\mathcal{F}$ based at the vertex $0$.  We start our path $P$ at the vertex $0$ and iteratively extend the path by adding both a single edge (with a colour from $S$) and one of the rainbow paths $P_F$ from the previous paragraph.  At each step, we identify a hitherto-unincorporated gadget $F$ and consider extensions of our current path by a single edge with some colour $x \in S \setminus (\bigcup_{F \in \mathcal{F}} (F \cup \{c(F)\}))$ followed by $P_F$.  Note that $|S \setminus (\bigcup_{F \in \mathcal{F}} (F \cup \{c(F)\}))| \geq |S|-2|E|$.  The structure that we have already built has (crudely) at most $|Q|+2|E|< 4|E|$ vertices.  The new path that we wish to append has at most $8$ vertices, so there are fewer than $32|E|<|S|-2|E|$ choices of $x$ (i.e., translates of this new path) that cause collision, so we choose some $x \in S \setminus E$ that does not cause any collisions.
\end{proof}

We will build our long rainbow path in stages.  We will start with an absorbing fork, as guaranteed by \Cref{lem:absorbing-mop}, which embeds a large family of gadgets. In each subsequent step we will append a new short path to the final vertex of the fork at the cost of ``activating'' up to two gadgets. In order to iterate this multi-stage procedure, we need to make sure that the final vertex of the fork is not ``blocked'' by other vertices of the fork (see \Cref{prop:modelexpansion} from our proof overview for a model version of this argument). With this in mind, we say a final vertex $v$ of an $\mathcal{F}$-absorbing fork $(P,Q)$ is $t$-\emph{extendable} if at least $t$ of the paths of the out-spider of $\mathcal{F}$ based at $v$ are disjoint from $P \cup Q$ (except at $v$, of course).  See \Cref{fig:expanding case} for a picture of the proof strategy.  It will be notationally convenient to phrase the iteration argument in terms of analysing the properties of a suitable maximal substructure.

\par We the symbol $\ll$ in the following theorem as we used it in Theorem~\ref{thm:non-expandingcase}, viz., $x\ll y$ means that $x$ is sufficiently small with respect to $y$.

\begin{theorem}\label{thm:expandingcase}
    Suppose that $1/s \ll 1/K \ll \gamma \ll \alpha \ll 1$.  Let $S \subseteq \Fn \setminus \{0\}$ be a subset of size $s:=|S|\ge  \sqrt{N}/2$ (where as usual $N:=2^n$). If $S$ has a $(\gamma/\alpha,K/\gamma)$-everywhere-expanding subset $E$ of size $\alpha s$, then $\Cay{S}$ has a rainbow path  of length $|S|-1$. 
\end{theorem}

    \begin{figure}
        \centering
        \hspace*{5mm}\includegraphics[width=1\linewidth]{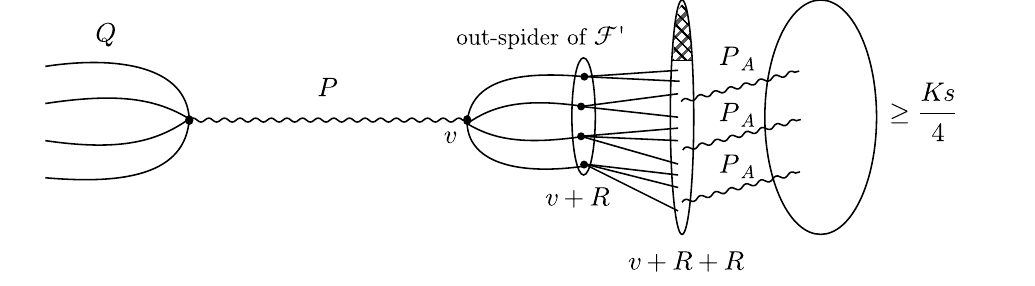}
        \caption{Illustration of the proof of \Cref{thm:expandingcase}. We take a maximal $\mathcal{F}$-absorbing fork $(P,Q)$ whose final vertex is suitably extendable. The everywhere-expansion of the set $E$ provides many potential ways to extend $P$ by following a leg of the out-spider, then taking one more step using a colour from a different gadget, and finally appending a translate of a short rainbow path $P_A$.  The maximality hypothesis ensures that in fact $P$ already used nearly all of the colours of $S$, and we can integrate the remaining colours using one of the legs of $Q$. 
        }
        \label{fig:expanding case}
    \end{figure}

\begin{proof}
    \Cref{lem:good-family} applied to $E$ provides a flexible family $\mathcal F_0$ of gadgets in $E$ with $|\mathcal{F}_0|= \alpha^2|E|/2^{52} = \alpha^3s/2^{52}$. Let $\beta:=\alpha^3/2^{52}$, so $|\mathcal F_0|=\beta s.$
    Now \Cref{lem:absorbing-mop} tells us that $\Cayley_{\Fn}(S)$ contains an $\mathcal{F}_0$-absorbing fork $(P_0,Q_0)$ with $|P_0| \le 9 |\mathcal{F}_0|= 9\beta s$; let $v_0$ denote its final vertex. 

    There are at least $s-9\beta s$ colours not appearing in $P_0$. Let $N_0$ denote the set of vertices in $\Fn \setminus (P_0 \cup Q_0)$ that are reachable from $v_0$ via single edge with one of these colours.  We have $|N_0| \ge s-24\beta s \ge s/2$. 
    Consider an auxiliary bipartite graph whose left side is the vertex set of $P_0 \cup Q_0$ and whose right side is $N_0$, with an edge between $a \in P_0 \cup Q_0$ and $b \in N_0$ if the out-spider of $\mathcal{F}_0$ based at $b$ contains $a$.  
    Equivalently, $a$ is adjacent to $b$ if the in-spider of $\mathcal{F}_0$ based at $a$ contains $b$, so the degree of each vertex on the left is at most $5\beta s$.  Recall that the left side has at most $15\beta s$ vertices, so the graph has at most $75 \beta^2 s^2$ edges in total. At the same time, $|N_0| \geq s/2$, so there is some $u \in N_0$ with degree at most $150\beta^2 s$ in the auxiliary bipartite graph. Extending $P_0$ to such a vertex gives a new absorbing fork whose final vertex $u$ is $\gamma s$-extendable (since $|\mathcal{F}_0|-150\beta^2 s \ge \gamma s$). 

    Now, let $(P,Q)$ be a maximal-size $\mathcal{F}$-absorbing fork such that its final vertex $v$ is $\gamma s$-extendable, $\mathcal{F} \subseteq \mathcal {F}_0$, and 
    $\frac12| \mathcal{F}_0 \setminus \mathcal{F}|\le |P|-(1-\gamma)\min \{|P|, s-8/\gamma\}$. 
    Note that $(P_0,Q_0)$ is such a fork for $\mathcal{F}_0$, so this is well-defined.  
    Let $A \subseteq S$ denote the subset of \emph{available} colours not used in $P$.  
    We claim that $\Cay{A}$ has a rainbow path $P_A$ using precisely $4/\gamma$ colours if $|A| \ge 8/\gamma$, and using $\min\{|A|,7\}$ colours from $A$ otherwise. If $|A| \ge 14$ this follows from the standard greedy argument which always allows us to find a path of length at least $|A|/2$ (see \cite[Observation 2.2]{towards-graham}). If $|A| < 14$, then we may simply use the fact from \cite{alspach2020strongly} that all sets of size at most $7$ have valid orderings. 

    Since $v$ is $\gamma s$-extendable, there is some $\mathcal{F}' \subseteq \mathcal{F}$ of size $\gamma s$ such that the out-spider of $\mathcal{F}'$ based at $v$ is disjoint from the vertices of $P \cup Q$ (other than $v$). Let $R$ denote the set of last elements of gadgets in $\mathcal{F}'$ (i.e., if we write each gadget as $F=\{f_1,\dots,f_{|F|}\}$, then $R=\{f_{|F|}:F\in \mathcal{F}'\}$).  Note that since each gadget is $0$-sum, the set of leaves of this spider is precisely $v+R$. The subset $R \subseteq E$ and has size $\gamma s$, so our everywhere-expanding assumption on $E$ ensures that $|R+R| \ge K s$. This means that if we take a second step with an edge with colour in $R$, then we can reach at least $Ks$ vertices in $v+R+R$. There are at least $Ks-1-s-5\beta s-5\gamma s\ge Ks/2$ such vertices which are not in $P \cup Q$ (which contains at most $1+s+4\beta s$ vertices) or in the out-spider of $\mathcal{F}'$ based at $v$ (which contains at most $5\beta s$ vertices). Further, as $P_A$ is a path using at most $8/\gamma$ colours, at most $(8/\gamma)\cdot(1+5\beta+5\gamma)s   \le Ks/4$ translations of the path $P_A$ starting at the vertices in $v+R+R$ can intersect $P \cup Q$ or the out-spider of $\mathcal{F}'$ based at $v$; here we used $K \gg \alpha, \gamma$. So at least $Ks/4$ such translates of $P_A$ will be disjoint from both $P \cup Q$ and the out-spider of $\mathcal{F}'$ based at $v$.
    
    Let $z=v+f_{|F_1|}+f_{|F_2|}\in v+R+R$, with $F_1,F_2\in\mathcal{F}'$, be such a \emph{good} vertex, in the sense that if we start at $z$ and follow the colours of $P_A$ in order, then we obtain a rainbow path whose vertices are disjoint from $P \cup Q$ and the  out-spider of $\mathcal{F}'$ based at $v$. Now we can extend the path $P$ in our current $\mathcal{F}$-absorbing fork $(P,Q)$ adding the edges with the colours from $F_1\setminus\{f_{|F_1|}\}$, then the edge with the colour $f_{|F_2|}$ (to reach $z$), and finally the translation of $P_A$. This produces a genuine path since $F_1\in\mathcal{F}'$ and the out-spider of $\mathcal{F}'$ based at $v$ is disjoint from $P\cup Q$, while the choice of $z$ guarantees that $z$ and the translation of $P_A$ do not cause collisions. This procedure has produced a new $\mathcal{F}\setminus \{F_1,F_2\}$-absorbing fork by activating the gadgets $F_1,F_2$ (we replace $P$ by $P-F_1-F_2$ to maintain rainbowness). Note that we reintegrated all except the last colour of $F_1$ into our extended rainbow path, but we reintegrated only the first colour of $F_2$. Hence, as our gadgets have size at most $6$, our new absorbing path $P'$ has size $$|P'|\ge |P|-6+4/\gamma\ge \frac1{2\gamma}| \mathcal{F}_0 \setminus \mathcal{F}| +\frac{2}{\gamma}\ge \frac1{2\gamma}| \mathcal{F}_0 \setminus (\mathcal{F}\setminus \{F_1,F_2\})|$$ if $|A|=s-|P| \ge 8/\gamma$. It still has size $$|P'| \ge |P|+1 \ge  \frac12|\mathcal{F}_0 \setminus \mathcal{F}|+(1-\gamma)(s-8/\gamma)+1 = \frac12 | \mathcal{F}_0 \setminus (\mathcal{F}\setminus \{F_1,F_2\})|+(1-\gamma)(s-8/\gamma)$$ if $7 \le |A|=s-|P| < 8/\gamma$.  The only remaining case is where $P'$ is missing exactly $6$ colours, including the last $5$ colours of $F_2$.

    In the final case, we can just append the leg of the in-spider $Q$ corresponding to the $5$ leftover colours from $F_2$ at the beginning of the path $P'$; this produces the desired rainbow path of length $|S|-1$.  We will be done if we show that the first two cases are actually impossible by the maximality assumption on $(P,Q)$. For this, we need to verify that at least one of the $Ks/4$ possible new final vertices is $\gamma s$-extendable in its new fork. We can repeat the argument from above with the auxiliary bipartite graph; this time there are at most $s+5\beta s + 5\gamma s \le 2s$ vertices on the left side (consisting of $P \cup Q$ together with the vertices of the out-spider of $\mathcal{F'}$ based at $v$), each sending at most $\alpha s$ edges\footnote{We note that the edges in our auxiliary bipartite graph are defined by the \emph{full} $\mathcal{F}$ out-spiders and not just the $\mathcal{F}'$ ones.} to the right side (consisting of the final vertices of our potential extended forks, namely, the translations of good vertices by the path $P_A$), which has size at least $Ks/4$. 
    So some vertex $v'$ on the right has degree at most $\frac{8 \alpha}{K} s\le \gamma s$. At most $\gamma s$ of the leaves of the out-spider of $\mathcal{F}$ based at $v'$ are blocked by vertices in $P \cup Q$ or the out-spider of $\mathcal{F'}$ based at $v$. Besides these, at most $8/\gamma$ additional leaves of the out-spider of $\mathcal{F}$ based at $v'$ are blocked by the translate of $P_A$ that we used to reach $v'$. In total, our new endpoint $v'$ is at least $|\mathcal{F}|-\gamma s - 8/\gamma \ge \gamma s$-extendable, since $$|\mathcal{F}|=|\mathcal{F}_0|-| \mathcal{F}_0 \setminus \mathcal{F}|\ge \beta s- 2\max\{\gamma s, s-(1-\gamma)(s-8/\gamma)\}\ge \alpha^3 s/2^{52}-2\gamma s\ge 2\gamma s+8/\gamma$$ (here using $\beta=\alpha^3/2^{52}$ and the assumption that $\gamma\ll \alpha$). 
\end{proof}

\section{General dense case}\label{sec:general-dense}
In this section we prove Theorem~\ref{thm:densecase-intro}, which we restate for convenience.

\generaldense*

Due to a lack of $0$-sum subsets, we will need to work with a different type of gadget. 

\begin{definition}[$g$-pair]  Let $G$ be a group, and let $g \in G$.  A \emph{$g$-pair} is a pair of distinct elements $a,b \in G$ such that $ab=g$.  A \emph{family of $g$-pairs} in a subset $S \subseteq G$ is a collection of disjoint $g$-pairs contained in $S$; we say that the number of such pairs is the \emph{size} of the family.

\end{definition}

The following easy lemma lets us find a large family of $g$-pairs in any large subset of a finite group.
\begin{lemma}\label{lem:common-sum}
    Let $G$ be a group, and let $S\subseteq G$. If $|S|\geq |G|^{1-\eps}\ge 2$, then for some $g\in G$ there exists a family of $g$-pairs in $S$ of size at least $|G|^{1-2\eps}/6$.
\end{lemma}
\begin{proof}
    There are $|S|(|S|-1)$ ordered pairs of distinct elements $a,b \in S$. There are at most $|G|$ possible values for the product $ab$, so by the pigeonhole principle there is some $g \in G$ such that $S$ contains a collection of at least $|S|(|S|-1)/|G|\ge |G|^{1-2\eps}/2$ $g$-pairs. 
    Each such $g$-pair $(a,b)$ can have elements in common with at most two other $g$-pairs $(a',b')$.  So we can find a sub-collection consisting of at least  $|G|^{1-2\eps}/6$ disjoint $g$-pairs.
\end{proof}

Consider a vertex $v$ of $\Cayley_G(S)$ and a family of $g$-pairs of $S$.  For each $g$-pair $(a,b)$, the two-edge path $v \to va \to vab=vg$ using the colours $a,b$ terminates at the vertex $vg$.  The union of these paths over all of the $g$-pairs forms a \emph{theta-graph}.  We will later stitch together several such theta-graphs along a path-like structure which we call a \textit{waveform};  see \Cref{waveform} for an illustration (formal definitions forthcoming). 

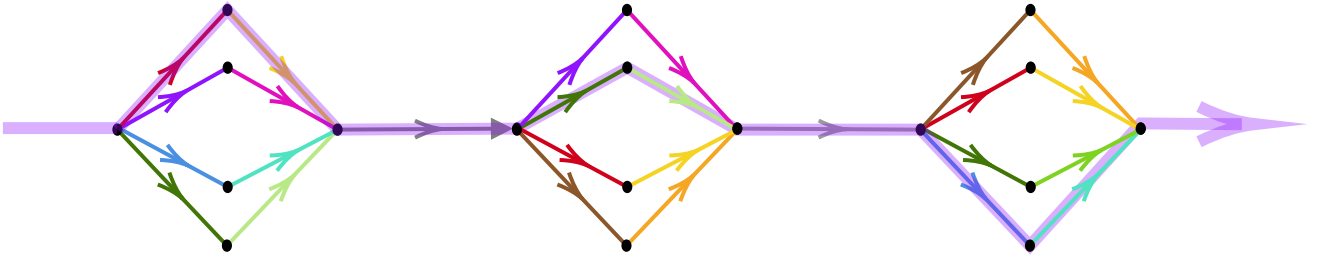
\begin{figure}[h]
    \centering

\tikzset{every picture/.style={line width=0.75pt}} 

\begin{tikzpicture}[x=0.75pt,y=0.75pt,yscale=-1,xscale=0.95]

\draw [color={rgb, 255:red, 208; green, 2; blue, 27 }  ,draw opacity=1 ][fill={rgb, 255:red, 74; green, 144; blue, 226 }  ,fill opacity=1 ][line width=1.5]    (514.13,62.76) -- (572.08,31.6) ;
\draw [shift={(549.97,43.49)}, rotate = 151.73] [color={rgb, 255:red, 208; green, 2; blue, 27 }  ,draw opacity=1 ][line width=1.5]    (14.21,-4.28) .. controls (9.04,-1.82) and (4.3,-0.39) .. (0,0) .. controls (4.3,0.39) and (9.04,1.82) .. (14.21,4.28)   ;
\draw [color={rgb, 255:red, 139; green, 87; blue, 42 }  ,draw opacity=1 ][line width=1.5]    (514.13,62.76) -- (571.95,2.6) ;
\draw [shift={(548.44,27.06)}, rotate = 133.86] [color={rgb, 255:red, 139; green, 87; blue, 42 }  ,draw opacity=1 ][line width=1.5]    (14.21,-4.28) .. controls (9.04,-1.82) and (4.3,-0.39) .. (0,0) .. controls (4.3,0.39) and (9.04,1.82) .. (14.21,4.28)   ;
\draw [color={rgb, 255:red, 65; green, 117; blue, 5 }  ,draw opacity=1 ][fill={rgb, 255:red, 74; green, 144; blue, 226 }  ,fill opacity=1 ][line width=1.5]    (516.26,62.76) -- (572.08,91.6) ;
\draw [shift={(551.1,80.76)}, rotate = 207.33] [color={rgb, 255:red, 65; green, 117; blue, 5 }  ,draw opacity=1 ][line width=1.5]    (14.21,-4.28) .. controls (9.04,-1.82) and (4.3,-0.39) .. (0,0) .. controls (4.3,0.39) and (9.04,1.82) .. (14.21,4.28)   ;
\draw [color={rgb, 255:red, 74; green, 144; blue, 226 }  ,draw opacity=1 ][line width=1.5]    (514.13,62.76) -- (571.67,121.12) ;
\draw [shift={(548.38,97.49)}, rotate = 225.4] [color={rgb, 255:red, 74; green, 144; blue, 226 }  ,draw opacity=1 ][line width=1.5]    (14.21,-4.28) .. controls (9.04,-1.82) and (4.3,-0.39) .. (0,0) .. controls (4.3,0.39) and (9.04,1.82) .. (14.21,4.28)   ;
\draw [color={rgb, 255:red, 245; green, 211; blue, 35 }  ,draw opacity=1 ][line width=1.5]    (207.87,62.76) -- (150,2.6) ;
\draw [shift={(185.04,39.02)}, rotate = 226.11] [color={rgb, 255:red, 245; green, 211; blue, 35 }  ,draw opacity=1 ][line width=1.5]    (14.21,-4.28) .. controls (9.04,-1.82) and (4.3,-0.39) .. (0,0) .. controls (4.3,0.39) and (9.04,1.82) .. (14.21,4.28)   ;
\draw [color={rgb, 255:red, 224; green, 16; blue, 190 }  ,draw opacity=1 ][line width=1.5]    (207.87,62.76) -- (150.13,31.6) ;
\draw [shift={(186.74,51.36)}, rotate = 208.36] [color={rgb, 255:red, 224; green, 16; blue, 190 }  ,draw opacity=1 ][line width=1.5]    (14.21,-4.28) .. controls (9.04,-1.82) and (4.3,-0.39) .. (0,0) .. controls (4.3,0.39) and (9.04,1.82) .. (14.21,4.28)   ;
\draw [color={rgb, 255:red, 184; green, 233; blue, 134 }  ,draw opacity=1 ][line width=1.5]    (207.87,62.76) -- (149.72,121.12) ;
\draw [shift={(185.01,85.71)}, rotate = 134.9] [color={rgb, 255:red, 184; green, 233; blue, 134 }  ,draw opacity=1 ][line width=1.5]    (14.21,-4.28) .. controls (9.04,-1.82) and (4.3,-0.39) .. (0,0) .. controls (4.3,0.39) and (9.04,1.82) .. (14.21,4.28)   ;
\draw [color={rgb, 255:red, 144; green, 19; blue, 254 }  ,draw opacity=1 ][fill={rgb, 255:red, 74; green, 144; blue, 226 }  ,fill opacity=1 ][line width=1.5]    (92.19,62.76) -- (150.13,31.6) ;
\draw [shift={(128.03,43.49)}, rotate = 151.73] [color={rgb, 255:red, 144; green, 19; blue, 254 }  ,draw opacity=1 ][line width=1.5]    (14.21,-4.28) .. controls (9.04,-1.82) and (4.3,-0.39) .. (0,0) .. controls (4.3,0.39) and (9.04,1.82) .. (14.21,4.28)   ;
\draw [color={rgb, 255:red, 80; green, 227; blue, 194 }  ,draw opacity=1 ][line width=1.5]    (207.87,62.76) -- (150.13,91.6) ;
\draw [shift={(186.87,73.25)}, rotate = 153.46] [color={rgb, 255:red, 80; green, 227; blue, 194 }  ,draw opacity=1 ][line width=1.5]    (14.21,-4.28) .. controls (9.04,-1.82) and (4.3,-0.39) .. (0,0) .. controls (4.3,0.39) and (9.04,1.82) .. (14.21,4.28)   ;
\draw [color={rgb, 255:red, 155; green, 155; blue, 155 }  ,draw opacity=1 ][line width=1.5]    (512.13,62.76) -- (420,62.25) ;
\draw [shift={(474.87,62.56)}, rotate = 180.32] [color={rgb, 255:red, 155; green, 155; blue, 155 }  ,draw opacity=1 ][line width=1.5]    (14.21,-4.28) .. controls (9.04,-1.82) and (4.3,-0.39) .. (0,0) .. controls (4.3,0.39) and (9.04,1.82) .. (14.21,4.28)   ;
\draw [color={rgb, 255:red, 128; green, 128; blue, 128 }  ,draw opacity=1 ][line width=1.5]    (296.05,62.27) -- (207.87,62.76) ;
\draw [shift={(262.76,62.46)}, rotate = 179.68] [color={rgb, 255:red, 128; green, 128; blue, 128 }  ,draw opacity=1 ][line width=1.5]    (14.21,-4.28) .. controls (9.04,-1.82) and (4.3,-0.39) .. (0,0) .. controls (4.3,0.39) and (9.04,1.82) .. (14.21,4.28)   ;
\draw [shift={(300.05,62.25)}, rotate = 179.68] [fill={rgb, 255:red, 128; green, 128; blue, 128 }  ,fill opacity=1 ][line width=0.08]  [draw opacity=0] (11.61,-5.58) -- (0,0) -- (11.61,5.58) -- cycle    ;
\draw  [fill={rgb, 255:red, 0; green, 0; blue, 0 }  ,fill opacity=1 ] (205.74,62.76) .. controls (205.74,61.33) and (206.69,60.16) .. (207.87,60.16) .. controls (209.05,60.16) and (210,61.33) .. (210,62.76) .. controls (210,64.2) and (209.05,65.36) .. (207.87,65.36) .. controls (206.69,65.36) and (205.74,64.2) .. (205.74,62.76) -- cycle ;
\draw [color={rgb, 255:red, 208; green, 2; blue, 27 }  ,draw opacity=1 ][line width=1.5]    (92.19,62.76) -- (150,2.6) ;
\draw [shift={(126.5,27.06)}, rotate = 133.86] [color={rgb, 255:red, 208; green, 2; blue, 27 }  ,draw opacity=1 ][line width=1.5]    (14.21,-4.28) .. controls (9.04,-1.82) and (4.3,-0.39) .. (0,0) .. controls (4.3,0.39) and (9.04,1.82) .. (14.21,4.28)   ;
\draw [color={rgb, 255:red, 74; green, 144; blue, 226 }  ,draw opacity=1 ][fill={rgb, 255:red, 74; green, 144; blue, 226 }  ,fill opacity=1 ][line width=1.5]    (94.32,62.76) -- (150.13,91.6) ;
\draw [shift={(129.15,80.76)}, rotate = 207.33] [color={rgb, 255:red, 74; green, 144; blue, 226 }  ,draw opacity=1 ][line width=1.5]    (14.21,-4.28) .. controls (9.04,-1.82) and (4.3,-0.39) .. (0,0) .. controls (4.3,0.39) and (9.04,1.82) .. (14.21,4.28)   ;
\draw [color={rgb, 255:red, 65; green, 117; blue, 5 }  ,draw opacity=1 ][line width=1.5]    (92.19,62.76) -- (149.72,121.12) ;
\draw [shift={(126.43,97.49)}, rotate = 225.4] [color={rgb, 255:red, 65; green, 117; blue, 5 }  ,draw opacity=1 ][line width=1.5]    (14.21,-4.28) .. controls (9.04,-1.82) and (4.3,-0.39) .. (0,0) .. controls (4.3,0.39) and (9.04,1.82) .. (14.21,4.28)   ;
\draw  [fill={rgb, 255:red, 0; green, 0; blue, 0 }  ,fill opacity=1 ] (147.87,2.6) .. controls (147.87,1.16) and (148.82,0) .. (150,0) .. controls (151.18,0) and (152.13,1.16) .. (152.13,2.6) .. controls (152.13,4.04) and (151.18,5.2) .. (150,5.2) .. controls (148.82,5.2) and (147.87,4.04) .. (147.87,2.6) -- cycle ;
\draw  [fill={rgb, 255:red, 0; green, 0; blue, 0 }  ,fill opacity=1 ] (148,91.6) .. controls (148,90.16) and (148.95,89) .. (150.13,89) .. controls (151.31,89) and (152.26,90.16) .. (152.26,91.6) .. controls (152.26,93.04) and (151.31,94.2) .. (150.13,94.2) .. controls (148.95,94.2) and (148,93.04) .. (148,91.6) -- cycle ;
\draw  [fill={rgb, 255:red, 0; green, 0; blue, 0 }  ,fill opacity=1 ] (147.59,121.12) .. controls (147.59,119.68) and (148.55,118.51) .. (149.72,118.51) .. controls (150.9,118.51) and (151.85,119.68) .. (151.85,121.12) .. controls (151.85,122.55) and (150.9,123.72) .. (149.72,123.72) .. controls (148.55,123.72) and (147.59,122.55) .. (147.59,121.12) -- cycle ;
\draw  [fill={rgb, 255:red, 0; green, 0; blue, 0 }  ,fill opacity=1 ] (90.05,62.76) .. controls (90.05,61.33) and (91.01,60.16) .. (92.19,60.16) .. controls (93.36,60.16) and (94.32,61.33) .. (94.32,62.76) .. controls (94.32,64.2) and (93.36,65.36) .. (92.19,65.36) .. controls (91.01,65.36) and (90.05,64.2) .. (90.05,62.76) -- cycle ;
\draw  [fill={rgb, 255:red, 0; green, 0; blue, 0 }  ,fill opacity=1 ] (415.74,62.25) .. controls (415.74,60.81) and (416.69,59.65) .. (417.87,59.65) .. controls (419.05,59.65) and (420,60.81) .. (420,62.25) .. controls (420,63.69) and (419.05,64.85) .. (417.87,64.85) .. controls (416.69,64.85) and (415.74,63.69) .. (415.74,62.25) -- cycle ;
\draw  [fill={rgb, 255:red, 0; green, 0; blue, 0 }  ,fill opacity=1 ] (300.05,62.25) .. controls (300.05,60.81) and (301.01,59.65) .. (302.19,59.65) .. controls (303.36,59.65) and (304.32,60.81) .. (304.32,62.25) .. controls (304.32,63.69) and (303.36,64.85) .. (302.19,64.85) .. controls (301.01,64.85) and (300.05,63.69) .. (300.05,62.25) -- cycle ;
\draw  [fill={rgb, 255:red, 0; green, 0; blue, 0 }  ,fill opacity=1 ] (512.13,62.76) .. controls (512.13,61.33) and (513.09,60.16) .. (514.26,60.16) .. controls (515.44,60.16) and (516.39,61.33) .. (516.39,62.76) .. controls (516.39,64.2) and (515.44,65.36) .. (514.26,65.36) .. controls (513.09,65.36) and (512.13,64.2) .. (512.13,62.76) -- cycle ;
\draw [color={rgb, 255:red, 144; green, 19; blue, 254 }  ,draw opacity=0.34 ][line width=4.5]    (92.19,62.76) -- (150,2.6) -- (207.87,62.76) -- (302.19,62.25) -- (360.08,31.6) -- (417.81,62.76) -- (514.26,62.76) -- (571.67,121.12) -- (629.87,59.65) -- (675,60) -- (675,60) -- (675,60) -- (683,60) ;
\draw [shift={(690,60)}, rotate = 180] [color={rgb, 255:red, 144; green, 19; blue, 254 }  ,draw opacity=0.34 ][line width=4.5]    (29.51,-8.88) .. controls (18.76,-3.77) and (8.93,-0.81) .. (0,0) .. controls (8.93,0.81) and (18.76,3.77) .. (29.51,8.88)   ;
\draw  [fill={rgb, 255:red, 0; green, 0; blue, 0 }  ,fill opacity=1 ] (148,31.6) .. controls (148,30.16) and (148.95,29) .. (150.13,29) .. controls (151.31,29) and (152.26,30.16) .. (152.26,31.6) .. controls (152.26,33.04) and (151.31,34.2) .. (150.13,34.2) .. controls (148.95,34.2) and (148,33.04) .. (148,31.6) -- cycle ;
\draw [color={rgb, 255:red, 224; green, 16; blue, 190 }  ,draw opacity=1 ][line width=1.5]    (417.81,62.76) -- (359.95,2.6) ;
\draw [shift={(394.98,39.02)}, rotate = 226.11] [color={rgb, 255:red, 224; green, 16; blue, 190 }  ,draw opacity=1 ][line width=1.5]    (14.21,-4.28) .. controls (9.04,-1.82) and (4.3,-0.39) .. (0,0) .. controls (4.3,0.39) and (9.04,1.82) .. (14.21,4.28)   ;
\draw [color={rgb, 255:red, 184; green, 233; blue, 134 }  ,draw opacity=1 ][line width=1.5]    (417.81,62.76) -- (360.08,31.6) ;
\draw [shift={(396.69,51.36)}, rotate = 208.36] [color={rgb, 255:red, 184; green, 233; blue, 134 }  ,draw opacity=1 ][line width=1.5]    (14.21,-4.28) .. controls (9.04,-1.82) and (4.3,-0.39) .. (0,0) .. controls (4.3,0.39) and (9.04,1.82) .. (14.21,4.28)   ;
\draw [color={rgb, 255:red, 245; green, 166; blue, 35 }  ,draw opacity=1 ][line width=1.5]    (417.81,62.76) -- (359.67,121.12) ;
\draw [shift={(394.95,85.71)}, rotate = 134.9] [color={rgb, 255:red, 245; green, 166; blue, 35 }  ,draw opacity=1 ][line width=1.5]    (14.21,-4.28) .. controls (9.04,-1.82) and (4.3,-0.39) .. (0,0) .. controls (4.3,0.39) and (9.04,1.82) .. (14.21,4.28)   ;
\draw [color={rgb, 255:red, 65; green, 117; blue, 5 }  ,draw opacity=1 ][fill={rgb, 255:red, 74; green, 144; blue, 226 }  ,fill opacity=1 ][line width=1.5]    (302.13,62.76) -- (360.08,31.6) ;
\draw [shift={(337.97,43.49)}, rotate = 151.73] [color={rgb, 255:red, 65; green, 117; blue, 5 }  ,draw opacity=1 ][line width=1.5]    (14.21,-4.28) .. controls (9.04,-1.82) and (4.3,-0.39) .. (0,0) .. controls (4.3,0.39) and (9.04,1.82) .. (14.21,4.28)   ;
\draw [color={rgb, 255:red, 245; green, 211; blue, 35 }  ,draw opacity=1 ][line width=1.5]    (417.81,62.76) -- (360.08,91.6) ;
\draw [shift={(396.82,73.25)}, rotate = 153.46] [color={rgb, 255:red, 245; green, 211; blue, 35 }  ,draw opacity=1 ][line width=1.5]    (14.21,-4.28) .. controls (9.04,-1.82) and (4.3,-0.39) .. (0,0) .. controls (4.3,0.39) and (9.04,1.82) .. (14.21,4.28)   ;
\draw [color={rgb, 255:red, 144; green, 19; blue, 254 }  ,draw opacity=1 ][line width=1.5]    (302.13,62.76) -- (359.95,2.6) ;
\draw [shift={(336.44,27.06)}, rotate = 133.86] [color={rgb, 255:red, 144; green, 19; blue, 254 }  ,draw opacity=1 ][line width=1.5]    (14.21,-4.28) .. controls (9.04,-1.82) and (4.3,-0.39) .. (0,0) .. controls (4.3,0.39) and (9.04,1.82) .. (14.21,4.28)   ;
\draw [color={rgb, 255:red, 208; green, 2; blue, 27 }  ,draw opacity=1 ][fill={rgb, 255:red, 74; green, 144; blue, 226 }  ,fill opacity=1 ][line width=1.5]    (304.26,62.76) -- (360.08,91.6) ;
\draw [shift={(339.1,80.76)}, rotate = 207.33] [color={rgb, 255:red, 208; green, 2; blue, 27 }  ,draw opacity=1 ][line width=1.5]    (14.21,-4.28) .. controls (9.04,-1.82) and (4.3,-0.39) .. (0,0) .. controls (4.3,0.39) and (9.04,1.82) .. (14.21,4.28)   ;
\draw [color={rgb, 255:red, 139; green, 87; blue, 42 }  ,draw opacity=1 ][line width=1.5]    (302.13,62.76) -- (359.67,121.12) ;
\draw [shift={(336.38,97.49)}, rotate = 225.4] [color={rgb, 255:red, 139; green, 87; blue, 42 }  ,draw opacity=1 ][line width=1.5]    (14.21,-4.28) .. controls (9.04,-1.82) and (4.3,-0.39) .. (0,0) .. controls (4.3,0.39) and (9.04,1.82) .. (14.21,4.28)   ;
\draw  [fill={rgb, 255:red, 0; green, 0; blue, 0 }  ,fill opacity=1 ] (357.81,2.6) .. controls (357.81,1.16) and (358.77,0) .. (359.95,0) .. controls (361.12,0) and (362.08,1.16) .. (362.08,2.6) .. controls (362.08,4.04) and (361.12,5.2) .. (359.95,5.2) .. controls (358.77,5.2) and (357.81,4.04) .. (357.81,2.6) -- cycle ;
\draw  [fill={rgb, 255:red, 0; green, 0; blue, 0 }  ,fill opacity=1 ] (357.95,91.6) .. controls (357.95,90.16) and (358.9,89) .. (360.08,89) .. controls (361.25,89) and (362.21,90.16) .. (362.21,91.6) .. controls (362.21,93.04) and (361.25,94.2) .. (360.08,94.2) .. controls (358.9,94.2) and (357.95,93.04) .. (357.95,91.6) -- cycle ;
\draw  [fill={rgb, 255:red, 0; green, 0; blue, 0 }  ,fill opacity=1 ] (357.54,121.12) .. controls (357.54,119.68) and (358.49,118.51) .. (359.67,118.51) .. controls (360.85,118.51) and (361.8,119.68) .. (361.8,121.12) .. controls (361.8,122.55) and (360.85,123.72) .. (359.67,123.72) .. controls (358.49,123.72) and (357.54,122.55) .. (357.54,121.12) -- cycle ;
\draw  [fill={rgb, 255:red, 0; green, 0; blue, 0 }  ,fill opacity=1 ] (300,62.76) .. controls (300,61.33) and (300.95,60.16) .. (302.13,60.16) .. controls (303.31,60.16) and (304.26,61.33) .. (304.26,62.76) .. controls (304.26,64.2) and (303.31,65.36) .. (302.13,65.36) .. controls (300.95,65.36) and (300,64.2) .. (300,62.76) -- cycle ;
\draw  [fill={rgb, 255:red, 0; green, 0; blue, 0 }  ,fill opacity=1 ] (357.95,31.6) .. controls (357.95,30.16) and (358.9,29) .. (360.08,29) .. controls (361.25,29) and (362.21,30.16) .. (362.21,31.6) .. controls (362.21,33.04) and (361.25,34.2) .. (360.08,34.2) .. controls (358.9,34.2) and (357.95,33.04) .. (357.95,31.6) -- cycle ;
\draw  [fill={rgb, 255:red, 0; green, 0; blue, 0 }  ,fill opacity=1 ] (415.74,62.25) .. controls (415.74,60.81) and (416.69,59.65) .. (417.87,59.65) .. controls (419.05,59.65) and (420,60.81) .. (420,62.25) .. controls (420,63.69) and (419.05,64.85) .. (417.87,64.85) .. controls (416.69,64.85) and (415.74,63.69) .. (415.74,62.25) -- cycle ;
\draw  [fill={rgb, 255:red, 0; green, 0; blue, 0 }  ,fill opacity=1 ] (627.74,62.25) .. controls (627.74,60.81) and (628.69,59.65) .. (629.87,59.65) .. controls (631.05,59.65) and (632,60.81) .. (632,62.25) .. controls (632,63.69) and (631.05,64.85) .. (629.87,64.85) .. controls (628.69,64.85) and (627.74,63.69) .. (627.74,62.25) -- cycle ;
\draw [color={rgb, 255:red, 245; green, 166; blue, 35 }  ,draw opacity=1 ][line width=1.5]    (629.81,62.76) -- (571.95,2.6) ;
\draw [shift={(606.98,39.02)}, rotate = 226.11] [color={rgb, 255:red, 245; green, 166; blue, 35 }  ,draw opacity=1 ][line width=1.5]    (14.21,-4.28) .. controls (9.04,-1.82) and (4.3,-0.39) .. (0,0) .. controls (4.3,0.39) and (9.04,1.82) .. (14.21,4.28)   ;
\draw [color={rgb, 255:red, 245; green, 211; blue, 35 }  ,draw opacity=1 ][line width=1.5]    (629.81,62.76) -- (572.08,31.6) ;
\draw [shift={(608.69,51.36)}, rotate = 208.36] [color={rgb, 255:red, 245; green, 211; blue, 35 }  ,draw opacity=1 ][line width=1.5]    (14.21,-4.28) .. controls (9.04,-1.82) and (4.3,-0.39) .. (0,0) .. controls (4.3,0.39) and (9.04,1.82) .. (14.21,4.28)   ;
\draw [color={rgb, 255:red, 80; green, 227; blue, 194 }  ,draw opacity=1 ][line width=1.5]    (629.81,62.76) -- (571.67,121.12) ;
\draw [shift={(606.95,85.71)}, rotate = 134.9] [color={rgb, 255:red, 80; green, 227; blue, 194 }  ,draw opacity=1 ][line width=1.5]    (14.21,-4.28) .. controls (9.04,-1.82) and (4.3,-0.39) .. (0,0) .. controls (4.3,0.39) and (9.04,1.82) .. (14.21,4.28)   ;
\draw [color={rgb, 255:red, 126; green, 211; blue, 33 }  ,draw opacity=1 ][line width=1.5]    (629.81,62.76) -- (572.08,91.6) ;
\draw [shift={(608.82,73.25)}, rotate = 153.46] [color={rgb, 255:red, 126; green, 211; blue, 33 }  ,draw opacity=1 ][line width=1.5]    (14.21,-4.28) .. controls (9.04,-1.82) and (4.3,-0.39) .. (0,0) .. controls (4.3,0.39) and (9.04,1.82) .. (14.21,4.28)   ;
\draw  [fill={rgb, 255:red, 0; green, 0; blue, 0 }  ,fill opacity=1 ] (569.81,2.6) .. controls (569.81,1.16) and (570.77,0) .. (571.95,0) .. controls (573.12,0) and (574.08,1.16) .. (574.08,2.6) .. controls (574.08,4.04) and (573.12,5.2) .. (571.95,5.2) .. controls (570.77,5.2) and (569.81,4.04) .. (569.81,2.6) -- cycle ;
\draw  [fill={rgb, 255:red, 0; green, 0; blue, 0 }  ,fill opacity=1 ] (569.95,91.6) .. controls (569.95,90.16) and (570.9,89) .. (572.08,89) .. controls (573.25,89) and (574.21,90.16) .. (574.21,91.6) .. controls (574.21,93.04) and (573.25,94.2) .. (572.08,94.2) .. controls (570.9,94.2) and (569.95,93.04) .. (569.95,91.6) -- cycle ;
\draw  [fill={rgb, 255:red, 0; green, 0; blue, 0 }  ,fill opacity=1 ] (569.54,121.12) .. controls (569.54,119.68) and (570.49,118.51) .. (571.67,118.51) .. controls (572.85,118.51) and (573.8,119.68) .. (573.8,121.12) .. controls (573.8,122.55) and (572.85,123.72) .. (571.67,123.72) .. controls (570.49,123.72) and (569.54,122.55) .. (569.54,121.12) -- cycle ;
\draw  [fill={rgb, 255:red, 0; green, 0; blue, 0 }  ,fill opacity=1 ] (569.95,31.6) .. controls (569.95,30.16) and (570.9,29) .. (572.08,29) .. controls (573.25,29) and (574.21,30.16) .. (574.21,31.6) .. controls (574.21,33.04) and (573.25,34.2) .. (572.08,34.2) .. controls (570.9,34.2) and (569.95,33.04) .. (569.95,31.6) -- cycle ;
\draw  [fill={rgb, 255:red, 0; green, 0; blue, 0 }  ,fill opacity=1 ] (627.74,62.25) .. controls (627.74,60.81) and (628.69,59.65) .. (629.87,59.65) .. controls (631.05,59.65) and (632,60.81) .. (632,62.25) .. controls (632,63.69) and (631.05,64.85) .. (629.87,64.85) .. controls (628.69,64.85) and (627.74,63.69) .. (627.74,62.25) -- cycle ;
\draw [color={rgb, 255:red, 144; green, 19; blue, 254 }  ,draw opacity=0.34 ][line width=4.5]    (32,62) -- (47,62) -- (62,62) -- (77,62) -- (77,62) -- (77,62) -- (92,62) ;

\end{tikzpicture}
    \caption{A waveform with three theta-graphs. Notice that some colour pairs repeat across the theta-graphs. A rainbow path omitting the brown-orange colour pair is highlighted.}
    \label{waveform} 
\end{figure}

\par 

We have great flexibility in taking a path through a waveform, since for each theta-graph we may (independently) choose which length-two path to include.
The downside here, compared to the zero-sum gadget approach from the previous sections, is that the remaining pairs in each theta-graph remain unused. As each colour pair must eventually appear somewhere, we need to carefully specify which colour pairs occur in which theta-graphs. This procedure will be simplest to describe using the language of a certain auxiliary bipartite graph on parts $(X,Z)$, where vertices of $Z$ correspond to theta-graphs and vertices of $X$ represent $g$-pairs as given by Lemma~\ref{lem:common-sum}. We wish to build a waveform where the theta-graph corresponding to the vertex $z\in Z$ contains precisely the $g$-pairs in its neighbourhood $N(z)\subseteq X$. No theta-graph should be too big (as then finding absorbing structures in random subgraphs would be problematic), so our bipartite graph should have small maximum degree. Furthermore, after we later (during an absorption step) saturate some fraction of the $g$-pairs elsewhere, we will wish to integrate exactly the remaining (unused) $g$-pairs into the path coming from our waveform. In the language of the bipartite graph, we wish to find perfect matchings between $X\setminus X'$ and $Z$ for a wide variety of choices of subsets $X'$. 

The following proposition from \cite{randomspanningtree} shows that bipartite graphs with such properties exist. 
\begin{proposition}[\cite{randomspanningtree}, Lemma 10.7]\label{lem:robustbipartite}
Let $\ell \le k$ be positive integers. There exists a bipartite graph $B$ with bipartition $(X \cup Y,Z)$, where $X$ and $Y$ are disjoint, $|X|=k+\ell, |Y|=2k$ and $|Z|=3k$, such that:
\begin{enumerate}
    \item\label{itm:max-deg} The maximum degree is at most $40$.
    \item\label{itm:matching} For any subset $X'\subset X$ with $|X'|= k$, there is a perfect matching between $X' \cup Y$ and $Z$.
\end{enumerate} 
\end{proposition}

We remark that \cite[Lemma 10.7]{randomspanningtree} states this result only with $\ell=k$, but one can obtain the more general version above by simply deleting an appropriate number of vertices from $X$. The construction from \cite{randomspanningtree} is essentially a union of $40$ perfect matchings, sampled uniformly at random, and it is not hard to show that the desired properties hold with positive probability.

We can now precisely describe the type of absorbing structure we will utilise.
\begin{definition}[Absorbing family]
Let $G$ be a group, $g \in G$, and $S \subseteq G$. Let $\mathcal{P}_{\mathrm{flex}} \subseteq \mathcal{P}$ be families of $g$-pairs in $S$.  A \emph{$(\mathcal{P}_{\mathrm{flex}},\ell)$-absorbing family in $\mathcal{P}$} consists of a sequence $\mathcal{P}_1,\ldots, \mathcal{P}_{|\mathcal{P}|-\ell} \subseteq \mathcal{P}$ of sub-families such that each $|\mathcal{P}_i|\le 40$ and for every $\mathcal{P}' \subseteq \mathcal{P}_{\mathrm{flex}}$ with $|\mathcal{P}'|=\ell$ there exists a system of distinct representatives $p_i \in \mathcal{P}_i \setminus \mathcal{P}'$.
\end{definition}

The following is essentially an immediate consequence of \Cref{lem:robustbipartite}. 
\begin{corollary}\label{cor:absorbing-pairs}
    Let $G$ be a group.  Let $g \in G$, and let $S \subseteq G$, and let $\mathcal{P}$ be a family of $g$-pairs in $S$ of size $3t+\ell$, with $\ell \le t$.  Then for some sub-family $\mathcal{P}_{\mathrm{flex}} \subseteq \mathcal{P}$ of size $t+\ell$ there is a $(\mathcal{P}_{\mathrm{flex}},\ell)$-absorbing family in $\mathcal{P}$.
\end{corollary}
\begin{proof}
Consider the bipartite graph $B$ from \Cref{lem:robustbipartite}. 
We identify $X \cup Y$ with $\mathcal{P}$, set $\mathcal{P}_{\mathrm{flex}}:=X$, and define $\mathcal{P}_i$ to be the set of neighbours of the $i$-th vertex in $Z$. Note that each $|\mathcal{P}_i|\le 40$ since $B$ has maximum degree at most $40$. 
 The matching property of $B$ precisely translates to  
$\mathcal{P}_1,\ldots, \mathcal{P}_{3t},\mathcal{P}_{\mathrm{flex}}$ having the desired property about distinct representatives.     
\end{proof}

As we alluded to above, the building blocks for our absorbing structures are theta-graphs.

\begin{definition}[Theta-graph] Let $\mathcal{P}$ be a family of $g$-pairs in a group $G$.  The \emph{theta-graph} of $\mathcal{P}$ anchored at $v$, denoted $T(v,\mathcal{P})$, is the union of the $|\mathcal{P}|$ length-two paths obtained by starting at $v$ and following the edges of colours $a,b$ for each $g$-pair $(a,b) \in \mathcal{P}$.
\end{definition} 
Notice that each of the paths in $T(v,\mathcal{P})$ terminates at the vertex $vab=vg$.

\begin{definition}[Waveform]
\par  Let $G$ be a group, $g \in G$, $S \subseteq G$, and let $\mathcal{P}_1,\ldots,\mathcal{P}_{t}$ be families of $g$-pairs in $S$. A corresponding \emph{waveform} starting at a vertex $u \in G$ is a subgraph of $\Cayley_G(S)$ consisting of $u$ and $T(v_1,\mathcal{P}_1),\ldots, T(v_t,\mathcal{P}_t)$, together with the edges $(u,v_1),(v_1 g,v_2),(v_2 g,v_3),\ldots, (v_{t-1}g,v_t)$, such that:
\begin{itemize}
    \item the theta-graphs $T(v_i,\mathcal{P}_i)$ are disjoint from one another and from $u$;
    \item the edges $(u,v_1),(v_1 g,v_2),(v_2 g,v_3),\ldots, (v_{t-1}g,v_t)$ are all edges of $\Cayley_G(S)$ of distinct colours, and none of these colours appear in the $T(v_i,\mathcal{P}_i)$'s.
\end{itemize}
\end{definition}

A waveform of a $(\mathcal{P}_{\mathrm{flex}},\ell)$-absorbing family in $\mathcal{P}$ has the key property that by choosing a single length-two path from each theta-graph, we can construct a rainbow path which uses all the pairs of colours in $\mathcal{P}$ apart from any desired subcollection of $\ell$ pairs of $\mathcal{P}_{\mathrm{flex}}$; this allows us to ``absorb'' such subfamilies. We will refer to the specification of such a subpath of the waveform as \emph{collapsing} the waveform. 
We now present a waveform analogue of \Cref{lem:absorbing-path}, with a similar proof. 

\begin{lemma}\label{lem:absorbing-waveform}
   Let $G$ be a group, $g \in G$, and $E \subseteq G$. Let $\mathcal{P}_1,\ldots,\mathcal{P}_{t}$ be families of $g$-pairs in $E$ with each $|\mathcal{P}_i|\le 40$.  Let $0<p \leq 1$ be such that $|E|p^{42}/2^{18} \ge \max\{t,\log |G|\}$. Let $R$ be a $p$-random subset of $G$. Then with high probability, 
   for every vertex $u \in G$
   we can find a waveform in $\Cayley_G(E)$ corresponding to $\mathcal{P}_1,\ldots,\mathcal{P}_{t}$ which starts at the vertex $u$ and is otherwise contained in $R$.
\end{lemma}

\begin{proof}
Let $N:=|G|$.
For each vertex $v \in G$ and family $\mathcal{P}_i$, let $E_{v,i}$ be the event that we can find a collection of at least $50t$ elements $e \in E \setminus \bigcup_{i=1}^t \mathcal{P}_i$ whose corresponding theta-graphs $T(ve,\mathcal{P}_i)$ are all vertex-disjoint and contained in $R$. Notice that each theta-graph $T(ve,\mathcal{P}_i)$ intersects at most $42^2$ other such theta-graphs, since each theta-graph has at most $42$ vertices and the translate of the other theta-graph is determined by the relative positions of the intersection point on the two paths.  Thus we can find a collection of at least $(|E|-80t)/(42^2+1)\ge |E|/2^{11}$ vertex-disjoint $T(ve,\mathcal{P}_i)$'s. Each survives in $R$ with probability at least $p^{42}$, and these events are independent. Hence the number of surviving theta-graphs stochastically dominates $\Bin(|E|/2^{11},p^{42})$, and by a Chernoff bound at least $$|E|p^{42}/2^{12} \ge 50t$$ survive with probability at least $1-\exp(-|E|p^{42}/2^{14})\ge 1-1/N^3$.  Thus $\mathbb{P}[E_{v,i}]\ge 1-1/N^3$, and by a union bound we conclude that with probability at least $1-1/N$ all of the events $E_{v,i}$ occur. 

Suppose we are in such an outcome. We find our waveform by incorporating theta-graphs one at a time, as in the proof above. We start our waveform $W$ at the vertex $u$ and iteratively add on graphs of the form $T(ve,\mathcal{P}_i)$, where $v$ is the current endpoint of $W$. At each step, we identify a hitherto-unincorporated $\mathcal{P}_i$ and consider the $50t$ theta-graphs $T(ve,\mathcal{P}_i)$ identified in the previous paragraph. Of these, at least $49t$ correspond to colours $e$ that have not yet been used. Since $|W|< 42t$, there are at least $7t$ theta-graphs $T(ve,\mathcal{P}_i)$ that remain disjoint from $W$; we choose one such theta-graph and add it to the end of our waveform $W$.
\end{proof}

We will also need a slightly tweaked version of the absorbing lemma (\Cref{lem:tails}).

\begin{lemma}\label{lem:pair-tails}
    Let $p \in (0,1]$, let $\mathcal{P}_{\mathrm{flex}}$ be a family of $g$-pairs in a group $G$ with $|\mathcal{P}_{\mathrm{flex}}|\geq t+\ell \ge 2^{9}p^{-3}\log |G|$, and let $T$ be a $p$-random subset of $G$. Then with high probability, the following holds for every $L \subseteq G \setminus \bigcup \mathcal{P}_{\mathrm{flex}}$ of size $|L| \le \ell < tp^3/80$ and every vertex $v \in G$: There exist a subfamily $\mathcal{P}' \subseteq \mathcal{P}_{\mathrm{flex}}$ of size $\ell$ and a rainbow path in $\Cayley_{\Fn}(L \cup \bigcup \mathcal{P}')$ that starts at $v$, is otherwise contained in $T$, and uses all except possibly one colour from $L \cup \bigcup \mathcal{P}'$. 
\end{lemma}
\begin{proof}     
    Consider a pair of distinct colours $a,b \in G$. For each $(a_i,b_i) \in \mathcal{P}_{\mathrm{flex}}$, consider the length-three path that starts at the vertex $\id$ and then traverses the edges of colours $a_i,a,b$.  Our first goal is to construct a family $\mathcal{P}_{a,b}$ of at least $t/10$ such paths that are vertex-disjoint (except at the shared vertex $\id$). Note that each such path can intersect at most nine other paths, so we can find a vertex-disjoint collection of $|\mathcal{P}_{\mathrm{flex}}|/10 \ge t/10$ of them.

     For each vertex $u \in G$ pair of distinct colours $a,b \in S$, let $E_{u,a,b}$ be the event that we can find a collection of more than $4\ell$ $g$-pairs $P \in \mathcal{P}_{a,b}$ such that the (left-)translates by $u$ of the corresponding length-three paths are all contained in $T$ (except for possibly $u$).
     The number of surviving paths  stochastically dominates $\Bin(t/10,p^3)$, so by a Chernoff bound at least $tp^3/20>4\ell$ survive with probability at least $1- \exp(tp^3/80)\ge 1-1/N^4$.  Thus $\mathbb P[E_{u,a,b}] \geq 1-1/N^4$. A union bound over all $u,a,b$ ensures that with probability at least $1-1/N$ all of the events $E_{u,a,b}$ occur.  Suppose we are in such an outcome.  
     
     We will construct a sequence of sets $L=L_0 ,  L_1 , \ldots ,  L_{|L|-1}$ of sizes $|L_i|=|L|-i$ and a sequence of directed rainbow paths $v=P_0 \subset P_1 \subset \cdots \subset P_{|L|-1}$ of sizes $|P_i|=3i+1$, as follows.
     Suppose we have already constructed $L_i,P_i$, and suppose that $|L_i| \geq 2$. Pick some distinct $a,b \in L_i$. There is a collection of more than $4\ell \ge 4|L|$ vertex-disjoint rainbow paths in $T$, where each starts at the endpoint of $P_i$ and then traverses the edges with colours $a_j,a,b$ for some $(a_j,b_j) \in \mathcal{P}_{\mathrm{flex}}$. Some such path uses a colour $a_j$ from a new pair 
     and is vertex-disjoint from $P_i$;
     we append it to $P_i$ to obtain $P_{i+1}$. To obtain $L_{i+1}$ from $L_i$, we remove $a,b$ and add $b_j$ (so indeed $|L_{i+1}|=|L_i|-1$).  Thus the process can indeed run $|L|-1$ steps, provided we can always find a suitable short path to extend by. Note also that by construction $P_i$ uses colours from at most $i$ pairs in $\mathcal{P}_{\mathrm{flex}}$ (besides the colours from $L$) and has length $3i$. So when constructing $P_{i+1}$ at most $i$ pairs are already used, and at most $|P_i|-1=3i$ of our short paths can intersect $P_i$, so we indeed can always choose a short path to extend $P_i$ by into $P_{i+1}$.
     
     At the end of this process we used at most $|L|-1 \le \ell$ pairs from $\mathcal{P}_{\mathrm{flex}}$ and embedded all the colours from these pairs as well as from $L$ except possibly one.   
     To ensure we use up exactly $\ell$ pairs, we continue the process to find  $L=L_{|L|} , \ldots ,  L_{\ell}$ each of size one and $P_{\ell} \supset \ldots \supset P_{|L|} \supset P_{|L|-1}$ such that for each $i \ge |L|$ we have $|P_i|=|P_{i-1}|+2$, and $P_i \setminus P_{i-1}$ uses the colour in $L_{i-1}$ together with a new colour from $\bigcup \mathcal{P}_{\mathrm{flex}}$.
     To see that we can do this, suppose that we are at stage $i-1\in[|L|-1,\ell)$ and that the current path $P_{i-1}$ ends at the vertex $v$. Then we pick $a$ to be the (unique) colour in $L_{i-1}$ and $b\neq a$ to be an arbitrary other colour. As the event $E_{v,a,b}$ holds, there is some new $\mathcal{P}_i=(a_i,b_i)$ (not yet used on the path) such that we may extend our current path by appending the edges going from $v$ to $va_i$ to $va_ia$ (and we simply do not use the colour $b$ edge in the path guaranteed by the above process), in order to construct $P_i$ (and replace $a$ with $b_i$ in $L_i$). Since we have more than $4\ell$ choices, we can continue until we have used up exactly $\ell$ pairs from $\mathcal{P}_{\mathrm{flex}}$, as desired.
\end{proof}

As in \Cref{sec:f2n-dense} we start by establishing the main result 
in the regime $|S| \ge \frac34 N$, which is slightly different due to tighter space constraints.

\begin{theorem}\label{thm:very-dense-general}
    Let $G$ be a group of order $N$, where $N$ is sufficiently large. If $S \subseteq G$ is a subset of size $|S| \ge \frac34 N$, then $\Cayley_G(S)$ has a rainbow path of length $|S|-1.$ 
\end{theorem}
\begin{proof}
Set $\gamma:=2^{-20}$. If $|S| \ge N-N^{1-\gamma}$, then we are done by \Cref{thm:ultra-dense} so let us assume $|S| \le N-N^{1-\gamma}$. Let us also set $p=N^{-2\gamma}/2$.

Let $E$ be a $\frac18$-random subset of $S$. Let us also partition $\Fn$ into three sets $R \sqcup M \sqcup T$ by independently assigning each vertex to $R,M,T$ with probabilities $p,1-2p,p$, respectively.

We apply \Cref{lem:asymptoticinrandom-dense} with $S=S, J=\emptyset, M=M, S'=E$ and the parameters
$$\eps=\frac34, \quad \zeta=\frac14, \quad q=1-2p, \quad q'=\frac18, \quad \mu=N^{-90\gamma}/2^{76}.$$
Since $1-2p=q\ge (1+\mu)(1-N^{-\gamma})$, and $q' \le 1-\mu q/4$, we may indeed apply the lemma. Thus with high probability we have:

\begin{enumerate}[label = {{{\textbf{D\arabic{enumi}}}}}] \setcounter{enumi}{0}
    \item\label{P1} For any $S_F \subseteq E$ and any two vertices $v,u \in \Fn$, we can find a rainbow path from $v$ to $u$ in $\Cay{S \setminus S_F}$, using all but at most $\mu q$ colours from $S \setminus S_F$, such that all of the internal vertices of the path lie in $M$.
\end{enumerate}

Let us now reveal the random subset $E$. Chernoff's bound guarantees that with high probability $|E| \ge N/16$ (as before, if this is not the case then we declare failure and do not apply the following lemmas).
Let 
$$t:=N^{1-84\gamma}/2^{66}, \quad \ell:=tp^3/2^7=N^{1-90\gamma}/2^{76}.$$ 
Using \Cref{lem:common-sum} we can find some $g \in G$ and a family $\mathcal{P}$ of $g$-pairs in $E$ of size $3t+\ell\le N/2^{11}$. Using \Cref{cor:absorbing-pairs} we can find a $(\mathcal{P}_{\mathrm{flex}},\ell)$-absorbing family $\mathcal{P}_1,\mathcal{P}_2,\ldots,\mathcal{P}_{3t}$ in $\mathcal{P}$ for some $\mathcal{P}_{\mathrm{flex}} \subseteq \mathcal{P}$ of size $t+\ell$.

We apply \Cref{lem:absorbing-waveform} to $\mathcal{P}_1,\mathcal{P}_2,\ldots,\mathcal{P}_{3t}$ with the random set $R$; note that this lemma applies since $|E| p^{42}/2^{18} \ge 3t \ge \log N$. Thus with high probability we have:

\begin{enumerate}[label = {{{\textbf{D\arabic{enumi}}}}}] \setcounter{enumi}{1}
    \item\label{P2} For any vertex $v$, there is a waveform corresponding to $\mathcal{P}_1,\mathcal{P}_2,\ldots,\mathcal{P}_{3t}$ that ends at $v$ and is otherwise contained in $R$.
\end{enumerate}

Finally, we apply \Cref{lem:pair-tails} to $\mathcal{P}_{\mathrm{flex}}$ with random set $T$; we can do so since $t+\ell \ge 2^9p^{-3} \log N$ and $\ell < tp^3/80$. Thus with high probability we have:
\begin{enumerate}[label = {{{\textbf{D\arabic{enumi}}}}}] \setcounter{enumi}{2}
    \item\label{P3} For any $L \subseteq S$ of size $|L|\le \ell$ and any vertex $u$, there is a sub-family $\mathcal{P}' \subseteq \mathcal{P}_{\mathrm{flex}}$ of size $\ell$ such that $\Cayley_G(L \cup \bigcup \mathcal{P}')$ contains a rainbow path that starts at $u$, is otherwise contained in $T$, and uses all but possibly one colour from $L \cup \bigcup \mathcal{P}'$. 
\end{enumerate}

Let us fix an outcome in which all three of the above properties hold.  Fix some distinct vertices $v \in M$, $u \in T$. \ref{P2} gives us a waveform corresponding to $\mathcal{P}_1,\mathcal{P}_2,\ldots,\mathcal{P}_{3t}$ that is 
completely contained in $R$ and ends at $v$. This waveform uses a total of $6t+\ell$ colours: It uses $3t+\ell$ colours within the theta-graphs (we have $\mathcal{P}_1 \cup \cdots \cup \mathcal{P}_{3t}=\mathcal{P}$ by the properties of an absorbing family) and an additional $3t$ colours connecting the theta-graphs to one another (and to the final vertex $v$).  Let $S_F \subseteq E$ denote this set of $6t+\ell$ colours. Now \ref{P1} produces a rainbow path $P_M$, starting at $v$ and ending $u$ and otherwise contained in $M$, which saturates all but some set $L$ of at most $\mu q N\le \ell$ colours from $S \setminus S_F$. 
Finally, by \ref{P3} we can find $\mathcal{P}' \subseteq \mathcal{P}_{\mathrm{flex}}$ of size precisely $\ell$ and a rainbow path $P_T$ contained in $T$ which uses all except possibly one colour from $L \cup \bigcup_{(a_i,b_i) \in \mathcal{P}'} \{a_i,b_i\}$. 
By the absorbing property, there is a system of distinct representatives for $\{\mathcal{P}_i \setminus \mathcal{P}': i \in [3t]\}$. Now we can follow the length-two paths in our waveform corresponding to these pairs in order to use up the colours in $S_F \setminus \bigcup_{(a_i,b_i) \in \mathcal{P}'} \{a_i,b_i\}$. The concatenation of these three paths uses all except possibly one of the colours of $S$, as desired.
\end{proof}

We are now ready to prove the main theorem of the section, Theorem~\ref{thm:densecase-intro}, in the following more precise form.  

\begin{theorem}\label{thm:dense-general}
Set $\gamma:=2^{-20}$.  Let $G$ be a group of order $N$, where $N$ is sufficiently large. Then for any subset $S \subseteq G$ of size $|S| \ge N^{1-\gamma}$, the graph $\Cayley_G(S)$ has a rainbow path of length $|S|-1.$ 
\end{theorem}

\begin{proof}
Set $\eps:=N^{-\gamma}$, so that $|S| \ge \eps N$.  We may assume that $\eps \le 2^{-20}$.  We apply our regularity result \Cref{cor:nosparsecutreg} (with $\sigma=|S|/|G| \ge \eps$) to find a subgroup $H$ of $G$ such that $|S \cap H|\ge (1-\eps)|S|$ and 
$\Cayley_H(S \cap H)$ has no $\eps^4/1000$-sparse cuts.
Let $S_0:=S \cap H$ and $J:=S \setminus H$. 
We now define two subsets $S_1,S_2$ of $S_0$; how we do so depends on the proportion of $H$ occupied by $S_0$.

\textbf{Case 1.} $|S_0| \le \frac34 |H|$.\\
In this case we set $S_1:=S_0$ and $S_2:=\emptyset$.

\textbf{Case 2.} $|S_0| \ge \frac34 |H|$.\\
Note that if $S \setminus H =\emptyset$, then we are done by \Cref{thm:very-dense-general}, so we may assume that $J=S\setminus H \neq \emptyset$. 
Take random subsets $S_1, S_2$ of $S_0$ by assigning each element of $S_0$ independently to $S_1$ with probability $\frac14$, to $S_2$ with probability $\frac14$, and to both $S_1,S_2$ with probability $\frac12$.  So $S_1,S_2$ are $\frac34$-random subsets of $S_0$, and $S_1 \cup S_2=S_0$.  Although $S_1,S_2$ are not independent, it is true that after we reveal $S_1$ (respectively, $S_2$), the intersection $S_1 \cap S_2$ is a $\frac{2}{3}$-random subset of $S_1$ (respectively, $S_2$).

Let $S'$ and $E$ be disjoint $\frac14$-random subsets of $S_0$. Set $p:=1/32$, and let $A \sqcup R \sqcup M \sqcup T$ be a random partition of $G$ where each vertex is (independently) assigned to $A,R,M,T$ with probabilities $p,p,1-3p,p$, respectively.

We now reveal $S_1$\footnote{Note that $S_1 \cap S_2$ remains a $\frac{2}{3}$-random subset of $S_1$.}. We have
\begin{equation}\label{ineq:S1}
    \frac58 |S_0| \le |S_1| \le \frac56|H|
\end{equation}
deterministically in Case 1, and with high probability by Chernoff in Case 2; suppose that this inequality holds.
Fix a coset $sH$ of $H$  and apply \Cref{lem:asymptoticinrandom-dense} with $G=sH, S=S_1, J=\emptyset, M=M, S'=S' \cup E \cup (S_1 \cap S_2)$ and the parameters
\begin{equation*}
    \zeta=\eps^{4}/1000, \quad q=1-3p,\quad \mu =N^{-2\gamma}/2^{38},\quad q'=5/6, 
\end{equation*}
and we replace $\eps$ in \Cref{lem:asymptoticinrandom-dense} with $\eps/2$. 
Let us check that the assumptions of the lemma are satisfied.  
By \eqref{ineq:S1} we have $|S_1| \ge \frac58|S_0| \ge \frac58(1-\eps)|S| \ge \frac{\eps}2 N\ge \frac{\eps}2 |H|$.  
We also need
$\Cayley_H(S_1)$ to have no $\zeta$-sparse cuts: This holds in Case 1, since $S_1=S_0$ has no $\zeta$-sparse cuts by construction; and it holds in Case 2 by \eqref{ineq:S1} which implies $\Cayley_H(S_1)$ has no $\frac18$-sparse cuts. Further, we have $q\ge (1+\mu)|S_1|/|H|$, since $|S_1| \le \frac56|H|$, by \eqref{ineq:S1}. Finally, we have $q' \le 1-\mu q/4$ with plenty of room to spare.
Hence, \Cref{lem:asymptoticinrandom-dense} tells us that with probability at least $1-7/|H|$ we have:
\begin{enumerate}[label = {{{\textbf{U\arabic{enumi}}}}}] \setcounter{enumi}{0}
    \item\label{F1} For any $S_F \subseteq S' \cup E \cup (S_1 \cap S_2)$ and any vertices $w \in sH$, there is a rainbow path in $\Cayley_{sH}(S_1)$ starting at $w$, with all other vertices in $M$, such that the path uses all but at most $\mu q |H|$ of the colours of $S_1 \setminus S_F$.\footnote{\Cref{lem:asymptoticinrandom-dense} also gives us the freedom to specify the other endpoint of the path, but we will not need to do so in this proof.}
\end{enumerate}
Moreover, since there are only $\frac{|G|}{|H|} \le \frac{|G|}{(1-\eps)|S|} \le 2N^{\gamma}=o(|H|)$ left-cosets of $H$, with high probability the conclusion \ref{F1} holds simultaneously for all cosets $sH$.

If we are in Case 2, we will need a second application of 
\Cref{lem:asymptoticinrandom-dense}, this time with $S=S_2$ and with all of the other sets and parameters the same as in our first application (this time revealing $S_2$ but leaving $S_1$ as unrevealed so that $S_1 \cap S_2$ is a genuinely $\frac23$-random subset of $S_2$). By the same reasoning, we conclude that with high probability for all cosets $sH$ we have:

\begin{enumerate}[label = {{{\textbf{U\arabic{enumi}}}}}] \setcounter{enumi}{1}
    \item\label{F2}  For any $S_F \subseteq S' \cup E \cup (S_1 \cap S_2)$ and any vertex $z \in sH$, there is a rainbow path in $\Cayley_{sH}(S_2)$ starting at $z$, with all other vertices in $M$, such that the path uses all but at most $\mu q |H|$ of the colours of $S_2 \setminus S_F$.
\end{enumerate}

We know that $|E| \ge |S_0|/8 \ge |S|/16 \ge p^{-2}\max\{40|J|,96 \log |G|\}$ with high probability, so
\Cref{lem:absorbing-junk} (applied with $G=G$, $E=E \cup J$ and $J=J$) tells us that with high probability we have:

\begin{enumerate}[label = {{{\textbf{U\arabic{enumi}}}}}] \setcounter{enumi}{2}
    \item\label{F3} For any vertex $v \in G$, there is a rainbow path in $\Cayley_G(E \cup J)$, starting from $u$ and otherwise contained in $A$, that uses all of the colours of $J$. 
\end{enumerate}

We now reveal $S'$.  By a Chernoff bound, with high probability we have $|S'| \ge |S_0|/8\ge |S|/16$; we henceforth assume that we are in such an outcome. The next two lemmas use only the randomness in the subsets $R,T$ (respectively); the key point is that these are independent of $S'$. 
Set $$t:=N^{1-2\gamma}/2^{14},\quad \ell:=tp^3/81.$$ 
Using \Cref{lem:common-sum}, we can find an element $g \in H$ and a family $\mathcal{P}$ of $g$-pairs in $S'$ of size $3t+\ell\le N^{1-2\gamma}/2^{12}$. \Cref{cor:absorbing-pairs} provides a $(\mathcal{P}_{\mathrm{flex}},\ell)$-absorbing family $\mathcal{P}_1,\mathcal{P}_2,\ldots,\mathcal{P}_{3t}$ in $\mathcal{P}$ with some $\mathcal{P}_{\mathrm{flex}} \subseteq \mathcal{P}$ of size $t+\ell$.

We next apply \Cref{lem:absorbing-waveform} with $G=G,g=g, E=S'$, the $p$-random set $R$, and the families of $g$-pairs $\mathcal{P}_1,\ldots, \mathcal{P}_{3t}$.  The hypotheses of the lemma are satisfied since $|S'|p^{42}/2^{18} \ge \max\{3t,\log |G|\}$.  Thus with high probability we have:

\begin{enumerate}[label = {{{\textbf{U\arabic{enumi}}}}}] \setcounter{enumi}{3}
    \item\label{F4} For any vertex $u \in G$, there is a waveform in $\Cayley_G(S')$ corresponding to $\mathcal{P}_1,\ldots,\mathcal{P}_{3t}$ that starts at $u$ and is otherwise contained in $R$. 
\end{enumerate}

Finally, we apply \Cref{lem:pair-tails} with $G=G,g=g$, a $p$-random subset $T$ and the family $\mathcal{P}_{\mathrm{flex}}$.  The hypotheses of the lemma are satisfied since $|\mathcal{P}_{\mathrm{flex}}|=t+\ell \ge 2^9p^{-3}\log N$.  Thus with high probability we have:

\begin{enumerate}[label = {{{\textbf{U\arabic{enumi}}}}}] \setcounter{enumi}{4}
    \item\label{F5} For any $L$ of size $|L| \le \ell < tp^3/80$ and any vertex $y \in G$, there is a subfamily $\mathcal{P}'\subseteq \mathcal{P}_{\mathrm{flex}}$ with exactly $\ell$ pairs such that $\Cayley_{G}(L \cup \bigcup \mathcal{P}')$ contains a rainbow path that starts at $y$, is otherwise contained in $T$, and uses all except possibly one of the colours from $L \cup  \bigcup \mathcal{P}'$.
\end{enumerate}

Consider an outcome where all of \ref{F1}--\ref{F5} occur.  Fix any vertex $z \in M \cap H$.

If we are in Case $1$, then let $P_{M,1}$ be the empty path. 
If we are in Case $2$, then use \ref{F2} to find a rainbow path $P_{M,1}$, starting at $z$, that uses all of the colours from $S_2 \setminus (S' \cup E)$ except for some some subset $L_2$ of size at most $\mu q |H|$.  Note that since $S_2 \subseteq S_0 \subseteq H$ and $z \in H$, the path $P_{M,1}$ is completely contained in $H$. Let $v \in H$ be its other endpoint, and let $S_{F}''$ be the set of colours appearing in $P_{M,1}$.

Next, let $P_A$ be a minimal-length rainbow path that starts at $v$ (if we are in Case 1, we instead let it start from an arbitrary vertex $v \in A \cap H$), is otherwise contained in $A$, and uses all of the colours from $J$ and some subset of the colours from $E$. Such a path $P_A$ exists by \ref{F3}. Let $u$ be the endpoint of $P_A$. If $J \neq \emptyset,$ then the minimality of $P_A$ guarantees that the last edge of $P_A$ uses some colour $j^* \in J=S \setminus H$.  If in addition $u \in H$, then we delete this last edge from $P_A$ so that its endpoint $u$ now lies in a proper coset $sH$. 
Let $S_F'$ denote the set of colours from $E$ appearing in $P_A$.  Notice that $P_A$ is the empty path (so $u=v$) if $J=\emptyset$.

We now use \ref{F4} to find a waveform $W_R$ corresponding to $\mathcal{P}_1,\ldots, \mathcal{P}_{3t}$ that starts at $u$, is otherwise contained in $R$, and uses the colours of some subset $S_F \subseteq S'$. Let $w$ denote the other endpoint of $W_R$. Note that $w \in sH$ since $S' \subseteq H$. 

Using \ref{F1}, we find a rainbow path $P_{M,2}$, starting at $w$ and otherwise contained in $M$, which uses all of the colours from $S_1 \setminus (S_F \cup S_F' \cup S_F'')$ except for some subset $L_1$ of size at most $\mu q |H|$. If we are in Case 2, since $w \in sH$ and $S_1 \subseteq H$, the path $P_{M,2}$ is completely contained in $sH$; in particular it is vertex-disjoint from $P_{M,1}$ (which is contained in $H$).  (Notice that there is no $P_{M,1}$ to avoid when $J=\emptyset$ in Case 1.)  Let $y$ denote the final vertex of $P_{M,2}$.

Let $L:=L_1 \cup L_2$ be the set of  colours that we have yet to integrate in our rainbow path; we also include the element $j^*$ if we deleted it from $P_A$ previously. Notice that $|L|\le 2\mu q|H|+1 \le \ell$.

For any family $\mathcal{P}'$ of precisely $\ell$ pairs from $\mathcal{P}_{\mathrm{flex}}$, we can collapse the waveform $W_R$ into a path $P_R$ by following the system of distinct representatives for $\mathcal{P}_1\setminus \mathcal{P}',\ldots, \mathcal{P}_{3t}\setminus \mathcal{P}'$ guaranteed by the absorbing property (so in total we use the pairs in $\mathcal{P} \setminus \mathcal{P}'$).  Then, 
$P_{M,1} \cup P_A \cup P_R \cup P_{M,2}$ is a rainbow path using precisely the colours in $S \setminus \left(L \cup \bigcup_{(a_i,b_i)\in \mathcal{P}'} \{a_i,b_i\} \right)$. This path avoids the vertex set $T$ because $P_{M,1} \subseteq H \cap M$, $P_A \setminus \{v\} \subseteq A$, $P_R \setminus \{u\} \subseteq R$, and $P_{M,2}\setminus \{w\} \subseteq sH \cap M$. 

Finally, by \ref{F5} we can find a subfamily $\mathcal{P}' \subseteq \mathcal{P}_{\mathrm{flex}}$ of size precisely $\ell$ and a rainbow path $P_T$, starting at $y$ and otherwise contained in $T$, which uses all except possibly one of the colours from $L \cup \bigcup_{(a_i,b_i)\in \mathcal{P}'} \{a_i,b_i\}$. Take the corresponding path $P_R$ from the previous paragraph.  Then $P_{M,1} \cup P_{A} \cup P_{R} \cup P_{M,2} \cup P_T$ is a rainbow path using all but one colour from $S$, as desired.
\end{proof}

\section{Concluding remarks}\label{sec:concluding}
As we have seen in Section~\ref{sec:general-dense}, our methods in the case of dense subsets $S\subset G$ apply to Problem~\ref{question:main} just as well over arbitrary groups as in the specialised setting of $\Fn$. The basic randomness vs.\ structure dichotomy (see \Cref{sec:overview}) also translates well to general groups. However, a key complication for general groups is that the structure of subsets with bounded doubling is more complicated; already for $\mathbb{F}_p$ one has to work with generalised arithmetic progressions in place of proper subgroups. In particular, over $\mathbb{F}_p$, we have no means of passing to a robust expander of size $O(|S|)$ and finishing most of the job there. There are also further complications over $\mathbb{F}_p$ for the absorption part of the argument, not least because we no longer have access to popular sums as we did over $\Fn$, or as we did in the dense case. Novel ideas are required to settle both of these issues in order to use our framework to settle Graham's conjecture for large $p$.

\providecommand{\MR}[1]{}
\providecommand{\MRhref}[2]{%
  \href{http://www.ams.org/mathscinet-getitem?mr=#1}{#2}
}

\bibliographystyle{amsplain_initials_nobysame}
\bibliography{bib}

\newpage

\appendix

\section{Extremely dense case}\label{sec:appendix}
For brevity, we write $K_G^-:=\Cayley_G(G \setminus \{\id\})$ in this appendix.

For subsets $R,C\subseteq G$, we write $K^-_G[R;C]$ to denote the subgraph of $K^-_G$ induced on the vertex set $R$ by the edges with colours in $C$. For disjoint subsets $V_1,V_2\subseteq G$, we write $K^-_G[V_1,V_2;C]$ to denote the bipartite subgraph of $K^-_G$ obtained by keeping only the directed edges from $V_1$ to $V_2$ with colours in $C$.

The following lemma is part of Lemma 6.22 from \cite{muyesser2022random}. The proof combines the sorting network method and the statement of the random Hall--Paige conjecture. The original statement pertains to both addition and multiplication tables, but to reduce clutter we have included only the part that we will need.  In the remainder of this appendix, we will perform many calculations in the abelianisation $G/[G,G]$ of $G$; since the order of multiplication does not matter in the abelianisation, product notation such as $\prod_{v \in V}v$ is unambiguous. 

\begin{lemma}\label{lem:pathlikemain}
    Let $1/n\ll \gamma, p\leq 1$, and let $(\log n)^7 \leq t \leq (\log n)^8$ be an integer.  Set $q:=p/(t-1)$. Let $G$ be a group of order $n$. 
    Let $V_{\mathrm{str}}, V_{\mathrm{mid}}, V_{\mathrm{end}}$ be disjoint random subsets of $G$ with $V_{\mathrm{str}}, V_{\mathrm{end}}$ $q$-random and $V_{\mathrm{mid}}$ $p$-random. Let $C$ be a $(q+p)$-random subset of $G$, sampled independently of $V_{\mathrm{str}}, V_{\mathrm{mid}}, V_{\mathrm{end}}$. Then with high probability the following holds for all choices of $C' \subseteq G$ and disjoint subsets $V_{\mathrm{str}}',V_{\mathrm{end}}',V_{\mathrm{mid}}' \subseteq G$:
    \par If $C', V_{\mathrm{str}}',V_{\mathrm{end}}',V_{\mathrm{mid}}'$ satisfy
    \begin{enumerate}
        \item for each $R\in \{V_{\mathrm{str}}, V_{\mathrm{mid}}, V_{\mathrm{end}}, C\}$, we have $|R\Delta R'|\leq n^{1-\gamma}$;
        \item $\prod V_{\mathrm{end}}' \cdot  (\prod V_{\mathrm{str}}')^{-1}=\prod C' \pmod{[G,G]}$;
        \item $\id\notin C'$;
        \item $|V_{\mathrm{str}}'|=|V_{\mathrm{end}}'|=|V'_{\mathrm{mid}}|/(t-1)=|C'|/t$,
    \end{enumerate}
    then for every bijection $f\colon V_{\mathrm{str}}'\to V_{\mathrm{end}}'$, the graph $K^-_G[V_{\mathrm{str}}'\cup V_{\mathrm{end}}'\cup V_{\mathrm{mid}}';C']$ has a rainbow collection of vertex-disjoint paths $\{P_v: v \in V'_{\mathrm{str}}\}$, where each $P_v$ has length $t$ and starts at $v$ and ends at $f(v)$.   
\end{lemma}

We can now prove the main result of the appendix. 
Theorem 6.9 of \cite{muyesser2022random} gives a sharper version of this result in the regime $\gamma\geq 1/2$.  In fact, the same proof works verbatim for the larger range $1/n\ll \gamma<1$, but this flexibility is unfortunately not recorded in \cite{muyesser2022random}, as the authors did not anticipate that it would have further applications. We follow the proof from \cite{muyesser2022random} quite closely in our discussion here. The main idea is applying \Cref{lem:pathlikemain} twice, in a such a way that the starting-vertices of one collection of paths correspond to the ending-vertices of the other collection of paths, and vice versa, so that together all of the paths form a single long path.

\begin{theorem}\label{thm:ultra-dense-appendix}
Let $1/N\ll \gamma\leq 1$. If $G$ is a group of order $N$ and $S\subseteq G\setminus\{\id\}$ is a subset with $|S|\geq N - N^{1-\gamma}$, then $S$ has a valid ordering, i.e., the Cayley graph $\Cayley_G(S)$ has a directed rainbow path with $|S|-1$ edges. 
\end{theorem}

\begin{proof} 
\par Fix distinct $x,y\in G$ such that that $yx^{-1}=\prod S \pmod{[G,G]}$.  We will show that there is a directed rainbow path from $x$ to $y$ with $|S|$ edges. Note that if $G$ is abelian and $\sum S=0$, then there are no such distinct $x$ and $y$; in this case we simply delete an element of $S$ so that $\sum S\neq 0$, and applying our argument with this new $S$ still produces the desired rainbow path using all but one of the colours from the original set $S$.
\par Set $t:=2\lfloor  (\log N)^7\rfloor$ and $s:=|S|$. Set $q:=1/(2t)$ and $p:=(t-1)q$.  Take a random partition $V_{\mathrm{str}} \sqcup V_{\mathrm{end}} \sqcup V_{\mathrm{mid},1}\sqcup V_{\mathrm{mid},2}$ of $G$ where the former two parts are $q$-random and the latter two are $p$-random. Independently, take a random partition of $G$ into $1/2$-random sets $C_0$ and $C_1$ (note that $1/2=p+q$ here). 
\par With high probability, Lemma~\ref{lem:pathlikemain} applies with $V_{\mathrm{str}}, V_{\mathrm{end}}, V_{\mathrm{mid},1}$, $C_0$ (playing the role of $C$) and $t$; and Lemma~\ref{lem:pathlikemain} applies with $V_{\mathrm{str}}$ and $V_{\mathrm{end}}$ interchanged, with $V_{\mathrm{mid},2}$ instead of $V_{\mathrm{mid},1}$, and with $C_1$ playing the role of $C$. In each of these applications of Lemma~\ref{lem:pathlikemain}, let $\gamma/10$ play the role of $\gamma$. Furthermore, we can ensure that with high probability for each $g\in G$, each random set with randomness parameter $z$ contains a disjoint collection of $\Omega(z^3N)$ triples $(a,b,c)$ of distinct group elements with $abc=g$; call this property $(*)$. This property follows from Chernoff's bound and the fact that for each $g$ we can find a disjoint collection of $\Omega(N)$ distinct triples $(a,b,c)$ with $abc=g$ (and then we union-bound over $g$).
Notice also that with high probability each random set with parameter $z$ has $zN \pm N^{0.6}$ elements, again by Chernoff's bound.
\par Fix an outcome for the random sets satisfying the properties described in the previous paragraph.  We will construct slightly-modified ``prime'' versions of the sets $V_{\mathrm{str}}, V_{\mathrm{mid},1}, V_{\mathrm{mid},2}, V_{\mathrm{end}}$ so that the hypotheses of \Cref{lem:pathlikemain} are satisfied.  We will remove small ``junk sets'' to guarantee the divisibility condition (4), and then we will interchange a few elements among the sets to guarantee the product condition (2).

Let $\ell$ be the largest integer satisfying $2t\ell-t+1\leq s+1$. Note that $\ell=qs+O(1)=qN\pm N^{1-\gamma/2}$. Define $w:=(s+1)-(2t\ell-t+1)$, and note that $w\leq (\log n)^{10}$. Observe that we can greedily find an $S$-rainbow path $P_0$, from $x$ to some $x'$, which has exactly $w+1$ vertices and does not pass through $y$; fix such a path. 

By modifying at most $n^{1-\gamma/2}$ elements from each of the sets $V_{\mathrm{str}}, V_{\mathrm{mid},1}, V_{\mathrm{mid},2}, V_{\mathrm{end}}$, we can obtain new disjoint sets $V_{\mathrm{str}}',V_{\mathrm{mid},1}',V_{\mathrm{mid},2}',V_{\mathrm{end}}' \subseteq (V\setminus V(P_0))\cup \{x'\}$ such that $x'\in V_{\mathrm{str}}'$, $y\in V_{\mathrm{end}}'$, and $\ell=|V_{\mathrm{str}}'|=|V_{\mathrm{end}}'|=|V_{\mathrm{mid},1}'|/(t-1)=|V_{\mathrm{mid},2}'|/(t-1)$. Similarly, we partition $S\setminus C(P_0)$ into sets $C_0'$ of size $t\ell$ and $C_1'$ of size $t(\ell-1)$ (this is possible due to the divisibility constraint on the size of $P_0$) satisfying $|C_0\Delta C_0'|,|C_1\Delta C_1'|\leq n^{1-2\gamma/3}$. Furthermore, we can interchange a few elements, thanks to property $(*)$, to ensure that $$\prod V_{\mathrm{str}}'\setminus\{x'\}=\prod V_{\mathrm{end}}'\setminus\{y\}=\prod C_1'=\id \pmod{[G,G]}.$$ This implies that $\prod C_0'=\prod S  \left(\prod C(P_0)\right)^{-1}\pmod{[G,G]}$. Our interchanges maintain the property that $|Z\Delta Z'|\leq n^{1-\gamma/2}$ for each set $Z$.
\par We now invoke Lemma~\ref{lem:pathlikemain} for the sets $V_{\mathrm{str}}'\setminus\{x'\}$, $V_{\mathrm{end}}'\setminus\{y\}$, $V'_{\mathrm{mid},2}$, $C_1'$ with an arbitrary choice of bijection to get a partition into paths of length $t$ with starting points in $V_{\mathrm{end}}'\setminus\{y\}$ and endpoints in $V_{\mathrm{str}}'\setminus\{x'\}$. We wish to now invoke Lemma~\ref{lem:pathlikemain} in the opposite orientation, with the sets $V_{\mathrm{str}}'$, $V_{\mathrm{end}}'$, $V'_{\mathrm{mid},2}$, $C_1'$, and a choice of bijection that we will shortly specify (to ensure that everything links up to form a path). First, we check the relevant product condition.
\begin{claim}
    We have $\prod V_{\mathrm{end}}' (\prod V_{\mathrm{str}}')^{-1}=\prod C_0'\pmod{[G,G]}$.
\end{claim}
\begin{proof} We carry out the following calculations modulo $[G,G]$. Note first that $\prod V_{\mathrm{end}}' (\prod V_{\mathrm{str}}')^{-1}=y(x')^{-1}$ as from the previous exchanges we had ensured that $\prod V_{\mathrm{str}}'\setminus\{x'\}=\prod V_{\mathrm{end}}'\setminus\{y\}$. Recall that $yx^{-1}=\prod S$, and that $\prod C(P_0)=x'x^{-1}$, as $P_0$ is a path from $x$ to $x'$. Thus $y(x')^{-1}=\prod S (\prod C(P_0))^{-1}$. We also previously saw that $\prod C_0'=\prod S  \left(\prod C(P_0)\right)^{-1}$; this completes the proof of the claim.
\end{proof}
Now, we specify the bijection that will ensure that the concatanation of all paths we have constructed so far yields an $S$-rainbow path from $x$ to $y$. Suppose that the previously-constructed collection of paths had endpoints $y_1\to x_1$, $y_2\to x_2$, $\ldots$, $y_{\ell-1}\to x_{\ell-1}$. Then we choose the bijection that maps $x_1\to y_2$, $x_2\to y_3$, $\ldots$, $x_{\ell-2}\to y_{\ell-1} $, $x_{\ell-1}\to y$, $x'\to y_1$. The union of the resulting paths from the two applications of Lemma~\ref{lem:pathlikemain}, together with $P_0$, yields a rainbow path from $x$ to $y$ whose edges use precisely the colours of $S$.
\end{proof}

\end{document}